\documentclass[11pt]{amsart}
\usepackage{amssymb}
\usepackage{amsfonts}
\usepackage{amsmath}
\usepackage{graphicx}
\usepackage{xcolor}
\usepackage{amsmath}
\usepackage{mathrsfs}
\usepackage{stmaryrd}
\usepackage{epsfig,color}
\usepackage{blindtext}
\usepackage{enumerate}
\usepackage{hyperref}
\usepackage{url}
\usepackage{bbm}
\usepackage{nicefrac,mathtools}
\usepackage{bm}   
\usepackage{tabularx}
\DeclareGraphicsExtensions{.pdf,.jpeg,.png}
\usepackage{epstopdf}
\usepackage{cancel} 
\usepackage[normalem]{ulem} 
\usepackage{verbatim} 
\usepackage{scalerel}
\usepackage{enumitem}

\usepackage{cases}

\usepackage{subcaption}
\usepackage{soul}

\usepackage{color}
\usepackage[msc-links, lite]{amsrefs}
\usepackage{geometry}
\usepackage{stackengine,scalerel}

\newtheorem{theorem}{Theorem}[section]

\newtheorem{proposition}[theorem]{Proposition}
\newtheorem{lemma}[theorem]{Lemma}
\newtheorem{corollary}[theorem]{Corollary}
\newtheorem{question}[theorem]{Question}

\newtheorem{claim}[]{Claim}
\newtheorem*{acknowledgements}{Acknowledgements}
\theoremstyle{definition}
\newtheorem{definition}[theorem]{Definition}
\theoremstyle{remark}
\newtheorem{remark}[theorem]{Remark}

\numberwithin{equation}{section}

\newcommand{\mf}{\mathbf}
\newcommand{\mb}{\mathbb}
\newcommand{\mc}{\mathcal}

\newcommand{\mk}{\mathfrak}

\newcommand{\wti}{\widetilde}

\newcommand{\Si}{\Sigma}

\newcommand{\codim}{\operatorname{codim}}

\newcommand{\bd}{\partial}

\newcommand{\rom}[1]{\expandafter\romannumeral #1}
\newcommand{\Rom}[1]{\uppercase\expandafter{\romannumeral #1}}

\DeclareMathOperator{\genus}{genus}
\DeclareMathOperator{\Ric}{Ric}

\DeclareMathOperator{\spt}{spt}
\DeclareMathOperator{\interior}{int}
\DeclareMathOperator{\closure}{Clos}

\title{Embedded minimal $S^1$-bundles in $\mathbb{S}^4$}

\author{Tongrui Wang}
\address{School of Mathematical Sciences, Shanghai Jiao Tong University, 800 Dongchuan RD, Minhang District, Shanghai, 200240, China}
\email{wangtongrui@sjtu.edu.cn}

\begin{document}

\begin{abstract}
	We construct infinitely many embedded minimal hypersurfaces of pairwise distinct irreducible topological types in the unit $4$-sphere $\mathbb{S}^4$, which provide a new answer to a problem of Hsiang. 
    These examples are topologically principal $S^1$-bundles and Seifert fibered manifolds over closed orientable surfaces.  
	In particular, for any closed orientable surface $\Sigma_{2k-1}$ of odd genus $2k-1$, we show that $S^1\times \Sigma_{2k-1}$ admits a minimal embedding into $\mb S^4$. 
	The construction is based on the equivariant min-max theory and the suspended (weighted) Hopf action on $\mb S^4$. 
\end{abstract}

\maketitle

\section{Introduction}

Closed minimal hypersurfaces in the unit sphere $\mb S^n$ form one of the central subjects in geometry. 
Many fundamental questions and developments in the field concern their existence, rigidity, and analytic properties. 
In the present paper, we focus on a natural but still largely open topological realization problem: which closed $3$-manifolds can be realized as embedded minimal hypersurfaces in the round unit $4$-sphere $\mb S^4$?

Even in the $3$-dimensional unit sphere $\mb S^3$, constructing embedded minimal surfaces with different topological types is a highly nontrivial problem. 
For a long time, the only known closed embedded minimal surfaces in $\mb S^3$ were essentially limited to the totally geodesic $2$-spheres and the Clifford torus. 
A decisive breakthrough came in the late 1960s when Lawson \cite{lawson1970S3} discovered an infinite family of embedded orientable minimal surfaces in $\mb S^3$ with arbitrary genus. 
Since then, minimal surfaces in $\mb S^3$ have been studied by a wide range of methods, including solving
Plateau problems \cite{karcher1988new}\cite{choe2016new}, gluing and desingularization \cite{kapouleas2010doublingT2, kapouleas2017doublingS2-1, kapouleas2019doublingS2, wiygul2020stacking, kapouleas2022desingularizing}, eigenvalue optimization \cite{karpukhin2024minimal}, and min-max theory \cite{ketover2020catenoid}\cite{ketover2016equivariant}. 
In particular, Lawson’s celebrated result \cite{lawson1970S3} shows that there is no topological obstruction to the minimal realization of orientable closed surfaces  in $\mb S^3$.

However, the situation in higher-dimensional round spheres $\mb S^{n\geq 4}$ is quite different. 
Despite the abundance of pairwise non-congruent embedded minimal hypersurfaces, the list of known topological types remains limited. 
In 1993, Hsiang \cite{hsiang1993closed} proposed the following question. 
\begin{question}[Hsiang, 1993]
    How to construct embedded closed minimal hypersurfaces in $\mb S^4$ which are of topological types different from $S^3$, $S^1\times S^2$, $T^3$, and $SO(3)/\mb Z_2^2$? 
\end{question}

Although several powerful construction methods have been developed (as will be reviewed below), Hsiang's question reflects the scarcity of known topological types for minimal hypersurfaces in $\mathbb S^4$. 
From the viewpoint of the prime decomposition of $3$-manifolds, the four topological types listed in Hsiang's question, i.e. 
\[
 S^3,\qquad  S^1\times  S^2,\qquad  T^3,\qquad SO(3)/\mathbb Z_2^2,
\]
exhaust, to the best of our knowledge, all the prime $3$-manifold topologies that have so far been realized by known closed embedded minimal hypersurfaces in $\mb S^4$. 

Recently, some important progress has also been made in realizing certain reducible $3$-manifolds as minimal hypersurfaces in $\mb S^4$. 
By doubling the equator $\mb S^3\subset \mb S^4$ and gluing necks, Kapouleas and Zou \cite{kapouleas2024} constructed minimal hypersurfaces in $\mb S^4$ that are diffeomorphic to $\#_{m^2-1}S^1\times S^2$ for all $m\in \mb N$ sufficiently large. 
In \cite{pitts1987applications}, Pitts and Rubinstein outlined an equivariant min-max construction for minimal hypersurfaces of topology $\#_{2i}(\mathbb S^1\times \mathbb S^2)$ in $\mathbb S^4$ for some large $i\in \mb N$, which was recently carried out by Ko \cite{ko2025} under suitable equivariant hypotheses. 
These results provide an infinite family of topologically distinct examples, while their topology remains in the connected-sum class generated by the prime factor $ S^1\times  S^2$.

\subsection{Main results}
In this paper, we first show two infinite families of closed embedded minimal hypersurfaces in $\mb S^4$ with pairwise distinct irreducible (and thus prime) topological types.



\begin{theorem}\label{Main thm: new minimal hypersurfaces in S4}
    For any integer $n\in \{2,3,\dots\}$, there is an embedded minimal hypersurface $\Gamma_{n}$ in $\mb S^4$ that is a principal $S^1$-bundle over a closed orientable surface $\Gamma_n/S^1$. 
    Moreover, 
    \begin{itemize}
        \item[(i)] if $n=2k$ is even, then the principal $S^1$-bundle $\Gamma_n\to\Gamma_n/S^1$ is trivial, and $\genus(\Gamma_n/S^1)=n-1=2k-1$; 
        \item[(ii)] if $n=2k+1$ is odd and sufficiently large, then the principal $S^1$-bundle $\Gamma_n\to\Gamma_n/S^1$ is non-trivial, its Euler class $e(\Gamma_n)\in H^2(\Gamma_n/S^1;\mb Z)$ is a generator, and $\genus(\Gamma_n/S^1)=2n-2=4k$. 
    \end{itemize}
\end{theorem}

\begin{remark}
    For any odd $n=2k+1\geq 3$, the minimal hypersurface $\Gamma_n\subset \mb S^4$ in Theorem \ref{Main thm: new minimal hypersurfaces in S4} must have $\genus(\Gamma_n/S^1)\in \{2n-2, 0\}$. 
    Although we currently cannot rule out the case that $\genus(\Gamma_n/S^1)=0$ for small $n=2k+1$, one can obtain a non-equatorial minimal hypersphere $\Gamma_n$ that is non-congruent to Hsiang's examples in \cite{hsiang1983sphericalI} provided that $\genus(\Gamma_n/S^1)=0$ (Theorem \ref{Thm: new minimal hypersurfaces of odd n}). 
    In particular, every $\Gamma_n$, $n\in\{2,3,\dots\}$, is topologically either a $3$-sphere or an $S^1$-bundle over an aspherical $2$-manifold $\Gamma_n/S^1$, which indicates that $\pi_2(\Gamma_n)=0$, and $\Gamma_n$ is irreducible. 
\end{remark}

The following result is a direct corollary. Indeed, we only need $\Ric>0$ and certain symmetries to construct the following minimal hypersurfaces on any Riemannian $S^4$ (Remark \ref{Rem: valid in symmetric S4 with Ric>0 for n-even}). 
\begin{corollary}
    For any orientable closed surface $\Sigma_{2k-1}$ of odd genus $2k-1$, there is a minimal embedding from $S^1\times \Sigma_{2k-1}$ into $\mb S^4$. 
\end{corollary}

The proof of Theorem \ref{Main thm: new minimal hypersurfaces in S4} is inspired by the equivariant min-max approach developed in \cite{wang2026spherical} for Chern's spherical Bernstein problem in $\mb S^4$. In that work, Wang-Wang-Zhou established an equivariant min-max theory for minimal $G$-hypersurfaces in closed Riemannian $G$-manifolds $M$, where $G$ is a compact Lie group acting isometrically on $M$ with $\dim(M/G)=3$. This gives, in a broad setting, a rigorous form of the equivariant min-max program announced by Pitts-Rubinstein \cite{pitts1987applications} in 1987; see also \cite{ketover2016equivariant}\cite{ko2025} for related results under different hypotheses. 

One of the key constructions in \cite{wang2026spherical} is the suspended Hopf $S^1$-action on $\mb S^4$ (cf. Definition \ref{Def: suspended Hopf action}), i.e. the $S^1$-action on $\mb S^4=\{(z_1,z_2,x)\in \mb C\times \mb C\times \mb R:|z_1|^2+|z_2|^2+x^2=1\}$ defined by
\begin{align}\label{Eq: introduction - S1 action on S4}
    e^{i\theta}\cdot (z_1,z_2,x):=(e^{i\theta}z_1, e^{i\theta}z_2, x), \qquad\forall e^{i\theta}\in S^1\subset\mb C. 
\end{align}
Its orbit space $\mb S^4/S^1$ is topologically a $3$-sphere with two conical singular points at the north and south poles, which looks like an American football. By combining this $S^1$-action with suitable isometric $\mb Z_2$-actions on $\mb S^4/S^1$ (e.g. the antipodal symmetry), Wang-Wang-Zhou \cite{wang2026spherical}*{\S 7} provided the min-max constructions of minimal $S^3$ (non-equatorial), $T^3$, and $S^1\times S^2$ in $\mb S^4$. 

Although the orbit space $\mb S^4/S^1$ is singular, the regularity theory in \cite{wang2026spherical} implies that the $S^1$-equivariant min-max hypersurface $\Gamma\subset \mb S^4$ is smoothly embedded so that its quotient $\Gamma/S^1$ avoids the two singular points of $\mb S^4/S^1$. Hence, an embedded minimal $S^1$-hypersurface can be studied through a closed surface in the singular $3$-sphere $\mb S^4/S^1$. In view of Lawson's minimal surfaces in $\mb S^3$ \cite{lawson1970S3}, it is natural to seek minimal $S^1$-hypersurfaces $\Gamma\subset \mb S^4$ for which the quotient surface $\Gamma/S^1$ has arbitrarily large genus. 
However, to identify the diffeomorphism type of the equivariant min-max $S^1$-hypersurface $\Gamma$, one has to determine two pieces of data: 
\begin{itemize}
    \item[(1)] the genus of the quotient surface $\Gamma/S^1$, which determines the base surface of the bundle, 
    \item[(2)] the Euler class of the bundle $\Gamma\to\Gamma/S^1$, which is governed by whether $\Gamma/S^1$ separates the two singular points (i.e. the north and south poles) of $\mb S^4/S^1$.
\end{itemize}


We first address the genus of $\Gamma/S^1$. For this purpose, we introduce a modified reflectional dihedral $D_n$-action on $ \mb S^4/S^1$ (cf. Definition \ref{Def: reflectional dihedral action on S4/S1}) and construct $D_n$-equivariant sweepouts in the orbit space (Section \ref{Sec: new minimal hypersurfaces in S4}). The sweepouts are obtained by taking two parallel spheres $\{x=\pm\rho\}$ (resp. three parallel spheres $\{x=\pm \rho, 0\}$) in $\mb S^4/S^1$ and connecting them by $n$ (resp. $2n$) necks arranged according to the $D_n$-symmetry, which is analogous to the constructions in \cite{wang2026spherical}*{\S 7.2}. Applying equivariant min-max theory to these sweepouts gives embedded minimal $S^1$-hypersurfaces $\Gamma_n$ in $\mb S^4$ with $D_n$-symmetries on $\Gamma_n/S^1$. The genus information is then obtained by combining the reduced genus upper bounds in the equivariant min-max (cf. \cite{wang2026spherical}*{Theorem 1.5}, \cite{ketover2019genus}), the Riemann-Hurwitz formula (cf. \cite{ketover2016equivariant}*{(6.1)}, \cite{carlotto2022free}*{Appendix B}), and a compactness analysis of the resulting sequence as the symmetry order of $D_n$ tends to infinity.

Meanwhile, to determine the separating behavior of $\Gamma/S^1$, both the antipodal symmetry and the order-two elements in our modified reflectional $D_n$-action can be used to show that $\Gamma/S^1$ bounds a suitable equivariant open set in the orbit space. The symmetry of this open set then constrains the positions of the two singular points and determines the separating information needed to compute the Euler class of $\Gamma\to\Gamma/S^1$. 
In particular, using the antipodal symmetry, the separating behavior of $\Gamma/S^1$ can also be formulated after passing to the quotient $(\mb S^4/S^1)/\mb Z_2= RP^3$, where the orientability of $(\Gamma/S^1)/\mb Z_2$ provides the required topological information.

Moreover, our choice of the modified reflectional $D_n$-action is helpful for a further reason. It is not the direct suspension of the usual rotational dihedral action (used in  \cite{ketover2016free, carlotto2022free} for free boundary minimal surfaces in $\mb B^3$) along the $x$-coordinates. Rather, it is designed so that it lifts naturally to $\mb S^4$ and admits an analog under the {\em suspended weighted Hopf $S^1_{p,q}$-action} on $\mb S^4$:
\begin{align}\label{Eq: introduction - weighted S1 action on S4}
    e^{i\theta}\cdot (z_1,z_2,x):= (e^{ip\theta}z_1, e^{iq\theta}z_2,x), \qquad \forall e^{i\theta}\in S^1\subset \mb C, ~{\rm and~coprime}~p,q\in\mb Z\setminus\{0\}.  
\end{align}
This allows us to apply the equivariant min-max construction similar to that of Theorem \ref{Main thm: new minimal hypersurfaces in S4} in the weighted setting. 
Additionally, we generalize the $D_n$-sweepout construction in \cite{buzano2025self-shrinker} by desingularizing two intersecting spheres in a picture analogous to Lawson's $\xi_{1,n-1}$ surface. 
Although the antipodal symmetry used above is no longer available and the precise diffeomorphism type is harder to determine, the construction still produces infinitely many distinct embedded minimal Seifert fibered hypersurfaces in $\mb S^4$ (cf. Theorem \ref{Thm: minimal Seifert manifold in S4 - same side}, \ref{Thm: minimal Seifert manifold in S4 - different sides}, \ref{Thm: minimal Seifert manifold in S4 - Lawson type}).


\begin{theorem}\label{Main thm: minimal seifert hypersurfaces in S4}
    Given $n\in \{2,3,\dots\}$ and coprime $p,q\in \mb Z\setminus\{0\}$, there are embedded  minimal hypersurfaces $ \widehat\Gamma_n^{p,q},\wti \Gamma_n^{p,q}, \Gamma_n^{p,q} \subset \mb S^4$ that are Seifert fibered $3$-manifolds invariant under the $S^1_{p,q}$-action defined in \eqref{Eq: introduction - weighted S1 action on S4} such that for some $\wti N_{p,q}, N_{p,q}> 1$ large enough,
    \begin{itemize}
        \item[(i)] $(\mk  g(\widehat\Gamma_n^{p,q}),  \mk m_+(\widehat\Gamma_n^{p,q}), \mk m_-(\widehat\Gamma_n^{p,q}) )  \in \{(0,2,0),(0,0,2), (n-1,2,2),(n-1,4,0),(n-1,0,4)\}$; 
        \item[(ii)] $(\mk  g(\wti\Gamma_n^{p,q}),  \mk m_+(\wti\Gamma_n^{p,q}), \mk m_-(\wti\Gamma_n^{p,q}) )\in \{(n-1,3,1), (n-1,1,3), (2n-2,3,3),(2n-2,5,1),(2n-2,1,5)\}$ for all $n\geq \wti N_{p,q}$; 
        \item[(iii)] $\left(\mk  g(\Gamma_n^{p,q}),  \mk m_+(\Gamma_n^{p,q}), \mk m_-(\Gamma_n^{p,q}) \right)\in \{(n-1,3,1), (n-1,1,3)\}$ for all $n\geq N_{p,q}$;
    \end{itemize}
    where $ {\mk g}(\Gamma):=\genus(\Gamma/S^1_{p,q})$, $\mk m_\pm(\Gamma):= \# [(\Gamma/S^1_{p,q})\cap \Xi_0^\pm]$, and $\Xi_0^\pm$ are the distinguished arcs given by $\Xi_0^+:=\{(z_1,0^{(2)},x)\in\mb S^4\}/S^1_{p,q}$ and $\Xi_0^-:=\{(0^{(2)},z_2,x)\in\mb S^4\}/S^1_{p,q}$. 
    In particular, 
    \begin{itemize}
        \item[(1)] each of ${\mc {\wti{M}}}_{p,q} := \{\wti \Gamma_n^{p,q}\subset \mb S^4: n\geq  \wti N_{p,q}\}$ and ${\mc M}_{p,q} := \{\Gamma_n^{p,q}\subset \mb S^4: n\geq  N_{p,q}\}$ contains infinitely many embedded minimal hypersurfaces in $\mb S^4$ of distinct topological types;
        \item[(2)] ${\mc {\wti M}}_{p_1,q_1}\cap {\mc {\wti M}}_{p_2,q_2}=\emptyset$ and ${\mc M}_{p_1,q_1}\cap {\mc M}_{p_2,q_2}=\emptyset$ for any two distinct pairs of coprime integers $\{p_i>q_i\geq 1\}_{i=1}^2$. 
    \end{itemize}
\end{theorem}

If $|p|>1$ (resp. $|q|>1$), then $\mk m_+(\Gamma)$ (resp. $\mk m_-(\Gamma)$) denotes the number of cone points of order $|p|$ (resp. $|q|$) in the orbifold $\Gamma/S^1_{p,q}$. 

\begin{remark}
    In the above theorem, the Seifert fibration structure of the minimal hypersurfaces is determined (up to orientation) by $|p|,|q|$ and $(\mk g, \mk m_{\pm})$ (cf. Lemma \ref{Lem: S1pq invariant Seifert manifolds}). 
    In particular, the minimal hypersurfaces $\widehat \Gamma_n^{p,q}$ have their Euler number $\mk e=0$, whereas $\wti\Gamma_n^{p,q}$ and $\Gamma_n^{p,q}$ have Euler number $\mk e=\pm1/pq$. 
    However, by \cite{scott1983three-manifolds}*{Theorem 3.8}, some special $3$-manifolds, such as $S^3$ and $S^1\times S^2$, admit more than one Seifert fibration structure. 
    Hence, at present, it is not clear whether $\{\widehat\Gamma_n^{p,q}\}_{n\geq 2}$ can have distinct topological types since we cannot rule out the case $\mk g(\widehat\Gamma_n^{p,q})=0$. 
    Nevertheless, for $n$ large enough, each of $\wti \Gamma_n^{p,q}$ and $\Gamma_n^{p,q}$ has a unique Seifert fibration structure (cf. Remark \ref{Rem: Seifert invariants}(iii)) and is irreducible by the comments after \cite{scott1983three-manifolds}*{Lemma 3.1}. 
\end{remark}

Using Theorem \ref{Main thm: new minimal hypersurfaces in S4}(ii) and Theorem \ref{Main thm: minimal seifert hypersurfaces in S4}(iii) with $|pq|=1$, we obtain the following corollary. 

\begin{corollary}
    For any orientable closed surface $\Sigma_{\mk g}$ of sufficiently large genus $\mk g$, the non-trivial principal $S^1$-bundle over $\Sigma_{\mk g}$ whose Euler class is a generator of $H^2(\Sigma_{\mk g};\mb Z)$ admits a minimal embedding into $\mb S^4$. 
\end{corollary}

Note that a principal $S^1$-bundle $\Gamma$ can be naturally associated with a complex line bundle $\Gamma\times_{S^1} \mb C:=(\Gamma\times \mb C) / \sim$, where $( e^{i\theta}u,z)\sim (u,e^{i\theta}z)$ for any $(u,z)\in\Gamma\times \mb C$ and $e^{i\theta}\in S^1$. 
Then, by \cite{husemoller1994fiberbundle}*{p. 301, Remark 6.2}, the (oriented) topological type of a principal $S^1$-bundle over $\Sigma_{\mk g}$ is uniquely determined by $\mk g$ and the Euler class $e\in H^2(\Sigma_{\mk g};\mb Z)$. 
Hence, the minimal hypersurfaces $\Gamma_{2k+1}$ and $\Gamma_{4k+1}^{\pm 1,\pm 1}$ in Theorem \ref{Main thm: new minimal hypersurfaces in S4}(ii) and Theorem \ref{Main thm: minimal seifert hypersurfaces in S4}(iii), respectively, have the same topological type for $k$ large enough. 
Nevertheless, we expect that  $\Gamma_{2k+1}$ is non-congruent to $\Gamma_{4k+1}^{\pm 1,\pm 1}$ 
as suggested by the different prescribed dihedral actions in the two constructions. 

\subsection{\texorpdfstring{Constructions of minimal hypersurfaces in $\mb S^{n\geq 4}$}{Constructions of minimal hypersurfaces in Sn>=4}}
Besides the equivariant min-max constructions and the doubling-gluing constructions discussed above, there are several other methods for constructing minimal hypersurfaces in higher-dimensional unit spheres.

A classical systematic approach is the {\em equivariant differential geometry} initiated by Hsiang and Lawson \cite{hsiang71cohom}. 
By imposing a large symmetry group on the ambient sphere, this method reduces the minimal hypersurface equation to an ODE problem in a $2$-dimensional orbit space. 
Combining this reduction with the shooting method, Hsiang constructed infinitely many non-congruent minimal hypersurfaces in $\mb S^{n\geq 4}$ but with only a few simple topological types, including $S^{n-1}$, $S^1\times S^{n-2}$, and $S^2\times S^{n-3}$ \cite{hsiang1983sphericalI,hsiang1983sphericalII,hsiang1986sphericalIII,hsiang1987infinitly}. 
In the same equivariant-reduction spirit, Carlotto-Schulz \cite{carlotto2023hypertori} constructed minimal $S^1\times S^{n-1}\times S^{n-1}$ in $\mb S^{2n}$, while Firester-Tsiamis \cite{firester2026new} obtained further families of minimal $S^1\times S^p\times S^q$ in $\mb S^{p+q+2}$. 
When specialized to $\mb S^4$, these constructions give the minimal embedded $S^3$, $S^1\times S^2$, and $T^3$.

Another important source of examples comes from minimal isoparametric hypersurfaces in $\mb S^n$, which have been classified through the work of many authors; see, for instance, \cite{cecil2007isoparametric,miyaoka2013isoparametric,chi2020isoparametric}. 
In $\mb S^4$, besides the equatorial $\mb S^3$ and the Clifford hypersurfaces $S^1\times S^2$, the only isoparametric example is the Cartan hypersurface, which is diffeomorphic to $SO(3)/\mb Z_2^2$ \cite{Ki1987cartan,solomon1990isoparametric}.

There is also a related but different use of isoparametric geometry. 
Based on the isoparametric foliation, one can also reduce the construction of minimal hypersurfaces to an ODE problem, which extends the philosophy of Hsiang-Lawson. 
Very recently, Lai-Wei \cite{lai2026isoparametric} constructed generalized rotational minimal hypersurfaces in $\mb S^n$ of topological type $S^1\times M^{n-2}$, where $M^{n-2}$ is an isoparametric hypersurface in a subsphere $\mb S^{n-1}$. 
Related constructions were also developed by Firester-Tsiamis \cite{firester2026topology}. 
These methods produce further examples in $\mb S^{n\geq 4}$ dictated by the initial isoparametric data, but in $\mb S^4$ they still recover only the known types $S^1\times S^2$ and $T^3$.

Finally, Almgren-Pitts min-max theory \cite{almgren1962homotopy,pitts2014existence,schoen1981regularity} provides a flexible variational approach. 
Its recent developments have led to major breakthroughs, including the resolution of the Willmore conjecture \cite{marques2014min}, the proof of Yau's abundance conjecture \cite{marques2017existence,song2018existence}, results on the distribution of minimal hypersurfaces \cite{liokumovich2018weyl,irie2018density,marques2019equidistribution,song2021scarring}, and the Morse theory for the area functional \cite{marques2016morse,marques2021morse,marques2023morseinequalities,zhou2020multiplicity}. 
However, the topology of Almgren-Pitts min-max hypersurfaces is generally not controlled. 
In dimension three, the Simon-Smith min-max theory \cite{smith1983existence,colding2002minmax} and its refinements \cite{delellis2010genus,ketover2019genus} retain more topological information and have led to many constructions of closed minimal surfaces \cite{ketover2016equivariant,ketover2020catenoid}. 
This framework, however, is intrinsically three-dimensional and does not directly apply to minimal hypersurfaces in $\mb S^n$, $n\geq 4$.

Compared with the methods reviewed above, the equivariant min-max approach used in this paper is more flexible with respect to the ambient metric, requiring only suitable symmetries and curvature assumptions rather than the round metric itself (Remark \ref{Rem: valid in symmetric S4 with Ric>0 for n-even}). 
Additionally, it allows one to produce minimal hypersurfaces with a wider range of topological types in $\mb S^n$ with a fixed dimension. 
Further constructions in $\mb S^{n\geq 5}$ will be developed in forthcoming work.

\begin{acknowledgements}
    The author thanks Xin Zhou, Zhichao Wang, and Xingzhe Li for some helpful discussions. 
    T.W. is supported by the National Natural Science Foundation of China 12501076 and the Natural Science Foundation of Shanghai 25ZR1402252. 
\end{acknowledgements}



\section{\texorpdfstring{Isometric group actions on $\mb S^4$}{Isometric group actions on S4}}\label{Sec: group actions on S4}

In this section, we introduce some isometric group actions on $\mb S^4$ that are used to apply the equivariant min-max theory \cite{wang2026spherical}.

\subsection{Suspended Hopf action}
Let us first recall the $S^1$-action on $\mb S^4$ constructed in \cite{wang2026spherical}. 
Note that the standard {\em Hopf action} is an isometric $S^1$-action on $\mb S^3=\{(z_1,z_2)\in\mb C^2: |z_1|^2+|z_2|^2=1\}$ generated by the maps $e^{i\theta}\cdot(z_1,z_2):=(e^{i\theta}z_1,e^{i\theta} z_2)$ for all $e^{i\theta}\in S^1$. 
Thus, by regarding $\mb S^4$ as the suspension of $\mb S^3$, we can make the following definition. 
\begin{definition}\label{Def: suspended Hopf action}
	Let the unit $4$-sphere $\mb S^4$ be parameterized by
	\begin{align}\label{Eq: parameterize S4}
		\mb S^4:=\{(z_1,z_2,x)\in \mb C\times \mb C\times \mb R: |z_1|^2+|z_2|^2+|x|^2=1\}.
	\end{align}
	Then, we define the {\em suspended Hopf action on $\mb S^4$} by 
	\begin{align}\label{Eq: suspended Hopf action}
		e^{i\theta}\cdot (z_1,z_2,x) = R_\theta(z_1,z_2,x) := (e^{i\theta}z_1, e^{i\theta}z_2, x),\qquad \forall e^{i\theta}\in S^1=\{\lambda \in \mb C: |\lambda|=1\},
	\end{align}
where $(z_1,z_2,x)\in \mb S^4$. 
\end{definition}

This $S^1$-action has the fixed points given by the north and south poles $(0^{(2)},0^{(2)},\pm 1)$, and is a free action on 
\begin{align}
	\mathring{\mb S}^4:= \{(z_1,z_2,x)\in \mb S^4: -1<x<1\}. 
\end{align}
Note that for any fixed $x=\cos(\alpha)\in (-1,1)$, the above $S^1$-action restricted to $\mb S^4\cap \{x=\cos(\alpha)\} = \mb S^3_{\sin(\alpha)}$ is the standard Hopf action.  
Additionally, the Hopf map $h: \mb S^3_1\to \mb S^2_{1/2}$ defined by 
\begin{align}\label{Eq: Hopf map}
	h: (z_1,z_2)\in \mb S^3_1 \mapsto \left(z_1\bar z_2, \frac{|z_1|^2-|z_2|^2}{2} \right)\in \mb S^2_{1/2}=\{(w,t)\in\mb C\times \mb R: |w|^2+t^2=\frac{1}{4} \}
\end{align}
gives the Hopf fibration of $S^3$ over $S^2$ and is a Riemannian submersion under the associated round metrics. 
Hence, we can parameterize $\mb S^4/S^1$ by an ellipsoid 
\begin{align}\label{Eq: parameterize S4/S1}
	\mb S^4/S^1 &:= \left\{ (p, \cos(\alpha)): p\in \mb S^2_{\sin(\alpha)/2}, \alpha\in [0,\pi] \right\} \nonumber
	\\&= \left\{ (w,t,x)\in \mb C\times \mb R\times \mb R: 4\cdot (|w|^2+t^2) + x^2=1 \right\},
\end{align}
and define the metric $g_{\mb S^4/S^1}$ on $\mathring{\mb S}^4/S^1$ by 
\begin{align}\label{Eq: metric on S4/S1}
	g_{\mb S^4/S^1}= d\alpha^2+ \frac{\sin^2(\alpha)}{4} g_{\mb S^2_1}.
\end{align}
In particular, under the parameterizations \eqref{Eq: parameterize S4}\eqref{Eq: parameterize S4/S1}, the quotient map $\pi: \mb S^4\to \mb S^4/S^1$ is given by $\pi(0^{(2)},0^{(2)},\pm 1):=(0^{(2)},0,\pm 1)$ and
\begin{align}\label{Eq: quotient map of S4}
	\pi(z_1,z_2,x):= \left( \frac{z_1\bar z_2}{\sqrt{|z_1|^2+|z_2|^2}}, \frac{|z_1|^2-|z_2|^2}{2\sqrt{|z_1|^2+|z_2|^2}} , x \right) \quad \mbox{on $\mathring{\mb S}^4$},
\end{align}
which is a Riemannian submersion from $(\mathring{\mb S}^4,g_{\mb S^4})$ to $(\mathring{\mb S}^4/S^1 ,g_{\mb S^4/S^1}) $. 
For any $(p,\cos(\alpha))\in \mathring{\mb S}^4/S^1$, $\pi^{-1}((p,\cos(\alpha)))$ is a great circle in $\mb S^3_{\sin(\alpha)}\times \{\cos(\alpha)\}\subset \mb S^4$, and has the orbit volume function
\begin{align}\label{Eq: orbits volume function}
	V((p,\cos(\alpha))):= \mc H^1(\pi^{-1}((p,\cos(\alpha)))) = 2\pi \sin(\alpha). 
\end{align}
After continuously extending $V$ by $0$ to the singular points $(0^{(2)}, 0, \pm 1)\in \mb S^4/S^1$, we can define the weighted metric $\wti g_{{\mb S}^4/S^1}$ by 
\begin{align}\label{Eq: metric on S4/S1 weighted}
	\wti g_{\mb S^4/S^1} := V\cdot g_{\mb S^4/S^1}.
\end{align}
Then, it follows from the co-area formula that 
\begin{align}\label{Eq: area in S4/S1}
	\mc H^3_{g_{\mb S^4}}(\Sigma) = \mc H^2_{\wti g_{\mb S^4/S^1}}(\Sigma/S^1)
\end{align}
for any $S^1$-invariant hypersurface $\Sigma\subset \mb S^4$. 
Moreover, the following lemma is given by \cite{wang2026spherical}.

\begin{lemma}[\cite{wang2026spherical}*{Lemma 7.1-7.3} ]\label{Lem: properties of S1-action}
	Using the above notations, we have that
	\begin{itemize}
		\item[(i)] there is only one $S^1$-invariant equatorial hypersphere 
			\begin{align}\label{Eq: S1-invariant equator}
				E_0:=\mb S^4\cap \{x=0\};
			\end{align}
		\item[(ii)] $(\mathring{\mb S}^4/S^1, \wti g_{\mb S^4/S^1})$ has positive Ricci curvature;
		\item[(iii)] every embedded $S^1$-hypersurface of $\mb S^4$ is contained in $\mathring{\mb S}^4$. 
	\end{itemize}
\end{lemma}

\subsection{\texorpdfstring{Extension through the antipodal $\mb Z_2$-action}{Extension through the antipodal Z2-action}}\label{Subsec: extension by antipodal action}
By \cite{wang2026spherical}*{Lemma 7.4}, the suspended Hopf action can be extended to an isometric action on $\mb S^4$ of the Lie group 
\begin{align}\label{Eq: extended group action by antipodal}
	G=(S^1\rtimes \mb Z_4)/\mb Z_2. 
\end{align}

Specifically, let 
\begin{align}\label{Eq: antipodal}
	A: \mb S^4/S^1 \to \mb S^4/S^1, \qquad A(w,t,x):=(-w,-t,-x),
\end{align}
be the antipodal map on $\mb S^4/S^1$ under the parameterization \eqref{Eq: parameterize S4/S1}, which induces a free isometric $\mb Z_2=\langle A\rangle$ action on $\mb S^4/S^1$ under the metrics $g_{\mb S^4/S^1}$ and $\wti g_{\mb S^4/S^1}$. 
Note that the map $A$ can be lifted to an isometry $\sigma: \mb S^4\to\mb S^4$ given by
\begin{align}\label{Eq: lift antipodal map}
	\sigma(z_1,z_2,x):= (-\bar z_2, \bar z_1, -x), \qquad \forall (z_1,z_2,x)\in \mb S^4,
\end{align}
so that $A\circ \pi =\pi\circ \sigma$, which generates a $\mb Z_4=\langle\sigma\rangle$ action on $\mb S^4$ satisfying that 
\[ \sigma^2= R_\pi \qquad {\rm and}\qquad \sigma\circ R_\theta = R_{-\theta}\circ \sigma .\]
Hence, using the homomorphism $\varphi: \mb Z_4\to {\rm Aut}(S^1)$ with $\varphi(id)=\varphi(\sigma^2):= id$ and $\varphi(\sigma)= \varphi(\sigma^3): e^{i\theta}\mapsto e^{-i\theta}$, we can define the group $\wti G$ by 
\[ \wti G=S^1\rtimes_\varphi \mb Z_4 \]
with the group operation 
\[(e^{i\theta_1},\sigma^{k_1})\cdot (e^{i\theta_2},\sigma^{k_2}) := (e^{i\theta_1}\cdot \varphi(\sigma^{k_1})(e^{i\theta_2}), \sigma^{k_1}\cdot \sigma^{k_2})=(e^{i(\theta_1+(-1)^{k_1}\theta_2)}, \sigma^{k_1+k_2\mod 4}),\] 
which acts isometrically on $\mb S^4$ by 
	\[(e^{i\theta},\sigma^k)\cdot (z_1,z_2,x):= R_\theta\circ \sigma^k(z_1,z_2,x), \qquad \forall (z_1,z_2,x)\in \mb S^4 .\] 
Since $\wti G$ has a non-trivial normal subgroup $H:=\{(e^{i0},\sigma^0), (e^{i\pi},\sigma^2)\}\cong \mb Z_2$ acting as the identity on $\mb S^4$, the $\wti G$-action descends to an action on $\mb S^4$ by the quotient group $G=\wti G/H$.

\subsection{Extension through the dihedral action}

In this section, we introduce the dihedral action and extend the previous $G=(S^1\rtimes \mb Z_4)/\mb Z_2$ action on $\mb S^4$. 

To begin with, for any integer $n\in \{2,3,\dots\}$, consider the dihedral group 
\begin{align}\label{Eq: dihedral group}
	D_n=\langle r,s | r^n=s^2=e, srs=r^{-1}\rangle.
\end{align}
Then, we can define the {\em reflectional $D_n$-action on $(x,y,z)\in \mb R^3$} to be generated by 
\begin{itemize}
	\item $r$: the $\frac{2\pi}{n}$-rotation in $\mb R^3$ around the $z$-axis;
	\item $s$: the reflection in $\mb R^3$ across the $(x,z)$-plane.
\end{itemize}
Intuitively, this action is the trivial extension of the standard $D_n$-action on $\mb R^2={\rm Span} (x,y)$. 
Indeed, one can take $n$ axes $\{\xi_l\}_{l=1}^n$ in the $(x,y)$-plane as 
\begin{align}\label{Eq: Dn reflectional axis}
	\xi_l:=\left\{t\cdot \left(\cos\frac{l\pi}{n} , \sin \frac{l\pi}{n}, 0 \right)\in \mb R^3:t\in\mb R \right\},
\end{align}
and the $D_n$-action is generated by all the reflections $\psi_l$, $l=1,\dots,n$, across the $(\xi_l,z)$-plane, where 
\[ \psi_l= 
\begin{pmatrix}
\cos\frac{2\pi l}{n} & \sin \frac{2\pi l}{n} & 0 \\
\sin \frac{2\pi l}{n} & -\cos \frac{2\pi l}{n} & 0 \\
0 & 0 & 1
\end{pmatrix}
\in O(3)\setminus SO(3).
\] 
Note that 
\begin{itemize}
	\item $r=\psi_{l+1}\circ\psi_{l}=\psi_1\circ\psi_n$ for any $l=1,\cdots, n-1$;
	\item $s=\psi_n$, and $\psi_{n-l}=sr^{l}$, $l=1,\dots, n-1$.
\end{itemize}
For simplicity, we identify $\mb R^3$ with $\mb C\times \mb R$, and the reflectional $D_n$-action is generated by 
\begin{align}\label{Eq: Dn reflectional action on R3}
	r: (w,t)\mapsto (e^{\frac{2\pi}{n}i}\cdot w, t) \qquad{\rm and}\qquad s: (w,t) \mapsto (\bar w, t),
\end{align} 
where $(w,t)\in \mb C\times \mb R$. 

\begin{remark}\label{Rem: rotational Dn action}
	The aforementioned action should be distinguished from the {\em rotational dihedral action} used in \cite{ketover2016free} and \cite{carlotto2022free}, where $s\in D_n$ acts by the rotation in $\mb R^3$ around the $x$-axis (i.e. the $(3,3)$-element is $-1$ in the matrices $\{\psi_l\}_{l=1}^n$). 
\end{remark}

\subsubsection{\texorpdfstring{Modified reflectional dihedral action on $\mb S^4/S^1$}{Modified reflectional dihedral action on S4/S1}}
By definition, the reflectional $D_n$-action on $\mb R^3$ is a family of orthogonal transformations, which induces an isometric $D_n$-action on each round $2$-sphere $\mb S^2_r\subset \mb R^3$. 
Recall that the orbit space (see \eqref{Eq: parameterize S4/S1})
\[\mb S^4/S^1=\{(w,t,x)\in\mb C\times \mb R\times \mb R: 4\cdot (|w|^2+t^2)+x^2=1\}\] 
with the metric $g_{\mb S^4/S^1}$ in \eqref{Eq: metric on S4/S1} is the suspension of a round $2$-sphere of radius $1/2$. 
Namely, for any fixed $\alpha\in [0,\pi]$, the slice of $\mb S^4/S^1$ at $x=\cos(\alpha)$ is a round $2$-sphere of radius $\sin(\alpha)/2$. 
Hence, we can suspend the $D_n$-action \eqref{Eq: Dn reflectional action on R3} along the $x$-axis and simultaneously reflect the $x$-coordinate to obtain the following orientation-preserving isometric $D_n$-action on $\mb S^4/S^1$. 
 
\begin{definition}\label{Def: reflectional dihedral action on S4/S1}
	Using the parameterization \eqref{Eq: parameterize S4/S1} and the metrics $g_{\mb S^4/S^1},\wti g_{\mb S^4/S^1}$ of $\mb S^4/S^1$, the {\em modified reflectional $D_n$-action on $\mb S^4/S^1$} is an orientation-preserving isometric $D_n$-action generated by 
	\begin{align}\label{Eq: Dn reflectional action on S4/S1}
	r: (w,t,x) \mapsto (e^{\frac{2\pi}{n}i} w, t, x)\qquad {\rm and}\qquad s: (w,t,x)\mapsto (\bar w, t, -x),
\end{align}
where $(w,t,x)\in \mb S^4/S^1$, and $r,s\in D_n$. 
\end{definition}
\begin{remark}
	If one simply takes the action of $s $ as $(w,t,x)\mapsto (\bar w, t,x)$ by suspending along the $x$-axis, then  
    such an action fails to be orientation-preserving, and thus $\mb S^4/S^1$ admits special exceptional orbits (\cite{wang2026spherical}*{Definition 2.2}), which violates the assumption of the equivariant min-max in \cite{wang2026spherical}. 
	Hence, we modify the suspended action of $s$ by reflecting the $x$-coordinates simultaneously so that $s$ becomes the $\pi$-rotation in the $({\rm Im}(w),x)$-plane.
\end{remark}

Consider the stereographic projection of $\mb S^4/S^1$ at the north pole $(0^{(2)},0,1)$:
\begin{align}\label{Eq: stereographic projection}
	P: (\mb S^4/S^1) \setminus \{(0^{(2)},0,1)\} \to \mb R^3, \qquad P(w, t, x ):=\left(\frac{{\rm Re} (w)}{1-x}, \frac{{\rm Im}(w)}{1-x}, \frac{t}{1-x} \right). 
\end{align}
Then, $\mb S^4/S^1$ can be identified with $\mb R^3\cup\{\infty\}$, and the action of $r\in D_n$ on $\mb S^4/S^1$ in \eqref{Eq: Dn reflectional action on S4/S1} can be identified with the action of $r\in D_n$ on $\mb R^3$ in \eqref{Eq: Dn reflectional action on R3}, i.e. 
\begin{align}\label{Eq: r action using stereographic projection}
	P\circ r= r\circ P ~:~ (\mb S^4/S^1) \setminus \{(0^{(2)},0,1)\}\to \mb R^3. 
\end{align}
However, the two actions of $s\in D_n$ in \eqref{Eq: Dn reflectional action on S4/S1} and \eqref{Eq: Dn reflectional action on R3} are related by
\begin{align}\label{Eq: s action using stereographic projection}
	P\circ s= I\circ s\circ P ~:~ (\mb S^4/S^1) \setminus \{(0^{(2)},0,\pm 1)\}\to \mb R^3\setminus\{0\}, 
\end{align}
where $I:\mb R^3\setminus\{0\}\to \mb R^3\setminus\{0\}$, $I(x,y,z):=\frac{(x,y,z)}{4(x^2+y^2+z^2)}$ is the inversion across the sphere $\mb S^2_{1/2}=P(E_0)$. 
Nevertheless, if one takes the stereographic projection $\widehat P$ of $\mb S^4/S^1$ at $(0^{(2)}, 1/2, 0)$, i.e. $\widehat P: (w,t,x)\in (\mb S^4/S^1)\setminus \{(0^{(2)},1/2,0)\} \mapsto (\frac{2{\rm Re}(w)}{1-2t}, \frac{2{\rm Im}(w)}{1-2t}, \frac{x}{1-2t})\in \mb R^3$, then for $r,s\in D_n$ in \eqref{Eq: Dn reflectional action on S4/S1}, $\widehat P\circ r\circ \widehat P^{-1}$ and $\widehat P\circ s\circ \widehat P^{-1}$ generate the rotational $D_n$-action on $\mb R^3$ used in \cite{ketover2016free}\cite{carlotto2022free}. 
Hence, the modified reflectional $D_n$-action in Definition \ref{Def: reflectional dihedral action on S4/S1} is obtained by suspending the rotational $D_n$-action on the ellipsoid $\{(w,x)\in\mb C\times \mb R: 4|w|^2+x^2=1\}$ along the $t$-axis.

\medskip
Before proceeding, we introduce some notations that will be used in the next section. 
For any $l\in \{1,\dots, n\}$, let 
\begin{align}\label{Eq: neck center for odd n}
	\hat\xi_l:= \left\{t\cdot \left(\cos(\frac{l\pi}{n}-\frac{\pi}{2n}) , \sin (\frac{l\pi}{n}-\frac{\pi}{2n}), 0 \right)\in \mb R^3:t\in\mb R \right\}
\end{align}
be the axes in the $(x,y)$-plane of $\mb R^3$ bisecting $\xi_l$ and $\xi_{l-1}$, where $\xi_{0}=\xi_{n}$ (Figure \ref{Fig: axes}). 
\begin{figure}[h]
\centering
\includegraphics[height=1.5in]{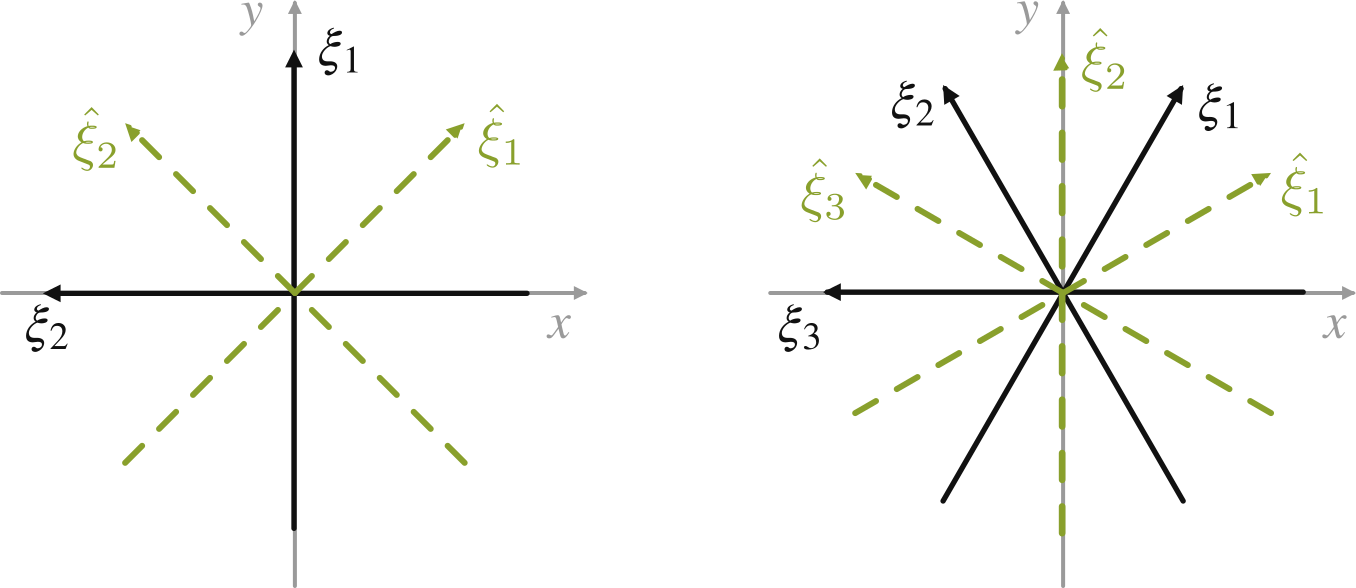}
\caption{Axes $\{\xi_l\}_{l=1}^n$ and $\{\hat\xi_l\}_{l=1}^n$ for $n=2,3$.}
\label{Fig: axes}
\end{figure}
For our convenience, we also introduce the following notations for $l\in\{1,\dots, n\}$:
\begin{align}
	\Xi_l^\pm &:=\left\{ \left(\pm\frac{\sin(\alpha)}{2}\cdot e^{\frac{l\pi}{n}i}, 0 , \cos(\alpha)\right)\in \mb C\times \mb R\times \mb R : \alpha\in [0,\pi] \right\}, \label{Eq: half axis xi on S3}
	\\
	\widehat\Xi_l^\pm &:=\left\{ \left(\pm\frac{\sin(\alpha)}{2}\cdot e^{(\frac{l\pi}{n}-\frac{\pi}{2n})i}, 0 , \cos(\alpha)\right)\in \mb C\times \mb R\times \mb R : \alpha\in [0,\pi] \right\}, \label{Eq: half axis hat xi on S3}
    \\
    \Xi_0^\pm &:= \left\{\left (0^{(2)},\pm\frac{\sin(\alpha)}{2}, \cos(\alpha) \right) \in  \mb C\times \mb R\times \mb R: \alpha\in [0,\pi] \right\}, \label{Eq: half axis z on S3}
\end{align}
which are the longitudes on $\mb S^4/S^1$ corresponding to the positive/negative semi-axes of $\xi_l,\hat\xi_l,z$ in $ \mb R^3$ via the stereographic projection $P$ respectively. 
For $l\in\{1,\dots, n\}$, denote by
\begin{align}\label{Eq: axes on S3}
    \Xi_l := \Xi_l^+\cup \Xi_l^-,\qquad 
    \widehat \Xi_l :=  \widehat\Xi_l^+\cup \widehat\Xi_l^-, 
    \qquad
    \Xi_0 := \Xi_0^+\cup\Xi_0^-
\end{align}
the circles on $\mb S^4/S^1$ corresponding to the $\xi_l$-axis, $\hat \xi_l$-axis and $z$-axis in $\mb R^3$ respectively. 
Moreover, for each $1\leq l \leq n$, we define 
\begin{align}\label{Eq: circle on E0/S1}
	 \Upsilon_l:=\left\{ \left( \frac{\sin(\beta)}{2}e^{\frac{l\pi}{n}i}, \frac{\cos(\beta)}{2} , 0 \right) \in \mb C\times \mb R\times \mb R: \beta\in [0,2\pi] \right\}\subset E_0/S^1, 
\end{align}
to be the great circle on $E_0/S^1=\mb S^2_{1/2}\times \{0\}$ containing $E_0/S^1 \cap \Xi_l$. 

\begin{lemma}\label{Lem: fixed points set of s}
	$\Upsilon_l$ is the fixed-point set of the $\mb Z_2=\langle s r^{n-l}\rangle$ action on $\mb S^4/S^1$. 
\end{lemma}
\begin{proof}
	Since $\{(\sin(\beta)  e^{\frac{l\pi}{n}i}, \cos(\beta)):\beta\in [0,2\pi]\}\subset  \mb C\times \mb R$ is the fixed-point set of the $\mb Z_2=\langle s r^{n-l}\rangle$ action on $\mb S^2\subset \mb C\times \mb R$ by \eqref{Eq: Dn reflectional action on R3}, the desired result follows directly from Definition \ref{Def: reflectional dihedral action on S4/S1}. 
\end{proof}

\subsubsection{\texorpdfstring{Lifted dihedral group action on $\mb S^4$}{Lifted dihedral group action on S4}}
Since the antipodal map $A: (w,t,x)\mapsto (-w, -t, -x)$ on $\mb S^4/S^1$ is expressed as the scalar matrix $-I_4$, we know the modified reflectional $D_n$-action commutes with the antipodal $\mb Z_2$-action. 
Hence, we obtain a finite group 
\[\mb Z_2\times D_n\] 
acting by orientation-preserving isometries on $\mb S^4/S^1$ under the metrics $g_{\mb S^4/S^1}$ and $\wti g_{\mb S^4/S^1}$.

In Section \ref{Subsec: extension by antipodal action}, we have shown that the antipodal map $A$ on $\mb S^4/S^1$ can be lifted to a map $\sigma$ on $\mb S^4$ so that the actions of $S^1$ and $\sigma$ generate a group $G= (S^1\rtimes \mb Z_4)/ \mb Z_2$ action on $\mb S^4$. 
We next lift the $D_n$-action into $\mb S^4$ and further extend the $G$-action.  

\begin{definition}\label{Def: reflectional dihedral action on S4}
	Using the parameterization \eqref{Eq: parameterize S4}, define the {\em modified reflectional $D_n$-action on $\mb S^4$} to be an isometric $D_n$-action generated by  
	\begin{align}\label{Eq: Dn reflectional action on S4}
	\hat  r (z_1,z_2,x) := (e^{\frac{2\pi}{n}i} z_1, z_2, x) \qquad{\rm and}\qquad \hat  s (z_1,z_2,x):= ( \bar z_1,  \bar z_2, -x),
\end{align}
where $(z_1,z_2,x)\in\mb S^4$ and $D_n:=\langle \hat  r,\hat  s: \hat  r^n=\hat  s^2=e, \hat  s\hat  r\hat  s=\hat  r^{-1}\rangle$. 
\end{definition}

Using the expression \eqref{Eq: quotient map of S4} of the quotient map $\pi: \mb S^4\to \mb S^4/S^1$, one verifies that 
\[\pi\circ \hat  r = r\circ \pi, \qquad {\rm and}\qquad \pi \circ \hat  s= s\circ \pi.\]
Hence, the modified reflectional $D_n$-action on $\mb S^4$ is the lift of the $D_n$-action on $\mb S^4/ S^1$ defined in \eqref{Eq: Dn reflectional action on S4/S1}. 
Additionally, using \eqref{Eq: suspended Hopf action}, \eqref{Eq: lift antipodal map}, and \eqref{Eq: Dn reflectional action on S4}, we have 
\begin{align}\label{Eq: relation of the lifted actions}
	\hat  r\circ R_\theta = R_\theta \circ \hat  r, \qquad \hat  r\circ \sigma= R_{\frac{2\pi}{n}}\circ\sigma\circ\hat  r, \qquad \hat  s\circ R_\theta= R_{-\theta}\circ \hat  s, \qquad \hat  s\circ \sigma = \sigma\circ\hat  s. 
\end{align}
Therefore, we can take the homomorphism
\[ \psi:D_n\to {\rm Aut}(G), \qquad \psi(\hat  s^a \hat  r^b): [(e^{i\theta}, \sigma^m)]\in G\mapsto [(e^{(-1)^a(\theta+\frac{2\pi mb}{n} ) i} , \sigma^m )]\in G,\]
where $m\in \{0,1\}$ and $\theta\in \mb R/ 2\pi \mb Z$; 
and define the group $ G_n$ by 
\begin{align}\label{Eq: Gn action on S4}
	G_n := G\rtimes_\psi D_n
\end{align}
with the group operation 
\[ (g_1,\hat  s^{a_1}\hat  r^{b_1})\cdot (g_2,\hat  s^{a_2}\hat  r^{b_2}) := (g_1\cdot \psi(\hat  s^{a_1}\hat  r^{b_1})(g_2), \hat  s^{a_1}\hat  r^{b_1}\hat  s^{a_2}\hat  r^{b_2}), \]
where $g_1,g_2\in G=(S^1\rtimes \mb Z_4)/\mb Z_2$ and $\hat  s^{a_1}\hat  r^{b_1}, \hat  s^{a_2}\hat  r^{b_2}\in D_n$. 
One verifies that the above definitions are all well-defined. 

Finally, by \eqref{Eq: suspended Hopf action}\eqref{Eq: lift antipodal map}\eqref{Eq: Dn reflectional action on S4} and \eqref{Eq: relation of the lifted actions}, the group $G_n$ acts by isometries on $\mb S^4$ via the maps
\[ ([(e^{i\theta},\sigma^m)], \hat  s^a\hat  r^b):= R_\theta\circ\sigma^m\circ \hat  s^a\circ \hat  r^b ~: ~\mb S^4\to \mb S^4, \]
for all $([(e^{i\theta},\sigma^m)], \hat  s^a \hat r^b)\in G_n$ with $m\in \{0,1\}$. 
Since the $G_n$-action on $\mb S^4$ reduces to the orientation-preserving $\mb Z_2\times D_n$-action on $\mb S^4/S^1$, there is no special exceptional $G_n$-orbit on $\mb S^4$ (\cite{wang2026spherical}*{Definition 2.2}).

\subsection{Expressions using quaternions}	

Using quaternions, we obtain simpler expressions for the above group actions. 
More precisely, let 
\[ \mb  H = \{a + bi + cj + dk: a, b, c, d\in \mb R\} \]
be the space of quaternions, where the symbols $i,j,k$ satisfy $i^2 = j^2 = k^2 =-1$, $ij = -ji = k$, $jk = -kj = i$, and $ki = -ik = j$. 
Then, by identifying 
\[ (z_1,z_2)\in \mb C^2  \quad {\rm and }\quad q=z_1+z_2j\in \mb H,\]
we can reparameterize $\mb S^4$ as
\begin{align}\label{Eq: parameterize S4 using quaternion}
	\mb S^4:=\{(q,x)\in \mb H\times \mb R: |q|^2+x^2=1\},
\end{align}
and the suspended Hopf action on $\mb S^4$ can be expressed as 
\[ e^{i\theta}:  (q,x)\in \mb S^4 \mapsto (e^{i\theta}\cdot q, x)\in \mb S^4, \qquad \forall e^{i\theta}\in S^1=\{\lambda\in \mb C\subset\mb H: |\lambda|=1\}. \]
Additionally, one verifies that the lift $\sigma$ \eqref{Eq: lift antipodal map} of the antipodal map $A$ can be written as 
\[\sigma : (q,x)\in\mb S^4  \mapsto (jq, -x)\in \mb S^4 .\]
Consider the pin group 
\begin{align}\label{Eq: pin group}
	{\rm Pin}_-(2):= S^1\cup S^1\cdot j \subset \mb H
\end{align}
with the group operation given by the multiplication of quaternions. 
Then, we see that $S^1=\{e^{i\theta}\in \mb C\subset \mb H\}$ and $\sigma$ generate a ${\rm Pin}_-(2)$-action on $\mb S^4$: 
\[e^{i\theta} j^m~:~ (q,x)\in \mb S^4 ~\mapsto ~(e^{i\theta}j^m q, (-1)^m x)\in \mb S^4 \]
for any $e^{i\theta}j^m \in {\rm Pin}_-(2)$. 
As we pointed out in \cite{wang2026spherical}*{Remark 7.5}, ${\rm Pin}_-(2)$ is isomorphic to $G=(S^1\rtimes\mb Z_4)/\mb Z_2$ via the map $e^{i\theta}j^m\in {\rm Pin}_-(2)\mapsto [(e^{i\theta}, \sigma^m)]\in G$.

Next, the modified reflectional $D_n$-action on $\mb S^4$ in Definition \ref{Def: reflectional dihedral action on S4} can also be rewritten using the quaternion coordinates. 
Specifically, we can write
\[ \hat  r: (q,x)\in \mb S^4\mapsto (e^{\frac{\pi}{n}i} \cdot q\cdot e^{\frac{\pi}{n}i}, x)\in \mb S^4
	\quad{\rm and }\quad 
	\hat  s: (q,x)\in \mb S^4\mapsto (j q j^{-1}, -x)\in \mb S^4,\]
which coincides with the action in \eqref{Eq: Dn reflectional action on S4}. 
Let $\psi: D_n\to {\rm Aut}({\rm Pin}_-(2))$ be the homomorphism defined by $\psi(\hat  s^a\hat  r^b): e^{i\theta} j^m\in   {\rm Pin}_-(2) \mapsto e^{(-1)^a(\theta+\frac{2\pi mb}{n} )i} j^m \in   {\rm Pin}_-(2)$, where $\theta\in \mb R/2\pi\mb Z$ and $m\in\{0,1\}$. 
Then the group 
\begin{align}\label{Eq: G_n action using quaternions}
	G_n:= {\rm Pin}_-(2)\rtimes_\psi D_n
\end{align}
acts on $\mb S^4$ by isometries: for $\theta\in \mb R/2\pi\mb Z$, $m,a\in\{0,1\}$, and $b\in \{0,1,\dots,n-1\}$,
\[ (e^{i\theta} j^m, \hat  s^a\hat  r^b): (q,x)\in \mb S^4 \mapsto (e^{i\theta}\cdot j^{m+a}\cdot e^{\frac{b\pi}{n}i} \cdot q \cdot  e^{\frac{b\pi}{n}i} \cdot j^{-a}, (-1)^{m+a}\cdot x)\in \mb S^4 ,\]
which recovers the isometric $G_n$-action on $\mb S^4$ defined in \eqref{Eq: Gn action on S4}.

\section{Notations in the equivariant min-max theory}\label{Sec: min-max}	
	
In this section, we collect some definitions in the equivariant min-max theory (cf. \cite{wang2026spherical}). 

\begin{definition}
    Let the compact Lie group $G_n$ act isometrically on $\mb S^4$ as defined in \eqref{Eq: Gn action on S4}. 
    A family of closed $3$-rectifiable subsets $\{\wti\Sigma_t\}_{t\in [0,1]}$ is said to be a {\em $G_n$-equivariant sweepout} of $\mb S^4$ if 
    \begin{itemize}
        \item for any $t\in [0,1]$, $\wti\Sigma_t$ is $G_n$-invariant, i.e. $g\cdot \wti\Sigma_t=\wti\Sigma_t$ for all $g\in G_n$;
        \item for any $t\in (0,1)$, $\wti\Sigma_t\subset \mb S^4$ is an embedded closed $G_n$-hypersurface;
        \item the map $t\mapsto\wti\Si_t$ is continuous in the smooth topology for $t\in (0,1)$, and is continuous in the varifold (weak) topology and the Hausdorff topology for $t\in [0,1]$. 
    \end{itemize}
\end{definition}

The above definition is a simplified version of \cite{wang2026spherical}*{Definition 6.1} for general groups. 
One can indeed allow every $\wti\Sigma_t$ to have a finite set $P(t)$ of singular orbits with $\sup_{t\in (0,1)}\# P(t)<\infty$ so that $\wti\Sigma_t$ varies smoothly away from $P(t)$ (cf. \cite{wang2026spherical}*{Remark 6.2}). 
Additionally, in \cite{wang2026spherical}*{Definition 6.1}, the $G$-invariant hypersurfaces $\wti \Sigma_t$, $t\in (0,1)$, are also required to be of {\em locally $G$-boundary-type}, i.e. in any small neighborhood $B_r(G\cdot u)$, $\wti\Sigma_t\llcorner B_r(G\cdot u)=\bd \wti \Omega$ for some open $G$-set $\wti\Omega\subset B_r(G\cdot u)$, which is used in the regularity theory \cite{wang2026spherical}*{\S 6.3}. 
Nevertheless, for the specific action of $G=G_n$ in \eqref{Eq: Gn action on S4}, we can remove this assumption by combining the regularity results of \cite{wang2026spherical}*{Theorem 6.5} and \cite{ketover2016free}*{Theorem 3.2} (see the proof of Theorem \ref{Thm: new minimal hypersurfaces of odd n}).

\begin{definition}
    Given a $G_n$-equivariant sweepout $\{\wti \Sigma_t\}_{t\in [0,1]}$, we define the $G_n$-equivariant $([0,1],\{0,1\})$-homotopy class $\Pi$ of $\{\wti\Sigma_t\}_{t\in [0,1]}$ by the set 
    $\Pi:=\{ \{\psi(t,\wti\Sigma_t)\}_{t\in [0,1]}  \}$,
    where $\psi:[0,1]\times \mb S^4\to \mb S^4$ is any smooth map with $\psi(0,\cdot)=\psi(1,\cdot)=id_{\mb S^4}$ so that for any $t\in (0,1)$, $\psi(t,\cdot)=\varphi_1^{(t)}$ for some $G_n$-equivariant isotopy $\{\varphi_s^{(t)}\}_{s\in [0,1]}$ of $\mb S^4$. 
    Then, the {\em min-max width $\mf L(\Pi)$} of $\Pi$ is defined by 
    \[\mf L(\Pi):= \inf_{\{\wti\Sigma'_t\}_{t\in[0,1]} \in \Pi} \sup_{t\in [0,1]} \mc H^3(\wti\Sigma_t'). \]
    Additionally, a sequence $\{\{\wti\Sigma_t^i\}_{t\in [0,1]}\}_{i\in\mb N}\subset \Pi$ is called a {\em minimizing sequence} if 
	\[ \mf L(\{\wti\Sigma_t^i\}_{t\in [0,1]}) := \sup_{t\in [0,1]} \mc H^3_{g_{\mb S^4}}(\wti\Sigma_t^i) ~\to~  \mf L(\Pi)\qquad \mbox{as $i\to\infty$}.  \]
	Given a minimizing sequence $\{\{\wti\Sigma_t^i\}_{t\in [0,1]}\}_{i\in\mb N}\subset \Pi$, a subsequence $\{\wti\Sigma_{t_j}^{i_j}\}_{j\in\mb N}$ with $t_j\in [0,1]$ is called a {\em min-max sequence} if $\mc H^3_{g_{\mb S^4}}(\wti \Sigma_{t_j}^{i_j}) ~\to~ \mf L(\Pi)$ as $ j\to\infty$. 
\end{definition}

If $\mf L(\Pi)>\max\{ \mc H^3_{g_{\mb S^4}}(\wti\Sigma_0), \mc H^3_{g_{\mb S^4}}(\wti\Sigma_1), 0 \}$ (cf. \cite{wang2026spherical}*{(6.1)}), then we can apply the equivariant min-max theory \cite{wang2026spherical}*{Theorem 6.5, 6.6} to find a min-max sequence $\{\wti\Sigma_{t_j}^{i_j}\}_{j\in\mb N}$ so that $|\wti\Sigma_{t_j}^{i_j} |$ converges to a stationary $G_n$-invariant varifold $V$ with certain regularity and topological controls on $\spt(\|V\|)$. 
Here and below, we use the following notations in geometric measure theory:
\begin{itemize}
    \item $|\wti\Sigma|$: the integer rectifiable varifold induced by a hypersurface $\wti\Sigma$;
    \item $\llbracket \wti\Sigma\rrbracket$: the mod $2$ flat chain associated with a hypersurface $\wti\Sigma$;
    \item $\mc C^{G_n}(\mb S^4)$: the space of $G_n$-invariant Caccioppoli sets $\wti\Omega$ in $\mb S^4$. 
\end{itemize}
We say that $\Phi:t\in [0,1]\mapsto\llbracket \wti\Sigma_t\rrbracket$ is a {\em (non-trivial) $1$-sweepout in the sense of Almgren-Pitts} if there exists a continuous family of Caccioppoli sets $\{\wti\Omega_t\}_{t\in [0,1]}$ with induced boundary $\bd\wti\Omega_t=\llbracket \wti\Sigma_t\rrbracket$ so that $\wti\Omega_1=\mb S^4\setminus \wti\Omega_0$, which implies that $\mf L(\Pi)>0$ by Almgren's isomorphism \cite{almgren1962homotopy}. 

Finally, we mention that the definitions and notations in this section can also be made in the orbit space $\mb S^4/S^1$ by using $\mb S^4/S^1$ and $\mb Z_2\times D_n$ in place of $\mb S^4$ and $G_n$ respectively.

\section{\texorpdfstring{Embedded minimal principal $S^1$-bundles in $\mb S^4$}{Embedded minimal principal S1-bundles in S4}}\label{Sec: new minimal hypersurfaces in S4}

In Section \ref{Sec: group actions on S4}, we defined the suspended $S^1$-action on $\mb S^4$ and its extensions $G={\rm Pin}_-(2)$ and $G_n={\rm Pin}_-(2)\rtimes D_n$ through the antipodal $\mb Z_2$-action and the modified reflectional $D_n$-action. 
In this section, we use these symmetries and the equivariant min-max theory \cite{wang2026spherical}*{Theorem 1.5}\cite{ketover2016equivariant}*{Theorem 1.3} to construct embedded minimal principal $S^1$-bundles in $\mb S^4$. 

\subsection{\texorpdfstring{New minimal $G_n$-hypersurfaces in $\mb S^4$ for even $n$}{New minimal Gn-hypersurfaces in S4 for even n}}\label{Subsec: new minimal hypersurface for even n}

Fix any even integer 
\[ n =2k \geq 2.\]
In this subsection, we will construct a family of minimal $G_n$-hypersurfaces in $\mb S^4$ that are topologically trivial principal $S^1$-bundles. 
We begin with the construction of a $G_n$-sweepout. 

\begin{lemma}\label{Lem: Gn-sweepout n-even}
	For any even integer $n=2k\geq 2$, there is a $G_n$-sweepout $\{\wti\Sigma_t\}_{t\in [0,1]}$ in $\mb S^4$ so that 
	\begin{itemize}
		\item[(i)] $\wti\Sigma_0=\cup_{b=1}^k \pi^{-1}( \Upsilon_{2b-1} )$, where each $ \Upsilon_{2b-1}$ is a great circle on $E_0/S^1$ (see  \eqref{Eq: circle on E0/S1});
		\item[(ii)] $\wti\Sigma_1=\cup_{b=1}^k\pi^{-1} (\Xi_{2b})$, where each $\Xi_{2b}$ is the circle on $\mb S^4/S^1$ corresponding to the $\xi_{2b}$-axis in $\mb R^3$ via the stereographic projection (see \eqref{Eq: axes on S3});
		\item[(iii)] for $t\in (0,1)$, $\wti\Sigma_t\subset \mb S^4$ is an embedded $G_n$-hypersurface with $\wti \Sigma_t=\bd \wti \Omega_t$ for some open $G_n$-set $\wti\Omega_t\subset \mb S^4$ so that $\wti\Sigma_t/{\rm Pin}_-(2)$ is an orientable surface in $RP^3=\mb S^4/{\rm Pin}_-(2)$ and 
		\[ \genus(\wti\Sigma_t/S^1) = n-1 = 2k-1; \]
		\item[(iv)] $\sup_{t\in [0,1]} \mc H^3_{g_{\mb S^4}}(\wti\Sigma_t) < 2 \mc H^3_{g_{\mb S^4}}(E_0)=4\pi^2$.
	\end{itemize}
	Additionally, $\Phi(t):=\llbracket \wti\Sigma_t\rrbracket$, $t\in [0,1]$ forms an Almgren-Pitts $1$-sweepout. 
\end{lemma}

\begin{proof}
	It is sufficient to construct the corresponding $\mb Z_2\times D_n$-equivariant sweepout $\{\Sigma_t\}_{t\in [0,1]}$ in $\mb S^4/S^1$. 
	Recall that $\Xi_l^\pm$, $\widehat \Xi_l^\pm$, $\Xi_l$, and $\widehat \Xi_l$ are defined in \eqref{Eq: half axis xi on S3}-\eqref{Eq: axes on S3} respectively for $l\in\{1,\dots, n\}$, and $\mb S^4/S^1$ is parameterized as the ellipsoid in \eqref{Eq: parameterize S4/S1}. 
	
	For any small $\epsilon>0$, let $B_\epsilon(\Xi_l)$ be a small $\epsilon$-neighborhood of the circle $\Xi_l$ in $\mb S^4/S^1$ , where $l=1,\dots, n$. 
	Since the circle $\Xi_l$ is transversal to each slice
    \begin{align}\label{Eq: slict at x=rho}
        S_\rho:= \{x=\rho\}\cap \mb S^4/S^1=\{(w,t,\rho)\in \mb C\times \mb R\times \mb R: 4(|w|^2+t^2)+\rho^2=1 \}, \quad \forall\rho\in (-1,1),
    \end{align}
	we see that $\bd B_\epsilon(\Xi_l)$ is transversal to $S_\rho$ by taking $\epsilon>0$ small enough. 
	Additionally, there exists a continuous function $\epsilon_0: [0,1] \to [0, \infty)$ so that
	\begin{itemize}
		\item $\lim_{|\rho|\to  1}\epsilon_0(|\rho|) = 0$, $\lim_{|\rho|\to 0} \epsilon_0(|\rho|) = \tilde\epsilon_0>0$; 
		\item for any $\epsilon\in (0, 2\epsilon_0(|\rho|))$ and $l\neq l'\in \{1, \dots, n\}$, $B_\epsilon(\Xi_l)\cap S_\rho$ is the disjoint union of two disks, and $S_\rho\cap B_\epsilon(\Xi_l)\cap B_\epsilon(\Xi_{l'})=\emptyset$. 
	\end{itemize}
	Therefore, using the parameterization \eqref{Eq: parameterize S4/S1}, we can define a connected surface $\Gamma_{\rho,\epsilon}$ in $\mb S^4/S^1$ for $\rho\in (0,1)$ and $\epsilon\in (0,2\epsilon_0(\rho))$ by
    \begin{align}\label{Eq: connect 2 spheres by n-1 necks}
        \Gamma_{\rho,\epsilon}:= [(S_\rho \cup S_{-\rho})\setminus \cup_{b=1}^k B_\epsilon(\Xi_{2b})] \cup \{(w,t,x)\in \cup_{b=1}^k\bd B_\epsilon(\Xi_{2b}): -\rho<x<\rho \},
    \end{align}
	which is formed by two spheres $S_\rho \cup S_{-\rho}$ parallel to the equator $E_0/S^1=S_0$ that are connected via $n=2k$ necks along $\cup_{b=1}^k\Xi_{2b}$ (see the left panel of Figure \ref{Fig: sweepout n-even} for the top view under the projection $P$ \eqref{Eq: stereographic projection}). 
	\begin{figure}[h]
    \centering
    \includegraphics[height=1.75in]{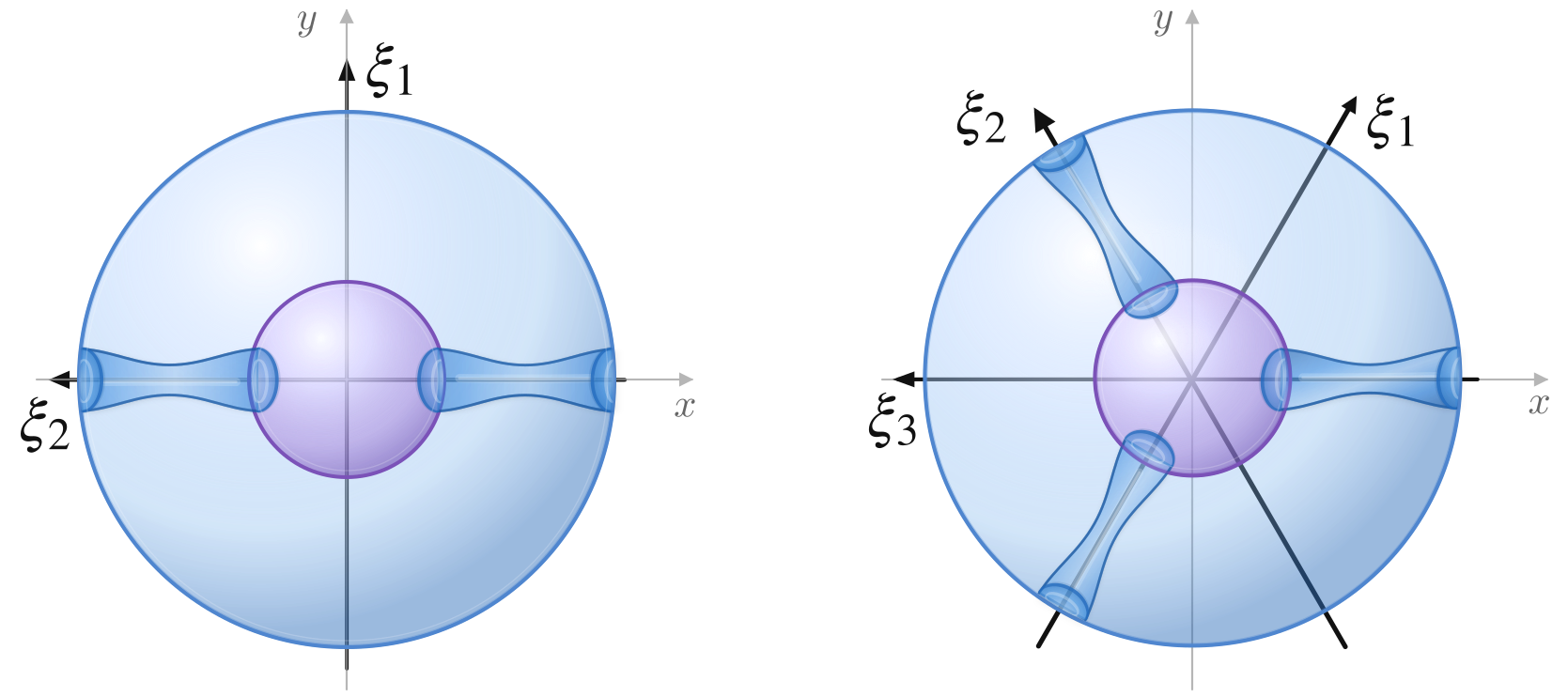}
    \caption{$\mb Z_2\times D_n$-invariant $\Gamma_{\rho,\epsilon}$ ($n=2$), and $D_n$-invariant $\widehat\Gamma_{\rho,\epsilon}$ ($n=3$).}
    \label{Fig: sweepout n-even}
    \end{figure}

	Recall that $n=2k$ is even. 
	Hence, $\cup_{b=1}^k\xi_{2b}$ (see \eqref{Eq: Dn reflectional axis}) is invariant under the reflectional $D_n$-action \eqref{Eq: Dn reflectional action on R3} on $\mb R^3$. 
	Then, one can see from \eqref{Eq: r action using stereographic projection} and \eqref{Eq: s action using stereographic projection} that $\cup_{b=1}^k \Xi_{2b}$ is $\mb Z_2\times D_n$-invariant, which implies that $\Gamma_{\rho,\epsilon}$ is $\mb Z_2\times D_n$-invariant. 
	Additionally, by the constructions, we have 
	\begin{itemize}
		\item $\genus(\Gamma_{\rho,\epsilon})=n-1=2k-1$;
		\item $\Gamma_{\rho,\epsilon}$ leaves the north and south poles $(0^{(2)},0,\pm1)$ of $\mb S^4/S^1$ on the same side. 
	\end{itemize}
	Since the maps $A$ and $sr^l\in D_n$ on $\mb S^4/S^1$ are orientation-preserving and interchange the north/south poles of $\mb S^4/S^1$, we see from the second bullet that $A$ and $sr^l$ also preserve the orientation of $\Gamma_{\rho,\epsilon}$. 
	Thus, $\Gamma_{\rho,\epsilon}/\langle A\rangle \subset RP^3$ 
    is an orientable closed surface.

	Consider the family $\{\Sigma_t'\}_{t\in (0,1)}$ of $\mb Z_2\times D_n$-invariant surfaces given by 
	\[ \Sigma_t':= \Gamma_{t, \epsilon_0(t) }. \]
	As $t\to 1$, $\Sigma_t'$ tends to the union of circles $\cup_{b=1}^k \Xi_{2b}$. 
	For $t$ near $0$, since $(\mathring{\mb S}^4/S^1, \wti g_{\mb S^4/S^1})$ has positive Ricci curvature (Lemma \ref{Lem: properties of S1-action}), it follows from the catenoid estimates in \cite{ketover2020catenoid} that  one can modify the necks equivariantly so that 
	\[ \mc H^3_{g_{\mb S^4}}(\pi^{-1}(\Sigma_t')) < 2 \mc H^3_{g_{\mb S^4}}(E_0) - Ct^2, \] 
	for some $C>0$. 
	Indeed, 
	the above construction and estimate are similar to \cite{haslhofer2019sphere}*{(10.4)}. 
	As $t\to 0$, $\Sigma_t'$ tends to the punctured sphere $(E_0/S^1)\setminus (\cup_{b=1}^k B_{\tilde\epsilon_0}(\Xi_{2b}))$ with multiplicity two. 
	Hence, we shall modify the surface $\Sigma_t'$ for $t$ near $0$ by opening the necks. 
	
	Specifically, the union of circles $\cup_{b=1}^k \Upsilon_{2b-1}$ separates the equator $E_0/S^1=\{x=0\}\cap \mb S^4/S^1$ into $n=2k$ lunes $L_1,\dots,L_n$ with closures $\{\overline L_l\}_{l=1}^n$ , each of which contains a single point of $\{p_1,\dots,p_n\}=(E_0/S^1)\cap (\cup_{b=1}^k \Xi_{2b})$. 
	Since every $\overline L_l\setminus p_l$ is topologically a punctured disk, we can take a $\mb Z_2\times D_n$-equivariant retraction $R$ on $(E_0/S^1)\setminus\{p_1,\dots, p_n\} =\cup_{l=1}^n (\overline L_l\setminus \{p_l\}) $: 
	\[R: [0,1]\times (\cup_{l=1}^n (\overline L_l\setminus p_l ))\to \cup_{l=1}^n ( \overline L_l\setminus \{p_l\} )\] 
	so that $R(t,\cdot)$ is smooth for $0\leq t<1$, 
	\begin{itemize}
		\item $R(0,\cdot )=id$, and 
		\item $R(1, \overline L_l\setminus \{p_l\} )= \bd L_l \subset \cup_{b=1}^k  \Upsilon_{2b-1}$, for $1\leq l\leq n$,
	\end{itemize}
	which can be further extended $\mb Z_2\times D_n$-equivariantly into a neighborhood of $E_0/S^1\subset \mb S^4/S^1$. 
	Next, for a fixed small $\delta>0$, we can modify $\Sigma_t'$, $t\in (0,1)$, by opening the neck via $R$ to obtain
	\[\Sigma_t:= R(T(t), \Sigma_t')\]
	where $T(t):=\max(1-t/\delta, 0)$. 
	Note that as $t\to 0$ and $t\to 1$, $\Sigma_t$ tends to $\Sigma_0:=\cup_{b=1}^k  \Upsilon_{2b-1}$ and $\Sigma_1:=\cup_{b=1}^k\Xi_{2b}$ respectively. 
	Similar to \cite{haslhofer2019sphere}*{(10.5)}, we still have $\mc H^3_{g_{\mb S^4}}(\pi^{-1}(\Sigma_t)) < 2 \mc H^3_{g_{\mb S^4}}(E_0)$. 
	
	Finally, by smoothing and taking $\wti\Sigma_t:=\pi^{-1}(\Sigma_t)$, we have the desired $\{\wti\Sigma_t\}_{t\in [0,1]}$. 
	One easily sees from the constructions that $\Phi(t):=\llbracket \wti\Sigma_t\rrbracket$ is an Almgren-Pitts $1$-sweepout, and $\wti\Sigma_t=\bd \wti\Omega_t$ for $t\in (0,1)$ and some open $G_n$-set $\wti\Omega_t\subset \mb S^4$. 
\end{proof}

Given an even integer $n=2k\geq 2$ and the $G_n$-sweepout $\{\wti \Sigma_t\}_{t\in [0,1]}$ constructed in Lemma \ref{Lem: Gn-sweepout n-even}, let $\Pi_n$ be the $G_n$-equivariant $([0,1], \{0,1\})$-homotopy class of $\{\wti\Sigma_t\}_{t\in [0,1]}$ in $\mb S^4$. 
\begin{proposition}\label{Prop: new minimal hypersurface of even n}
	Using the above notations, 
	\[ 0< \mf L(\Pi_n) < 2\mc H^3(E_0),\] 
	and $\mf L(\Pi_n)$ is realized by the area of an embedded $G_n$-invariant minimal hypersurface $\Gamma_n\subset \mb S^4$ with multiplicity one such that 
	\[\genus(\Gamma_n/S^1)=n-1,\] 
	and $\Gamma_n=\bd\wti\Omega_n$ for some $G_n$-invariant open set $\wti\Omega_n\subset \mb S^4$. In particular,  $\Gamma_n/{\rm Pin}_-(2)$ is orientable in $RP^3=\mb S^4/{\rm Pin}_-(2)$. 
\end{proposition}

\begin{remark}\label{Rem: valid in symmetric S4 with Ric>0 for n-even}
    By the proof of \cite{wang2026spherical}*{(7.8) and Theorem 7.10} for the multiplicity one, the existence of the above minimal $\Gamma_n$ is indeed valid in any $G_n$-invariant Riemannian $S^4$ with positive Ricci curvature.  
\end{remark}

\begin{proof}
	Since $\{\wti\Sigma_t\}$ generates a $1$-sweepout in the sense of Almgren-Pitts, we see $\mf L(\Pi_n)>0$. 
	By the equivariant min-max theorem (\cite{wang2026spherical}*{Theorem 1.5}) and the Frankel property in the round $\mb S^4$, we have a connected minimal $G_n$-hypersurface $\Gamma_n$ with $m_n\in\mb N$ so that $m_n\cdot \mc H^3(\Gamma_n)=\mf L(\Pi_n)$. 
	
	Note that the totally geodesic equator $E_0$ has the least area among all embedded closed minimal hypersurfaces in $\mb S^4$. 
	Hence, $\mc H^3_{g_{\mb S^4}}(\Gamma_n)
	\geq \mc H^3_{g_{\mb S^4}}(E_0)$, and thus $m_n=1$ by Lemma \ref{Lem: Gn-sweepout n-even} (iv). 
	One can also use the arguments for \cite{wang2026spherical}*{(7.8) and Theorem 7.10} to show $m_n=1$, where $\Ric_{\mb S^4}>0$ is essentially used. 
	
	Additionally, in the $G_n$-sweepout $\{\wti\Sigma_t\}_{t\in [0,1]}$, it follows from Lemma \ref{Lem: Gn-sweepout n-even}(iii) that each $\wti\Sigma_t/S^1$, $t\in (0,1)$, separates $\mb S^4/S^1$ into two regions that are both invariant under the $\mb Z_2\times D_n$-action defined in \eqref{Eq: antipodal} and \eqref{Eq: Dn reflectional action on S4/S1}. 
	Hence, combining $m_n=1$ with the proof of \cite{zhou2015positiveRic}*{Proposition 6.1}, we conclude that $\Gamma_n/S^1=\bd\Omega$ for some $\mb Z_2\times D_n$-set $\Omega\subset \mb S^4/S^1$, which gives the desired $\wti\Omega_n:=\pi^{-1}(\Omega)$. 
    In particular, $\Gamma_n/{\rm Pin}_-(2) = \bd(\Omega /\mb Z_2)$ must be orientable in $RP^3=\mb S^4/{\rm Pin}_-(2)$.

	Next, by the embeddedness of $\Gamma_n$, we see that $\Gamma_n/S^1\subset \mathring{\mb S}^4/S^1$ does not contain the north/south poles of $\mb S^4/S^1$ (Lemma \ref{Lem: properties of S1-action}(iii)). 
	Hence, by Lemma \ref{Lem: Gn-sweepout n-even}(iii) and the weighted genus upper bound for the $\mb Z_2\times D_n$-equivariant min-max in $(\mathring{\mb S}^4/S^1, \wti g_{\mb S^4/S^1})$ (cf. \cite{ketover2016equivariant}*{Theorem 1.3}), we have 
	\begin{align}\label{Eq: genus upper bound for even n in S4/S1}
		\genus(\Gamma_n/S^1) \leq n-1=2k-1.
	\end{align}
	It is now sufficient to show that the equality holds in the above formula.

	To begin with, consider the $\mb Z_2\times D_n$-invariant circle $\Xi_0\subset \mb S^4/S^1$ defined in \eqref{Eq: axes on S3} corresponding to the $z$-axis in $\mb R^3$ via the stereographic projection. 
	Then, by \cite{wang2026spherical}*{Lemma 2.8} and the facts that $\Gamma_n/S^1\subset \mathring{\mb S}^4/S^1$ while $\Xi_0\not\subset \mathring{\mb S}^4/S^1$, we conclude that $\Gamma_n/S^1$ intersects $\Xi_0\cap (\mathring{\mb S}^4/S^1)$ orthogonally at $m\in \mb N$ points. 
	 For any $p\in \Xi_0\cap (\mathring{\mb S}^4/S^1)$, since
	 \begin{align}\label{Eq: number of points in orbits on Xi_0}
	 	\mbox{$\#[(\mb Z_2\times D_n)\cdot p]=2$ if $p\in (E_0/S^1)\cap \Xi_0$, $\quad$ and $\quad$ otherwise $\#[(\mb Z_2\times D_n)\cdot p]=4$},  
	 \end{align}
	 we have $m=4j$ or $m=2+4j$ for some $j\in \{0,1,\dots\}$.
	 Additionally, denote by $\Omega_\pm$ the two $\mb Z_2\times D_n$-invariant components of $(\mb S^4/S^1) \setminus (\Gamma_n/S^1)$. 
	 Then, $\Xi_0$ is separated by $\Gamma_n/S^1 \cap \Xi_0=\{p_i\}_{i=1}^m$ into $m$ arcs $\{\gamma_i^\pm\}_{i=1}^{m/2}$ so that $\Omega_\pm\cap \Xi_0=\cup_{i} \gamma_i^\pm$, and any two adjacent arcs have opposite superscripts, i.e. $\gamma_i^\pm$ is adjacent to some $\gamma_{i'}^\mp$. 
     \begin{figure}[h]
        \centering
        \includegraphics[height=1.2in]{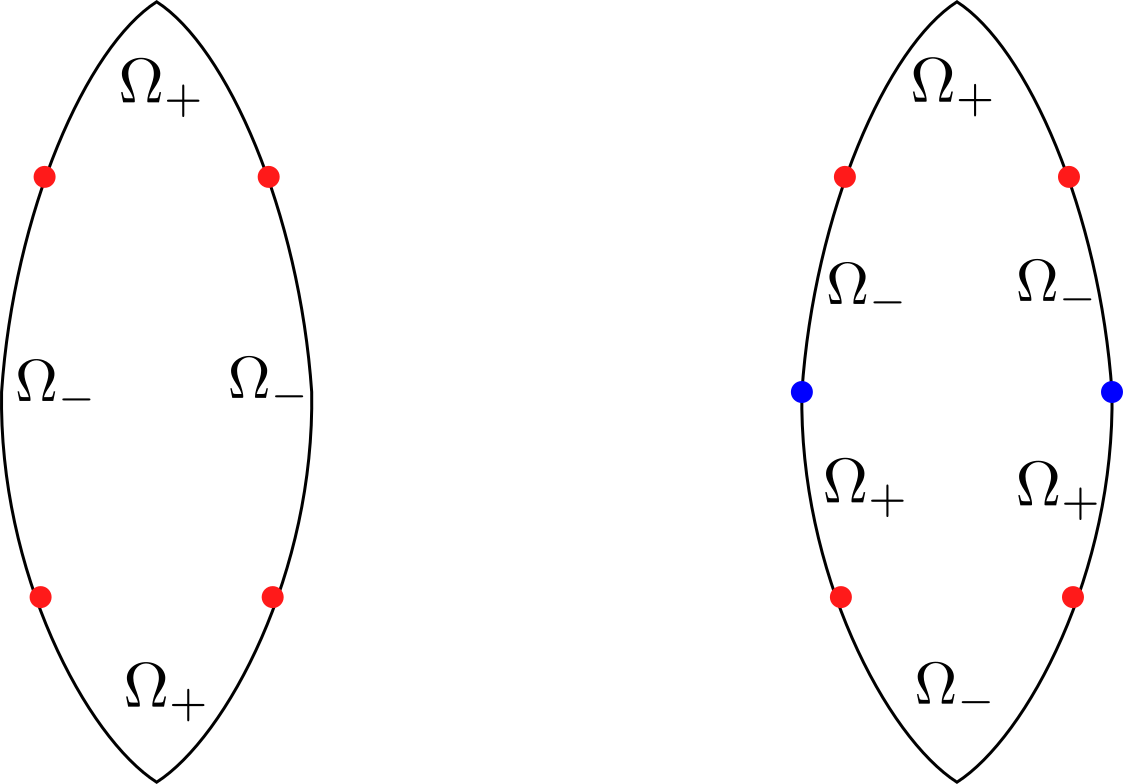}
        \caption{Non-separation of poles when $m=4j$; separating when $m=2+4j$.}
        \label{Fig: intersection points}
    \end{figure}
	 Note that the north and south poles of $\mb S^4/S^1$ lie in the same component of $(\mb S^4/S^1) \setminus (\Gamma_n/S^1)$. Therefore, we conclude that (Figure \ref{Fig: intersection points})
	 \begin{align}\label{Eq: intersection with C0 for n even}
	 	m=4j, \qquad \mbox{for some $j\in \{0,1,2,\dots\}$.}
	 \end{align}
	 
	 Then, consider the $\mb Z_n<D_n$ symmetries on $\mb S^4/S^1$ generated by $r\in D_n$ in \eqref{Eq: Dn reflectional action on S4/S1}. 
  Note that the $\mb Z_n$-action on $\mb S^4/S^1=S^3$ is orientation-preserving, and $\Omega_\pm$ are both $\mb Z_n$-invariant. 
	 Hence, the $\mb Z_n$-action on $\Gamma_n/S^1$ is orientation-preserving, and $(\Gamma_n/S^1)/\mb Z_n$ is orientable without boundary. 
	 In particular, the quotient map $\Gamma_n/S^1\to (\Gamma_n/S^1)/\mb Z_n$ is an $n$-degree covering with $4j$ branch points $\{p_i\}_{i=1}^{4j}$ of ramification index $n=|\mb Z_n|$. 
	 By the Riemann-Hurwitz formula (cf. \cite{ketover2016equivariant}*{(6.1)}\cite{carlotto2022free}*{Appendix B}), there exist integers $\gamma,j\in \{0,1,\dots\}$ so that
	 \[ 2-2\genus(\Gamma_n/S^1 ) = \chi(\Gamma_n/S^1) = n\cdot \chi ((\Gamma_n/S^1)/\mb Z_n) - \sum_{i=1}^{4j} (n-1)= n(2-2\gamma)-4j(n-1) .\]
  Therefore, 
	 \begin{align}\label{Eq: R-H formula of intersection 4j}
	 	\genus(\Gamma_n/S^1 )=n\gamma   + (2j-1)(n-1) . 
	 \end{align}
	For any $t\in (0,1)$, the $G_n$-hypersurface $\wti\Sigma_t$ in Lemma \ref{Lem: Gn-sweepout n-even} has $\genus(\wti\Sigma_t/S^1)=n-1$ and $\#((\wti\Sigma_t/S^1) \cap \Xi_0)=4$, which implies that $\genus((\wti\Sigma_t/S^1)/\mb Z_n) = 0$ by applying \eqref{Eq: R-H formula of intersection 4j} to $\wti\Sigma_t/S^1$. 
	By the proof of \cite{wang2026spherical}*{Theorem 6.6} and \cite{ketover2016equivariant}*{(1.3)}, $\gamma=\genus((\Gamma_n/S^1 )/\mb Z_n) =0$ and 
	\[ \genus(\Gamma_n/S^1 )=(2j-1)(n-1) .\] 
	 %
	 %
	 %
	 %
	 If $j\geq 2$, then $\genus( \Gamma_n/S^1)\geq 3n-3 > n-1$ as $n=2k\geq 2$, which contradicts \eqref{Eq: genus upper bound for even n in S4/S1}. 
	 If $j=0$, then $\genus( \Gamma_n/S^1)<0$, which is impossible. 
	 Hence, $j=1$, and $\genus( \Gamma_n/S^1)=n-1$. 	
\end{proof}

The following theorem gives a new class of minimal hypersurfaces in $\mb S^4$ whose topological types are $S^1\times \Sigma_{2k-1}$, where $\Sigma_{2k-1}$ is a closed surface of odd genus $2k-1$. 

\begin{theorem}\label{Thm: new minimal hypersurfaces of even n}
	For any even integer $n=2k\geq 2$, there exists an embedded $G_n$-invariant minimal hypersurface $\Gamma_n\subset \mb S^4$ that is a trivial principal $S^1$-bundle over an orientable closed surface $\Sigma_{2k-1}:=\Gamma_n/S^1$ of genus $n-1=2k-1$, namely,
	\[\Gamma_n\cong S^1\times \Sigma_{2k-1}.\]
	In particular, $\Gamma_n$, $n\geq 4$, has $\pi_2(\Gamma_n)=0$, and is not diffeomorphic to any finite connected sum of copies of $S^1\times S^2$, which gives a new topological type for minimal hypersurfaces in $\mb S^4$. 
\end{theorem}
\begin{remark}
	For $n=2$, the above theorem gives a minimal $T^3$ in $\mb S^4$ as in \cite{wang2026spherical}*{Theorem 1.2}. 
\end{remark}
\begin{proof}
	For any even integer $n=2k\geq 2$, it follows from Proposition \ref{Prop: new minimal hypersurface of even n} that there is an embedded $G_n$-invariant minimal hypersurface $\Gamma_n$ with $\genus(\Gamma_n/S^1)=n-1=2k-1$. 
	Since the $S^1$-action on $\Gamma_n$ is free, we have a principal $S^1$-bundle $\Gamma_n\to\Gamma_n/S^1$. 
	
	\begin{claim}\label{Claim: Euler class}
		The principal $S^1$-bundle $\Gamma_n\to\Gamma_n/S^1$ is a trivial bundle. 
	\end{claim}
	\begin{proof}
		Note that the suspended Hopf action is free on $\mathring{\mb S}^4$ so that 
		\[\mathring{\mb S}^4/S^1\cong S^2\times (-1,1),\] 
		and the Hopf fibration $E_0\to E_0/S^1$ of the $S^1$-invariant equator $E_0\subset \mathring{\mb S}^4$ has the Euler class $e(E_0)$ as a generator of $H^2(E_0/S^1;\mb Z) \cong \mb Z$ (\cite{hatcher2002AT}*{p. 438}). 
		Hence, the ambient bundle $\mathring{\mb S}^4 \to \mathring{\mb S}^4/S^1$ also has its Euler class $e(\mathring{\mb S}^4)$ as a generator of $H^2(\mathring{\mb S}^4/S^1;\mb Z)\cong H^2(E_0/S^1; \mb Z)\cong \mb Z$. 
		
		Let $i$ be the embedding of $\Gamma_n/S^1$ into $\mathring{\mb S}^4/S^1$.
		Then, $e(\Gamma_n):=i^*e(\mathring{\mb S}^4)\in H^2(\Gamma_n/S^1;\mb Z)$ is the Euler class of the restricted bundle $\Gamma_n\to \Gamma_n/S^1$. 
		Recall that $\Gamma_n/{\rm Pin}_-(2)$ is orientable, and the north/south poles of $\mb S^4/S^1$ lie on the same side of $\Gamma_n/S^1$. 
		Hence, $\Gamma_n/S^1=\bd\Omega$ for an open set $\Omega\subset \mathring{\mb S}^4/S^1$, which implies that 
		\[e(\Gamma_n)([\Gamma_n/S^1])=i^*e(\mathring{\mb S}^4)([\Gamma_n/S^1])=e(\mathring{\mb S}^4)(i_*[\Gamma_n/S^1])=e(\mathring{\mb S}^4)([\bd\Omega])=0.\] 
		Therefore, the principal $S^1$-bundle $\Gamma_n\to \Gamma_n/S^1$ has trivial Euler class $e(\Gamma_n)=0$, and thus 
		\[\Gamma_n\cong S^1\times (\Gamma_n/S^1)\]
		is a trivial principal $S^1$-bundle. 
	\end{proof}
	
	Since $\genus(\Gamma_n/S^1)=n-1>0$ and $\Gamma_n/S^1$ is orientable, we see that $\Gamma_n/S^1$ is aspherical, and thus $\pi_2(\Gamma_n/S^1)=0$. 
	Combining this with the principal $S^1$-bundle structure of $\Gamma_n$ and the fact that $\pi_2(S^1)=0$, we conclude from the homotopy exact sequence (\cite{hatcher2002AT}*{Theorem 4.41}) that $\pi_2(\Gamma_n)=0$. 
	For any $m\geq 1$, since $\pi_2(\#_m (S^2\times S^1))\neq 0$, $\Gamma_n$ is not diffeomorphic to $\#_m (S^2\times S^1)$. 
	Therefore, $\{\Gamma_n\}_{n=2k}$ gives a family of new topological types for minimal hypersurfaces in $\mb S^4$. 
\end{proof}

\subsection{\texorpdfstring{New minimal $G_n$-hypersurfaces in $\mb S^4$ for odd $n$}{New minimal Gn-hypersurfaces in S4 for odd n}}

Fix any odd integer 
\[n=2k+1\geq 3.\]
In this subsection, we will first construct a $G_n$-equivariant sweepout $\{\wti \Sigma_t \}_{t\in [0,1]}$ in $\mb S^4$ so that $\wti\Sigma_t/S^1$ has genus $\leq 2n-2=4k$. 
Then, we use the equivariant min-max theory to obtain embedded minimal non-trivial principal $S^1$-bundles in $\mb S^4$. 

\begin{lemma}\label{Lem: Gn sweepout}
	For any odd integer $n=2k+1$, there is a $G_n$-sweepout $\{\wti \Sigma_t \}_{t\in [0,1]}$ in $\mb S^4$ so that 
	\begin{itemize}
		\item[(i)] $\wti\Sigma_0=E_0$ is the unique $G_n$-invariant equator in $\mb S^4$;
		\item[(ii)] $E_0\subset \wti \Sigma_1\subset E_0\cup \pi^{-1}(\cup_{l=1}^n\widehat \Xi_l)$, where $\{\widehat \Xi_l\}_{l=1}^n$ are the circles on $\mb S^4/S^1$ corresponding to the axes $\{\hat \xi_l\}_{l=1}^n$ in $\mb R^3$ via the stereographic projection (see \eqref{Eq: axes on S3});
		\item[(iii)] for $t\in (0,1)$, $\wti\Sigma_t\subset\mb S^4$ is an embedded $G_n$-hypersurface separating $\mb S^4$ into two $S^1$-regions $\wti\Omega_t^\pm$ so that $\sigma\cdot \wti\Omega_t^\pm = \wti\Omega_t^\mp=\hat s\cdot \wti\Omega_t^\pm$(cf. \eqref{Eq: lift antipodal map}\eqref{Eq: Dn reflectional action on S4}), $\wti\Sigma_t/{\rm Pin}_-(2)$ is a non-orientable surface in $RP^3=\mb S^4/{\rm Pin}_-(2)$, and 
		\[\genus(\wti\Sigma_t/S^1)=2n-2=4k;\]
		\item[(iv)] $\sup_{t\in [0,1]} \mc H^3_{g_{\mb S^4}}(\wti\Sigma_t) < 3 \mc H^3_{g_{\mb S^4}}(E_0) = 6\pi ^2$. 
	\end{itemize}
	Additionally, $\Phi(t):=\llbracket \wti\Sigma_t\rrbracket $, $t\in [0,1]$, forms an Almgren-Pitts $1$-sweepout. 
\end{lemma}


\begin{proof}
	It is sufficient to construct the corresponding $\mb Z_2\times D_n$-equivariant sweepout $\{\Sigma_t\}_{t\in [0,1]}$ in $\mb S^4/S^1$ under the parameterization \eqref{Eq: parameterize S4/S1}. 
	
	For any small $\epsilon>0$ and $l=1,\dots, n$, let $B_\epsilon(\widehat \Xi_l^\pm)$ be an $\epsilon$-neighborhood of the longitude $\widehat\Xi_l^\pm$ (see \eqref{Eq: half axis hat xi on S3}) in $\mb S^4/S^1$. 
	Since the longitude $\widehat\Xi_l^\pm$ is transversal to each slice
	\[S_\rho:= \{x=\rho\}\cap \mb S^4/S^1, \qquad \rho\in (-1,1), \] 
	$\bd B_\epsilon(\widehat\Xi_l^\pm)$ is transversal to $S_\rho$ for $\epsilon>0$ small enough. 
	Additionally, there exists a continuous function $\epsilon_0: [0,1] \to [0, \infty)$ so that
	\begin{itemize}
		\item $\lim_{|\rho|\to  1}\epsilon_0(|\rho|) = 0$, $\lim_{|\rho|\to 0} \epsilon_0(|\rho|) = \tilde\epsilon_0>0$; 
		\item for any $\epsilon\in (0, 2\epsilon_0(|\rho|))$, $B_\epsilon(\widehat\Xi_l^\pm)\cap S_\rho$ is a disk and $S_\rho\cap B_\epsilon(\widehat\Xi_l^\kappa)\cap B_\epsilon(\widehat\Xi_{l'}^{\kappa'})=\emptyset$, where $l\neq l'\in \{1, \dots, n\}$ and $\kappa,\kappa'\in \{\pm\}$.
	\end{itemize}
	Therefore, using the above notations and the parameterization \eqref{Eq: parameterize S4/S1}, we can define the following surfaces in $\mb S^4/S^1$ for $\rho \in (0,1)$ and $\epsilon\in (0, 2\epsilon_0(\rho))$: 
	\begin{align*}
		S_{0,\epsilon} &:= S_0\setminus  \cup_{l=1}^{n} B_\epsilon(\widehat \Xi_{l}) = (E_0/S^1)\setminus \cup_{l=1}^{n}\cup_{\kappa\in \{\pm\}} B_\epsilon(\widehat \Xi_{l}^\kappa),
		\\
		S_{\rho,\epsilon}^+ &:= S_\rho\setminus \cup_{l=1}^nB_\epsilon(\widehat \Xi_l^{{\rm sgn}((-1)^{l})} ),
		\\
		S_{\rho,\epsilon}^- &:= S_{-\rho}\setminus \cup_{l=1}^nB_\epsilon(\widehat \Xi_l^{{\rm sgn}((-1)^{l+1})} ),
		\\
		C_{\rho,\epsilon}^+ &:=  \{(w,t,x)\in \cup_{l=1}^n\bd B_\epsilon(\widehat \Xi_l^{{\rm sgn}((-1)^{l})}): 0<x<\rho \},
		\\
		C_{\rho,\epsilon}^- &:=  \{(w,t,x)\in \cup_{l=1}^n\bd B_\epsilon(\widehat \Xi_l^{{\rm sgn}((-1)^{l+1})}): -\rho <x<0 \},
	\end{align*}
	where ${\rm sgn}(\pm 1) := \pm$ is the sign of $\pm 1$. 
	Then, 
	\[\Gamma_{\rho,\epsilon}:= S_{\rho,\epsilon}^+\cup C_{\rho,\epsilon}^+\cup S_{0,\epsilon} \cup C_{\rho,\epsilon}^-\cup S_{\rho,\epsilon}^-\]
	is a connected surface in $\mb S^4/S^1$ formed by three spheres $S_\rho \cup S_0 \cup S_{-\rho}$ that are connected via $2n$ necks along $\cup_{l=1}^n\widehat \Xi_{l}^{\pm}$ (see the right panel of Figure \ref{Fig: sweepout n-odd} for the top view under the projection $P$ \eqref{Eq: stereographic projection}). 
    \begin{figure}[h]
    \centering
    \includegraphics[height=1.75in]{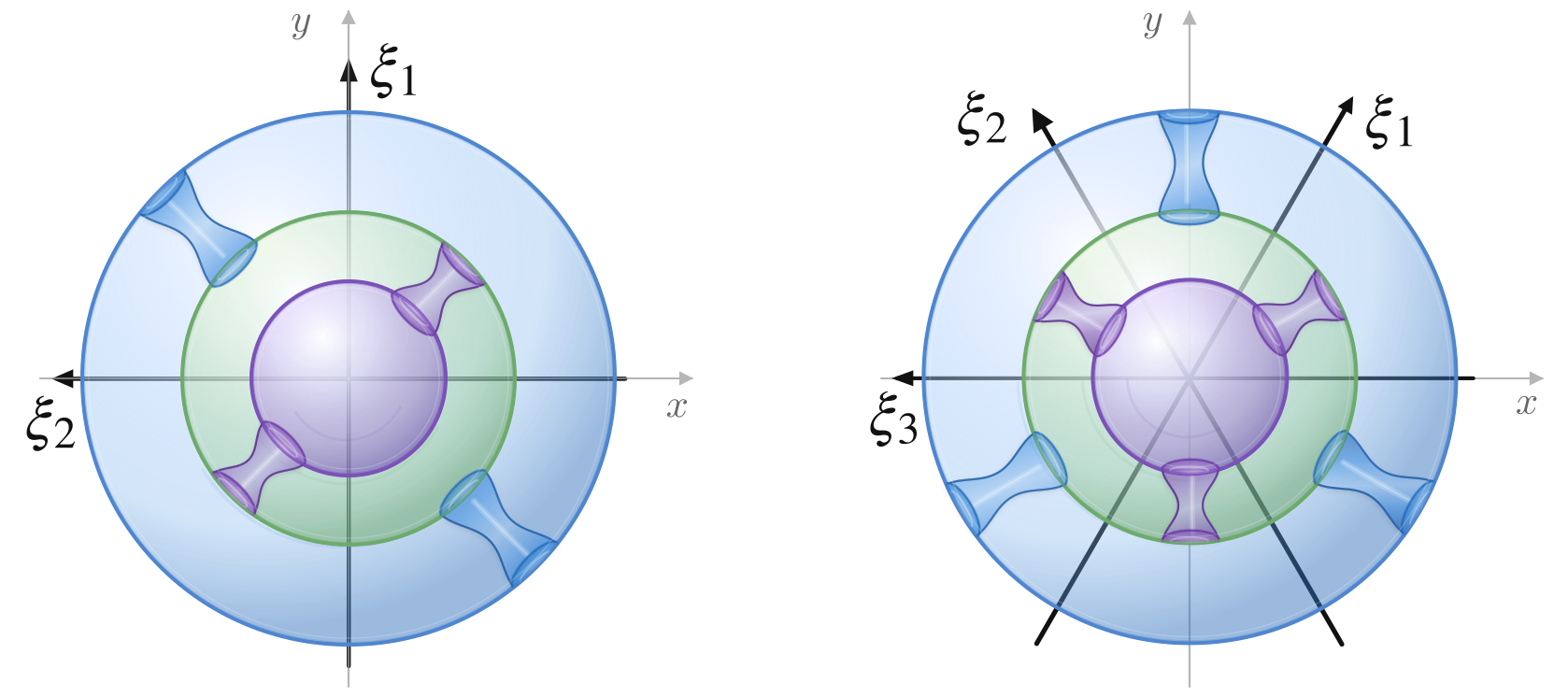}
    \caption{$D_n$-invariant $\widehat\Gamma_{\rho,\epsilon}$ ($n=2$), and $\mb Z_2\times D_n$-invariant $\Gamma_{\rho,\epsilon}$ ($n=3$).}
    \label{Fig: sweepout n-odd}
    \end{figure}

	Recall that $n=2k+1$ is odd. 
	One can verify that the unions of arcs 
	\[ \gamma^+:=\cup_{l=1}^n \{(w,t,x)\in  \widehat\Xi_{l}^{{\rm sgn}((-1)^l)} : x>0  \}
		\quad {\rm and}\quad 
	 	\gamma^-:=\cup_{l=1}^n \{(w,t,x)\in  \widehat\Xi_{l}^{{\rm sgn}((-1)^{l+1})} : x<0  \}\]
	are both invariant under the action of $\mb Z_n=\langle r\rangle$ in \eqref{Eq: Dn reflectional action on S4/S1}
	so that $A\cdot \gamma^\pm= \gamma^\mp = s\cdot \gamma^\pm$ for the antipodal map $A$ in \eqref{Eq: antipodal} and the action of $s\in D_n$ in \eqref{Eq: Dn reflectional action on S4/S1}.
	Since $C_{\rho,\epsilon}^\pm$ are unions of necks centered at $\gamma^\pm$ respectively, 
	we see that $C_{\rho,\epsilon}^+\cup C_{\rho,\epsilon}^-$ is $\mb Z_2\times D_n$-invariant, which implies that $\Gamma_{\rho,\epsilon}$ is $\mb Z_2\times D_n$-invariant. 
	Additionally, by the constructions, we have 
	\begin{itemize}
		\item $\genus(\Gamma_{\rho,\epsilon})=2n-2=4k$;
		\item $\Gamma_{\rho,\epsilon}$ separates the north and south poles $(0^{(2)},0,\pm 1)$ of $\mb S^4/S^1$ into different regions. 
	\end{itemize}
	Since the maps $A$ and $s\in D_n$ are orientation-preserving on $\mb S^4/S^1$, we see from the second bullet that $A$ and $s$ reverse the orientation of $\Gamma_{\rho,\epsilon}$, which implies that $\Gamma_{\rho,\epsilon}/\langle A\rangle \subset RP^3$ is $1$-sided.

	Consider the family $\{\Sigma_t'\}_{t\in (0,1)}$ of $\mb Z_2\times D_n$-invariant surfaces given by 
	\[ \Sigma_t':= \Gamma_{t,\epsilon_0(t)}. \]
	As $t\to 1$, $\Sigma_t'$ tends to the union of $E_0/S^1=S_0$ and the arcs $\gamma^+\cup\gamma^-$. 
	For $t$ near $0$, since $(\mathring{\mb S}^4/S^1, \wti g_{\mb S^4/S^1})$ has positive Ricci curvature (Lemma \ref{Lem: properties of S1-action}), one can apply the catenoid estimates in \cite{ketover2020catenoid} to modify the necks $C_{\rho,\epsilon}^+\cup C_{\rho,\epsilon}^-$ equivariantly so that 
	\[ \mc H^3_{g_{\mb S^4}}(\pi^{-1}(\Sigma_t')) < 3 \mc H^3_{g_{\mb S^4}}(E_0) - Ct^2, \] 
	for some $C>0$ (see \cite{haslhofer2019sphere}*{(10.4)} and \cite{wang2026spherical}*{(7.7)}). 
	As $t\to 0$, $|\Sigma_t'|$ tends to $|E_0/S^1|+2|S_{0,\tilde\epsilon_0}| $. 
	Hence, we shall modify the surface $\Sigma_t'$ for $t$ near $0$ by opening the necks so that $\Sigma_t'\to E_0/S^1$ as $t\to 0$. 
	
	To be exact, the union of circles $\cup_{l=1}^n \Upsilon_{l}$ separates the equator $E_0/S^1=\{x=0\}\cap \mb S^4/S^1$ into $2n=4k+2$ lunes $L_1^\pm,\dots,L_{n}^\pm$ with closures denoted by $\{\overline L_l^\pm\}_{l=1}^n$ so that $L_l^\pm$ contains $p_l^\pm:=(\pm \frac{(-1)^{l}}{2}e^{(\frac{l\pi}{n} - \frac{\pi}{2n} ) i}, 0,0) =  (E_0/S^1)\cap  \widehat \Xi_{l}^{{\rm sgn}(\pm (-1)^l)}$ for $1\leq l \leq n$ respectively.
	Since every $\overline L_l^\pm\setminus \{p_l^\pm\}$ is topologically a punctured disk that retracts onto its boundary, we can take two $\mb Z_n=\langle r\rangle$ equivariant retractions $R_+$ and $R_-$ on $U_\pm:=\cup_{l=1}^{n} (\overline L_l^\pm\setminus \{p_l^\pm\} )$:  
	\[R_\pm: [0,1]\times U_\pm\to U_\pm\] 
	so that $R_\pm(t,\cdot)$ is smooth for $0\leq t<1$, 
	\begin{itemize}
		\item $R_\pm (0,\cdot )=id$,  
		\item $R_\pm (1, \cup_{l=1}^{n} (L_l^\pm\setminus p_l^\pm ))= \cup_{l=1}^{n}(\bd L_l^\pm) = \cup_{l=1}^n  \Upsilon_{l}$, and 
		\item $A\cdot R_\pm(t, p)= R_\mp(t, A\cdot p)$ and $s\cdot R_\pm(t, p)= R_\mp(t,s\cdot p)$, for any $p\in U_\pm$,
	\end{itemize}
	which can be joined together as a retraction $R$ from $(E_0/S^1)\setminus (\cup_{l=1}^n \widehat \Xi_l)$ to $\cup_{l=1}^n  \Upsilon_{l}$. 
	Moreover, by suspending the retractions along the $x$-coordinate in $\mb S^4/S^1$, one can $\mb Z_n$-equivariantly extend $R_\pm$ into $\Sigma(U_\pm):=\left \{(w,t,x)\in \mb S^4/S^1: \frac{(w,t,0)}{\sqrt{1-x^2}}\in U_\pm, -1<x<1 \right\}$ so that the statement in the third bullet remains valid. 
	Therefore, for a fixed small $\delta>0$, after taking $R_\pm\equiv id $ on $ (\mb S^4/S^1)\setminus \Sigma(U_\pm)$, we  modify $\Sigma_t'$, $t\in (0,1)$, by opening the necks via $R_\pm$ and $R$ to obtain
	\[\Sigma_t:= R_+(T(t), S_{t,\epsilon_0(t)}^+\cup C_{t,\epsilon_0(t)}^+) \cup R(T(t), S_{0,\epsilon_0(t)}) \cup R_-(T(t), S_{t,\epsilon_0(t)}^-\cup C_{t,\epsilon_0(t)}^-)\]
	where $T(t):=\max(1-t/\delta, 0)$. 
	Note that as $t\to 0$ and $t\to 1$, $\Sigma_t$ tends to $\Sigma_0:=E_0/S^1$ and $\Sigma_1:=E_0/S^1\cup \gamma^+\cup \gamma^-\subset E_0/S^1\cup_{l=1}^n\widehat\Xi_l$ respectively. 
	Similar to \cite{haslhofer2019sphere}*{(10.5)}, we still have $\mc H^3_{g_{\mb S^4}}(\pi^{-1}(\Sigma_t)) < 3 \mc H^3_{g_{\mb S^4}}(E_0)$. 
	
	Finally, by smoothing and taking $\wti\Sigma_t:=\pi^{-1}(\Sigma_t)$, we have the desired $\{\wti\Sigma_t\}_{t\in [0,1]}$. 
	One easily sees from the constructions that $\Phi(t):=\llbracket \wti\Sigma_t\rrbracket$ is an Almgren-Pitts $1$-sweepout,  and $\wti\Sigma_t$, $t\in (0,1)$, separates $\mb S^4$ into two $S^1$-regions $\wti\Omega_t^\pm$ that are permuted under the action of $\sigma$ and $\hat s$. 
\end{proof}

Given an odd integer $n=2k+1\geq 3$ and the $G_n$-sweepout $\{\wti \Sigma_t\}_{t\in [0,1]}$ constructed in Lemma \ref{Lem: Gn sweepout}, let $\Pi_n$ be the $G_n$-equivariant $([0,1], \{0,1\})$-homotopy class of $\{\wti\Sigma_t\}_{t\in [0,1]}$ in $\mb S^4$. 
\begin{proposition}\label{Prop: new minimal hypersurface}
	Using the above notations, 
	\[\mc H^3_{g_{\mb S^4}}(E_0)< \mf L(\Pi_n) < 3\mc H^3_{g_{\mb S^4}}(E_0),\] 
	and $\mf L(\Pi_n)$ is realized by the area of a $G_n$-invariant minimal hypersurface $\Gamma_n\subset \mb S^4$ with multiplicity one such that 
	\[\genus(\Gamma_n/S^1)\in \{0, 2n-2\},\] 
	and $\Gamma_n$ separates $\mb S^4$ into two $S^1$-invariant open sets $\wti\Omega_n^\pm$ with $\sigma\cdot  \wti \Omega_n^\pm = \wti \Omega_n^\mp = \hat s\cdot \wti \Omega_n^\pm$ (see \eqref{Eq: lift antipodal map}\eqref{Eq: Dn reflectional action on S4}). 
    In particular, $\Gamma_n/{\rm Pin}_-(2)$ is non-orientable in $RP^3=\mb S^4/{\rm Pin}_-(2)$. 
\end{proposition}

\begin{proof}
	Since $\{\wti\Sigma_t\}_{t\in [0,1]}$ generates a $1$-sweepout in the sense of Almgren-Pitts, we see $\mf L(\Pi_n)>0$. 

    Recall that the totally geodesic equator has the least area among all minimal hypersurfaces in $\mb S^4$, and $E_0$ is the unique $S^1$-invariant equator (Lemma \ref{Lem: properties of S1-action}). 
	Hence, combining these with \cite{li2026infinite}*{Lemma 2.4}, we know $\llbracket E_0\rrbracket$ is the unique mass minimizer in 
    $\alpha_0:=\{\llbracket E_0\rrbracket + \bd \Omega: \Omega\in \mc C^{{\rm Pin}_-(2)}(\mb S^4) \}$.
	In particular, $E_0$ is the area minimizer among all the embedded ${\rm Pin}_-(2)$-hypersurfaces $\Sigma$ with $\Sigma/{\rm Pin}_-(2)$ non-orientable in $RP^3=\mb S^4/{\rm Pin}_-(2)$. 
	Hence, by Lemma \ref{Lem: Gn sweepout}(iii), 
	\[\mf L(\Pi_n)\geq \mc H^3_{g_{\mb S^4}}(E_0).\] 
	If $\mf L(\Pi_n)=\mc H^3_{g_{\mb S^4}}(E_0)$, 
	then for any minimizing sequence $\{\{\wti\Sigma^i_{t}\}_{t\in [0,1]}\}_{i\in\mb N}$ and any $\{t_i\}_{i\in\mb N}\subset [0,1]$, $\{\llbracket \wti \Sigma^{i}_{t_i} \rrbracket\}_{i\in\mb N}$ is a mass minimizing sequence in $\alpha_0$, and thus $\llbracket \wti \Sigma^{i}_{t_i} \rrbracket\to \llbracket E_0 \rrbracket$ up to a subsequence as $i\to\infty$. 
	Thus, for any $\epsilon>0$, there exists $i_0\in\mb N$ so that $\mc F(\llbracket \wti \Sigma^{i}_{t}\rrbracket, \llbracket E_0\rrbracket )<\epsilon$ for all $i\geq i_0$ and $t\in [0,1]$. 
	However, by Lemma \ref{Lem: Gn sweepout}, $\Phi_{i_0}(t):=\llbracket \wti\Sigma^{i_0}_t \rrbracket$, ($t\in [0,1]$), is a non-trivial Almgren-Pitts $1$-sweepout, which cannot be contained in an $\epsilon$-neighborhood of $\llbracket E_0\rrbracket$ for some $\epsilon>0$ small enough (\cite{marques2017existence}*{Corollary 3.4}). 
	Therefore, together with Lemma \ref{Lem: Gn sweepout}(iv), we conclude that 
    \begin{align}\label{Eq: width estimates for odd n}
        \mc H^3_{g_{\mb S^4}}(E_0)< \mf L(\Pi_n) \leq \sup_{t\in [0,1]}\mc H^3_{g_{\mb S^4}}(\wti\Sigma_t) < 3\mc H^3_{g_{\mb S^4}}(E_0).
    \end{align}
	In particular, $\Pi_n$ satisfies the width assumption in \cite{wang2026spherical}*{(6.1)}.
 
	Note that the $G_n$-sweepout $\{\wti\Sigma_t\}_{t\in [0,1]}$ is not of locally $G_n$-boundary-type (cf. \cite{wang2026spherical}*{Definition 2.12}) since every $\wti\Sigma_t/S^1$ contains $\cup_{l=1}^n \Upsilon_l$. 
	Nevertheless, we can combine the equivariant min-max theory in \cite{wang2026spherical} and \cite{ketover2016free} to obtain a connected embedded minimal hypersurface $\Gamma_n$ with  
	\[m_n\cdot \mc H^3_{g_{\mb S^4}}(\Gamma_n)=\mf L(\Pi_n)\qquad \mbox{for some $m_n\in\mb N$}.\] 
	Indeed, since the locally $G$-boundary-type assumption is only used in the regularity theory \cite{wang2026spherical}*{\S 6.3}, we only need to show the regularity of the min-max $G_n$-varifold $V_n\in  \mc V^{G_n}_3(\mb S^4)$ given by \cite{wang2026spherical}*{Theorem 6.12} associated with $\Pi_n$. 
	Note that under the $\mb Z_2\times D_n$-action with odd $n$, the singular loci (i.e. the union of points with non-trivial isotropy group) on $\mb S^4/S^1$ are given by $\{\Xi_0,\Upsilon_1,\dots\Upsilon_n,\widehat\Xi_1,\dots,\widehat\Xi_n\}$. 
    Hence, since every $\wti\Sigma_t/S^1$, $t\in (0,1)$, is orthogonal to $\Xi_0\cup(\cup_{l=1}^n\widehat\Xi_l)$, we conclude from \cite{wang2026spherical}*{Remark 2.13, Proposition 2.10(3)} that $\wti\Sigma_t \llcorner (\mb S^4\setminus B_\epsilon(E_0))$ is of locally $G_n$-boundary-type for all $t\in (0,1)$ and some small $\epsilon>0$. 
	Thus, the regularity of $V_n\llcorner \mb S^4\setminus B_\epsilon(E_0)$ follows from the proof in \cite{wang2026spherical}*{\S 6.3}. 
	Additionally, using \eqref{Eq: area in S4/S1}, one can apply the regularity theory of \cite{ketover2016free}*{Theorem 3.2, \S 7.2} for the $\mb Z_2\times D_n$-equivariant min-max in $(B_\epsilon(E_0)/S^1, \wti g_{\mb S^4/S^1} )$, which further provides the regularity of $V_n \llcorner B_\epsilon(E_0)$. 
	Finally, the connectedness of $\Gamma_n=\spt(\|V_n\|)$ follows from the Frankel property in the round $\mb S^4$.

	Additionally, in the $G_n$-sweepout $\{\wti\Sigma_t\}_{t\in [0,1]}$, each $\wti\Sigma_t/S^1$ separates $\mb S^4/S^1$ into two regions that are interchanged under the antipodal map $A$ and the action of $s\in D_n$. 
	Hence, the proof of \cite{li2025RP2}*{Proposition 5.4} also applies to show that 
    \begin{itemize}
        \item[(1)] $\Gamma_n/{\rm Pin}_-(2)$ is non-orientable, and the multiplicity $m_n$ is odd;
        \item[(2)] $\Gamma_n/S^1$ separates $\mb S^4/S^1$ into two regions $\Omega_\pm$ with $A\cdot \Omega_\pm=\Omega_\mp=s\cdot \Omega_\pm$.  
    \end{itemize}
	Combined with $\mc H^3_{g_{\mb S^4}}(E_0)\leq \mc H^3_{g_{\mb S^4}}(\Gamma_n)$ and the above width estimates, we see $m_n=1$. 
	One can also show $m_n=1$ by the proof of \cite{wang2026spherical}*{(7.8), Theorem 7.10}, where $\Ric_{\mb S^4}>0$ is essentially used.

	Moreover, the embeddedness of $\Gamma_n$ implies that $\Gamma_n/S^1\subset \mathring{\mb S}^4/S^1$ does not contain the north or south pole, and can only intersect $\Xi_0$ (cf. \eqref{Eq: axes on S3}) orthogonally (cf. \cite{wang2026spherical}*{Lemma 2.8}). 
	Hence, by applying the genus upper bound of \cite{ketover2016free}*{Theorem 3.2} in $\mathring{\mb S}^4/S^1$, we have that 
	\[\genus(\Gamma_n/S^1) \leq 2n-2.\] 
	It is now sufficient to show that $\genus(\Gamma_n/S^1)\in \{0,2n-2\}$. 
	
	 Consider the $\mb Z_2\times D_n$-invariant circle $\Xi_0\subset \mb S^4/S^1$ defined in \eqref{Eq: axes on S3} corresponding to the $z$-axis in $\mb R^3$ via the stereographic projection. 
	 By the above statement in (2) and \eqref{Eq: number of points in orbits on Xi_0}, we conclude similarly to \eqref{Eq: intersection with C0 for n even} that 
	 \begin{align}\label{Eq: intersection with C0}
	 	m=2+4j, \qquad \mbox{and}\qquad (0^{(2)},\pm \frac{1}{2},0 )\in \Gamma_n/S^1\cap \Xi_0.
	 \end{align}
	 
	 Consider the $\mb Z_n<D_n$ symmetries on $\mb S^4/S^1$ generated by $r\in D_n$ in \eqref{Eq: Dn reflectional action on S4/S1}, which is free on $(\mb S^4/S^1) \setminus \Xi_0$ with the fixed-point set $\Xi_0$. 
  Similar to the proof of Proposition \ref{Prop: new minimal hypersurface of even n}, we know that the $\mb Z_n$-action on $\Gamma_n/S^1$ is orientation-preserving, and $(\Gamma_n/S^1)/\mb Z_n$ is orientable without boundary. 
	 Hence, the quotient map $\Gamma_n/S^1\to (\Gamma_n/S^1)/\mb Z_n$ is an $n$-degree covering with $2+4j$ branch points $\{p_i\}_{i=1}^{2+4j}$ of ramification index $n=|\mb Z_n|$. 
	 By the Riemann-Hurwitz formula (\cite{ketover2016equivariant}*{(6.1)}\cite{carlotto2022free}*{Remark B.2}), there exist integers $\gamma,j\in \{0,1,\dots\}$ so that
  \[ 2-2\genus(\Gamma_n/S^1 ) = n(2-2\gamma)-(2+4j)(n-1), \]
	 and thus,
	 \begin{align}\label{Eq: R-H formula of intersection 2+4j}
	 	\genus(\Gamma_n/S^1 )=n\gamma   + 2j(n-1) . 
	 \end{align}
	For any $t\in (0,1)$, the $G_n$-hypersurface $\wti\Sigma_t$ given by Lemma \ref{Lem: Gn sweepout} has $\genus(\wti\Sigma_t/S^1)=2n-2$ and $\#((\wti\Sigma_t/S^1) \cap \Xi_0)=6$, which implies that $\genus((\wti\Sigma_t/S^1)/\mb Z_n) = 0$ by applying \eqref{Eq: R-H formula of intersection 2+4j} to $\wti\Sigma_t/S^1$. 
	By the proof of \cite{wang2026spherical}*{Theorem 6.6} and \cite{ketover2016equivariant}*{(1.3)}, we have $\gamma=\genus((\Gamma_n/S^1 )/\mb Z_n) =0$ and 
	\[ \genus(\Gamma_n/S^1 )=2j(n-1) .\] 
	 %
	 %
	 %
	 %
	 If $j\geq 2$, then $\genus( \Gamma_n/S^1)\geq 4n-4 > 2n-2$ as $n=2k+1\geq 3$, which contradicts $\genus( \Gamma_n/S^1)\leq 2n-2$. 
	 Hence, $j\in \{0,1\}$, and $\genus( \Gamma_n/S^1)\in \{0, 2n-2\}$. 
\end{proof}

The following theorem gives a new class of minimal hypersurfaces in $\mb S^4$ whose topological types are non-trivial principal $S^1$-bundles. 

\begin{theorem}\label{Thm: new minimal hypersurfaces of odd n}
	For any odd integer $n=2k+1\geq 3$, there exists an embedded $G_n$-invariant minimal hypersurface $\Gamma_n\subset \mb S^4$ that is a principal $S^1$-bundle over an orientable closed surface $\Gamma_n/S^1$ of $\genus(\Gamma_n/S^1)\in \{0, 2n-2=4k\}$ so that 
	\begin{itemize}
		\item[(i)] the $S^1$-bundle $\Gamma_n\to \Gamma_n/S^1$ is nontrivial, and its Euler class is a generator of $H^2(\Gamma_n/S^1;\mb Z)$;
		\item[(ii)] if $\genus(\Gamma_n/S^1)=0$, then $\Gamma_n$ is a non-equatorial hypersphere that is non-congruent to any of Hsiang's examples \cite{hsiang1983sphericalI};
		\item[(iii)] if $\genus(\Gamma_n/S^1)=2n-2\geq 4$, then $\pi_2(\Gamma_n)=0$  and $\Gamma_n$ is not diffeomorphic to any finite connected sum of copies of $S^2\times S^1$. 
	\end{itemize}
	Moreover, there exists $N\in\mb N$ so that $\Gamma_n/S^1$ has genus $2n-2$ for every odd integer $n\geq N$.  
\end{theorem}


\begin{proof}
	{\bf Part I}. 
    For any odd integer $n=2k+1\geq 3$, it follows from Proposition \ref{Prop: new minimal hypersurface} that there is an embedded $G_n$-invariant minimal hypersurface $\Gamma_n\subset \mathring{\mb S}^4$ with $\genus(\Gamma_n/S^1)\in\{0,2n-2\}$. 
	Since the $S^1$-action on $\Gamma_n\subset\mathring{\mb S}^4$ is free, we see that $\Gamma_n\to\Gamma_n/S^1$ is a principal $S^1$-bundle.

	By the proof of Claim \ref{Claim: Euler class}, the Euler class of the principal $S^1$-bundle $\Gamma_n\to\Gamma_n/S^1$ is $e(\Gamma_n):=i^*e(\mathring{\mb S}^4)\in H^2(\Gamma_n/S^1;\mb Z)$, where $i:\Gamma_n/S^1\to \mathring{\mb S}^4/S^1$ is the embedding, and $e(\mathring{\mb S}^4)\in H^2(\mathring{\mb S}^4/S^1;\mb Z)\cong H^2(E_0/S^1;\mb Z)\cong \mb Z$ is a generator. 
	Additionally, $\Gamma_n/S^1$ is homologous to $E_0/S^1$ in $\mathring{\mb S}^4/S^1$ since $\Gamma_n/{\rm Pin}_-(2)$ and $E_0/{\rm Pin}_-(2)$ are both non-orientable in $\mb S^4/{\rm Pin}_-(2)=RP^3$. 
	Hence, 
	\[ (e(\Gamma_n))([\Gamma_n/S^1])= (i^* e(\mathring{\mb S}^4) )([\Gamma_n/S^1] )=(e(\mathring{\mb S}^4)) (i_*[\Gamma_n/S^1]) =(e(\mathring{\mb S}^4))([E_0/S^1])=\pm 1,\]
	which implies that $e(\Gamma_n)\in H^2(\Gamma_n/S^1;\mb Z)$ is a generator, and the principal $S^1$-bundle $\Gamma_n\to \Gamma_n/S^1$ is non-trivial.  
	
	
	\medskip
    {\bf Part II}. 
	If $\genus(\Gamma_n/S^1)=0$, then $\Gamma_n$ is a minimal hypersphere in $\mb S^4$ (cf. \cite{wang2026spherical}*{Lemma 7.6}), which is non-equatorial since $\mc H^3(\Gamma_n)> \mc H^3(E_0)$ by Proposition \ref{Prop: new minimal hypersurface}. 
	By the constructions of Hsiang \cite{hsiang1983sphericalI}, there are non-equatorial minimal hyperspheres $\{\wti\Sigma_{2i+1}\}_{i=1}^\infty$ that are invariant under the action of $O(2)\times O(2)$ on $\mb S^4$ defined by $(Q_1,Q_2)\cdot (x_1,\dots, x_5)=(Q_1(x_1,x_2), Q_2(x_3,x_4), x_5)$ for any $Q_1,Q_2\in O(2)$. 
	By \cite{hsiang1983sphericalI}*{Lemma 2} and the central symmetry given in \cite{hsiang1983sphericalI}*{Theorem 1}, for each of Hsiang's examples $\wti\Sigma_{2i+1}$, there exists $r_i \in (0,1)$ so that the family of geodesic spheres $S_{\rho}:=\{(x_1,\dots,x_4,\rho)\in \mb S^4\}$, $\rho\in (-1,1)$, satisfies that 
	\begin{itemize}
		\item $S_\rho\cap \wti\Sigma_{2i+1}=\emptyset$ for $\rho\in (-1,-r_i)\cup (r_i,1)$;
		\item $S_{\pm r_i}\cap \wti\Sigma_{2i+1} $ is a single great circle on $S_{\pm r_i}$.
	\end{itemize}
	However, for any $t\in (-1,0)\cup (0,1)$ with $S_t\cap\Gamma_n\neq\emptyset$, we have that $(S_{t}\cap \Gamma_n)/S^1$ is non-empty and invariant under the actions of $\mb Z_2:=\langle A\circ s \rangle$ and $\mb Z_n:=\langle r\rangle$, which implies that $S_{t}\cap \Gamma_n$ contains at least two disjoint circles since $\max\{\#(\langle A\circ s \rangle\cdot p), \#(\langle r\rangle\cdot p)\}\geq 2$ for any $p\in S_t/S^1$. 
    Using this observation, we claim that $\Gamma_n$ is non-congruent to any non-equatorial minimal hyperspheres $\{\wti\Sigma_{2i+1}\}_{i\geq 1}$ constructed by Hsiang in \cite{hsiang1983sphericalI}. 

    Specifically, suppose there exists $F\in O(5)$ so that $\Gamma_n= F(\wti\Sigma_{2i+1})$ for some $i\geq 1$. 
    Then, since $\wti\Sigma_{2i+1}$ is invariant under the action of $T(2):=SO(2)\times SO(2)$, we conclude that 
    \[T':= F\cdot T(2) \cdot F^{-1} \subset {\rm Stab}_0(\Gamma_n)\subset{\rm Stab}(\Gamma_n):=\{g\in O(5): g\cdot \Gamma_n=\Gamma_n\} ,\]
    where ${\rm Stab}_0(\Gamma_n)$ is the identity component of the compact Lie group ${\rm Stab}(\Gamma_n)$. 
    Note that $T(2)=SO(2)\times SO(2)$ is a maximal torus in $SO(5)$ (cf. \cite{brocker1985representations}*{Chapter IV, Theorem 3.4}), where `torus' means it is a compact connected abelian subgroup, and `maximal' means it is not contained in any strictly larger torus in $SO(5)$. 
    Hence, it follows that $T'\subset {\rm Stab}_0(\Gamma_n)\subset SO(5)$ is a maximal torus in ${\rm Stab}_0(\Gamma_n)$. 
    Meanwhile, the $S^1$-invariance of $\Gamma_n$ indicates that $T(1):= {\rm diag}(T(2)) = \{(Q,Q): Q\in SO(2)\}$ is a connected abelian subgroup (i.e. a $1$-dimensional torus) of ${\rm Stab}_0(\Gamma_n)$. 
    Therefore, since $\dim({\rm Stab}_0(\Gamma_n))\leq \dim(SO(5))$ is finite, the non-empty collection $\{G\subset {\rm Stab}_0(\Gamma_n): G \mbox{ is a torus with } T(1)\subseteq G\}$ admits an element $T_{max}$ of maximal dimension, which must be a maximal torus in ${\rm Stab}_0(\Gamma_n)$. 
    By the maximal torus theorem \cite{brocker1985representations}*{Chapter IV, Theorem 1.6}, there exists $F'\in {\rm Stab}_0(\Gamma_n)$ so that $F'\cdot T'\cdot (F')^{-1} = T_{max}$. 
    In particular, for $\wti F:=F'\cdot F \in O(5)$, we have $T(1)\subset T_{max} = \wti F\cdot T(2)\cdot \wti F^{-1}\subset {\rm Stab}_0(\Gamma_n)$. 
    
    Consider the fixed-point set ${\rm Fix}_{\mb R^5}(T_{\max})$ of $T_{max}$ in $\mb R^5$. 
    Since $T(1)\subset T_{max}$, we see that 
    \[ {\rm Fix}_{\mb R^5}(T_{max}) \subset {\rm Fix}_{\mb R^5}(T(1)) = {\rm Fix}_{\mb R^5}(T(2)) = l:=\{(0^{(2)},0^{(2)},x)\in\mb R^5\} .\]
    Meanwhile, since $T_{max}$ is conjugate to $T(2)$, we see that $\dim({\rm Fix}_{\mb R^5}(T_{max}))=\dim({\rm Fix}_{\mb R^5}(T(2)))=1$, and thus ${\rm Fix}_{\mb R^5}(T_{max})=l$. 
    Note that $\wti F({\rm Fix}_{\mb R^5}(T(2))) = {\rm Fix}_{\mb R^5}(\wti F \cdot T(2) \cdot \wti F^{-1})$. 
    Therefore, 
    \[\wti F(l)=\wti F({\rm Fix}_{\mb R^5}(T(2))) = {\rm Fix}_{\mb R^5}(\wti F \cdot T(2) \cdot \wti F^{-1}) = {\rm Fix}_{\mb R^5}(T_{max}) = l.\]
    As $l$ is $\wti F$-invariant, we conclude that $\wti F$ is expressed as the matrix ${\rm diag}(Q',\delta)$, where $Q'\in O(4)$ and $\delta\in\{\pm 1\}$. 
    Hence, $\wti F(S_\rho)\in \{S_\rho,S_{-\rho}\}$, 
    and $\wti F(\wti\Sigma_{2i+1}\cap S_{r_i}) $ is either $ \Gamma_n\cap S_{r_i}$ or $ \Gamma_n\cap S_{-r_i}$, which contradicts the previous observation. 
	This proves (ii). 
	
	\medskip
    {\bf Part III}. 
	If $\genus(\Gamma_n/S^1)=2n-2\geq 4$, then $\pi_2(\Gamma_n/S^1)=0$, and thus $\pi_2(\Gamma_n)=0$ as in the proof of Theorem \ref{Thm: new minimal hypersurfaces of even n}.  
	Since $\pi_2(\#_m (S^2\times S^1))\neq 0$ for all $m\geq 1$, we conclude that $\Gamma_n$ is not diffeomorphic to any finite connected sum of copies of $S^2\times S^1$.  This proves (iii).

	\medskip
    {\bf Part IV}. 
	For the last statement in the above theorem, suppose for contradiction that there exists a sequence $\{\Gamma_{n_i}\}_{i\in\mb N}$ with $n_i\to \infty$ as $i\to\infty$ so that 
	\[\genus(\Gamma_{n_i}/S^1)=0\qquad \mbox{for every $i$}.\] 
	Since $\sup_{i\in\mb N}\mc H^3_{g_{\mb S^4}}(\Gamma_{n_i}) \leq  3\mc H^3_{g_{\mb S^4}}(E_0)$, we know that $\Gamma_{n_i}$ converges (up to a subsequence) to a stationary $S^1$-invariant $3$-varifold $V$ in $\mb S^4$. 
	Equivalently, there exists a stationary $2$-varifold $W$ in $\mb S^4/S^1$ (under the weighted metric $\wti g_{\mb S^4/S^1}$) so that for any open set $U\subset\subset \mathring{\mb S}^4/S^1$, $(\Gamma_{n_i}/S^1)\llcorner U$ converges (up to a subsequence) to $W\llcorner U$ in the sense of varifolds. 
	
	For any fixed $n_i$, consider the orientation-preserving $\mb Z_2$-action on $\mb S^4/S^1$ generated by $sr^{n_i-l}\in D_{n_i}$ in \eqref{Eq: Dn reflectional action on S4/S1}, where $l=1,\cdots, n_i$.
	Recall that $\Gamma_{n_i}/S^1$ separates the north and south poles of $\mb S^4/S^1$ into different regions. 
	Therefore, $sr^{n_i-l}$ is orientation-reversing on $\Gamma_{n_i}/S^1$, which indicates that the fixed-point set ${\rm Fix}_{\Gamma_{n_i}/S^1}(sr^{n_i-l})$ of $sr^{n_i-l}$ on $\Gamma_{n_i}/S^1$ is either empty or a disjoint union of simple closed curves. 
	By \eqref{Eq: intersection with C0}, ${\rm Fix}_{\Gamma_{n_i}/S^1}(sr^{n_i-l})\neq \emptyset$. 
	Meanwhile, the fixed-point set of $sr^{n_i-l}$ in the whole orbit space $\mb S^4/S^1$ is the great circle $\Upsilon_{l}\subset E_0/S^1\cong \mb S^2_{1/2}$ (cf. \eqref{Eq: circle on E0/S1}) by Lemma \ref{Lem: fixed points set of s}. 
	Together, we conclude that 
	\[ \Upsilon_{l}=  {\rm Fix}_{\Gamma_{n_i}/S^1}(  sr^{n_i-l}) \subset \Gamma_{n_i}/S^1, \qquad \forall l\in \{1,\dots,n_i\} .\]
	For clarity, denote by $\{\Upsilon_l^{n_i}\}_{l=1}^{n_i}$ the great circles in \eqref{Eq: circle on E0/S1} associated to the $D_{n_i}$-action. 
	
	We now claim that $E_0/S^1$ is contained in the support $\spt(\|W\|)$ of the limit varifold $W$. 
	Indeed, if $p\in (E_0/S^1)\setminus \spt(\|W\|)$, then there exists a small neighborhood $B$ of $p$ in $\mb S^4/S^1$ so that $\|W\|(B)=0$. 
	However, for any sufficiently large $n_i$, $\cup_{l=1}^{n_i} \Upsilon_{l}^{n_i}\subset \Gamma_{n_i}/S^1$ intersects $B$. 
	By the monotonicity formula, $ \| \Gamma_{n_i}/S^1\|(B) $ has a positive lower bound independent of $n_i$, which contradicts $\|W\|(B)=0$. 
	
	Next, take a smooth compact neighborhood $\Omega:= \{(w,t,x)\in \mb S^4/S^1: |x|\leq 1-\epsilon\}$ of $E_0/S^1$ so that $\bd \Omega$ is transversal to every $\Gamma_{n_i}/S^1$ and $\Omega$ is $\mb Z_2\times D_{n_i}$-invariant for every $n_i$. 
	Under the (non-degenerate) weighted metric $\wti g_{\mb S^4/S^1}$ on $\Omega$, the area and the genus of $(\Gamma_{n_i}/S^1) \cap \Omega$ are uniformly bounded from above. 
	Hence, by the compactness theorem of White \cite{white2018compactness}*{Theorem 1.1}, $(\Gamma_{n_i}/S^1) \cap \Omega$ converges (up to a subsequence) to a smooth, embedded, minimal surface $\Sigma\subset \Omega$ with integer multiplicity on each component. 
	By the previous claim, $E_0/S^1$ must be a connected component of $\Sigma$ with some multiplicity $m\in \{1,2,\dots\}$. 
	If $m\geq 2$, then it follows from \cite{white2018compactness}*{Theorem 1.1, (2)} that $E_0/S^1$ is stable in $\mb S^4/S^1$ under the weighted metric $\wti g_{\mb S^4/S^1}$, which contradicts Lemma \ref{Lem: properties of S1-action}(ii). 
	Therefore, $m=1$, and \cite{white2018compactness}*{Theorem 1.1, (1)} implies that the convergence to $E_0/S^1$ is smooth.
	Hence, for $n_i$ sufficiently large, the connectedness of $\Gamma_{n_i}/S^1$ implies that $\Gamma_{n_i}/S^1$ is entirely contained in $\Omega$, and is a small $C^\infty$-perturbation of $E_0/S^1$. 
	
	In particular, we have $\mc H^3_{g_{\mb S^4}}(\Gamma_{n_i})\to \mc H^3_{g_{\mb S^4}}(E_0)$ as $i\to\infty$. 
	Note that there exists $\epsilon>0$ so that any minimal hypersurface $S\subset \mb S^4$ that is not totally geodesic must have $\mc H^3(S)>\mc H^3(E_0)+\epsilon$ (cf. \cite{cheng1984heat}). 
	Hence, $\Gamma_{n_i}\equiv E_0$ for all $n_i$ large enough, which contradicts $\mc H^3_{g_{\mb S^4}}(\Gamma_{n_i})=\mf L(\Pi_{n_i}) > \mc H^3_{g_{\mb S^4}}(E_0)$ in Proposition \ref{Prop: new minimal hypersurface}. 
	Finally, we conclude that $\genus(\Gamma_n/S^1)=2n-2$ for all $n$ large enough.  
\end{proof}

\section{\texorpdfstring{Embedded minimal Seifert fibered manifolds in $\mb S^4$}{Embedded minimal Seifert fibered manifolds in S4}}

In this section, we generalize the suspended Hopf action by adding weights, which can be used to construct minimal $S^1$-bundles over orbifolds (i.e. Seifert fibered manifolds) in $\mb S^4$. 

\subsection{\texorpdfstring{Weighted Hopf action on $\mb S^3$}{Weighted Hopf action on S3}}

Given nonzero coprime integers $p,q\in \mb Z\setminus\{0\}$, 
the {\em weighted Hopf $S^1_{p,q}$-action} on $\mb S^3=\{(z_1,z_2)\in\mb C^2:|z_1|^2+|z_2|^2=1\}$ is defined by 
\[ e^{i\theta}\cdot (z_1,z_2) := (e^{i\theta p} z_1, e^{i\theta q} z_2), \qquad \forall e^{i\theta}\in S^1=\{\lambda\in \mb C :|\lambda|=1\},\]
so that the quotient map $h_{p,q}:\mb S^3\to \mb S^2_{1/2}:=\{(w,t)\in \mb C\times \mb R: |w|^2+t^2=1/4\}$ is the continuous extension of the following map on $\{(z_1,z_2)\in\mb S^3: z_1z_2\neq 0\}$:
\begin{align}\label{Eq: weighted Hopf map}
	h_{p,q}(z_1,z_2) 
	:= \left(|z_1||z_2| e^{i(q \arg(z_1) - p\arg(z_2) )}, \frac{|z_1|^2-|z_2|^2}{2} \right)=  \left( \frac{\sin(\eta)}{2} e^{i \vartheta}, \frac{\cos(\eta)}{2}  \right) ,
\end{align}
where $\arg(z)$ is the argument of $z\in \mb C$, $\vartheta:=q \arg(z_1) - p\arg(z_2)$, and $\eta\in [0,\pi]$ satisfies $|z_1|=\cos(\eta/2)$, $|z_2|=\sin(\eta/2)$. 
If $|p|>1$, then $\{(z_1,0^{(2)}): |z_1|=1\}$ is an exceptional $S^1_{p,q}$-orbit with the isotropy group $\mb Z_{|p|}=\langle e^{i \frac{2\pi}{p}}\rangle$. 
Similarly, if $|q|>1$, then $\{(0^{(2)}, z_2): |z_2|=1\}$ is an exceptional orbit with the isotropy group $\mb Z_{|q|}=\langle e^{i\frac{2\pi}{q}}\rangle$. 
Hence, the orbit space $\mb S^3/S^1_{p,q}$ is topologically a $2$-sphere with two possibly singular points at the north and south poles
\begin{align}\label{Eq: singular points on S3/S1pq}
	h_{p,q}(\{(z_1,0^{(2)}):|z_1|=1\})=(0^{(2)},\frac{1}{2}) \quad \mbox{and}\quad h_{p,q}(\{(0^{(2)}, z_2):|z_2|=1\})=(0^{(2)}, -\frac{1}{2})
\end{align} 
respectively. 
Indeed, under the parameterization $\mb S^2_{1/2}=\{ \frac{1}{2} (\sin(\eta)e^{i\vartheta}, \cos(\eta)): \eta\in[0,\pi],\vartheta\in[0,2\pi]\}$, we can define the metric $g_{\mb S^3/S^1_{p,q}}$ on $\mb S^3/S^1_{p,q}$ by 
\begin{align}\label{Eq: reduced metric of S3/S1pq}
	g_{\mb S^3/S^1_{p,q}}:= \frac{1}{4}\left(d \eta^2 + \frac{\sin^2(\eta)}{p^2\cos^2(\frac{\eta}{2} )+ q^2\sin^2(\frac{\eta}{2})} d\vartheta^2 \right)
\end{align}
so that the quotient map $h_{p,q}$ in \eqref{Eq: weighted Hopf map} is a Riemannian submersion on $\{(z_1,z_2)\in \mb S^3: z_1z_2\neq 0\}$. 
Note that $g_{\mb S^3/S^1_{p,q}}$ can have orbifold singularities at $\eta=0$ (if $|p|>1$) and  $\eta=\pi$ (if $|q|>1$). 

Additionally, using the coordinates $\eta=\arccos(|z_1|^2-|z_2|^2)\in [0,\pi]$ and $\vartheta=q \arg(z_1) - p\arg(z_2)\in [0,2\pi]$, the reflectional $D_n$-action \eqref{Eq: Dn reflectional action on R3} on $\mb S^3/S^1_{p,q}$ is generated by 
\begin{align}\label{Eq: reflectional Dn action in S3/S1pq}
	r: (\eta,\vartheta)\mapsto (\eta, \vartheta+\frac{2\pi}{n} ) \quad {\rm and}\quad s: (\eta,\vartheta) \mapsto (\eta, -\vartheta),
\end{align}
which remains an isometric action on $\mb S^3/S^1_{p,q}$ under the metric $g_{\mb S^3/S^1_{p,q}}$. 
Moreover, for any $(\eta,\vartheta)\in (0,\pi)\times [0, 2\pi]$, we have 
\[h_{p,q}^{-1}\left(\sin(\eta)e^{i\vartheta}/2, \cos(\eta)/2\right) = \left\{ \left(\cos(\eta/2) e^{i(p\theta+ \vartheta/q)} , \sin(\eta/2)e^{i q\theta } \right)\in\mb S^3: \theta\in [0,2\pi]  \right\}\]
is a circle on the torus $T=\{(z_1,z_2)\in \mb S^3: |z_1|=\cos(\eta/2)\}$ with length 
\begin{align}\label{Eq: orbit volume of S1pq action on S3}
	{\rm Length}(h_{p,q}^{-1}(\frac{\sin(\eta)e^{i\vartheta}}{2}, \frac{\cos(\eta)}{2} ) ) = 2\pi \sqrt {p^2\cos^2(\frac{\eta}{2}) + q^2\sin^2(\frac{\eta}{2})}=2\pi\sqrt{p^2|z_1|^2+q^2|z_2|^2},
\end{align}
which is invariant under the reflectional $D_n$-action.

\subsection{\texorpdfstring{Suspended weighted Hopf action and $G_n^{p,q}$-action on $\mb S^4$}{Suspended weighted Hopf action and Gnpq-action on S4}}
Next, we generalize the suspended Hopf action on $\mb S^4$ by the following {\em suspended $S^1_{p,q}$-action}: 
\begin{align}\label{Eq: suspended weighted Hopf action}
	e^{i\theta}=R_{\theta}^{p,q}: (z_1,z_2,x)\in \mb S^4 \mapsto (e^{i\theta p} z_1, e^{i\theta q} z_2, x)\in \mb S^4, 
\end{align}
where $e^{i\theta}\in S^1=\{\lambda\in \mb C :|\lambda|=1\}$. 
Then, we have the quotient map $\pi_{p,q}:\mb S^4\to \mb S^4/S^1_{p,q}$ as the continuous extension of the following map on $\{(z_1,z_2,x)\in\mb S^4: z_1z_2\neq 0\}$:
\begin{align}\label{Eq: quotient of weighted Hopf action}
	\pi_{p,q}(z_1,z_2,x) 
	& :=  \left(\frac{|z_1||z_2| e^{iq \arg(z_1) - p\arg(z_2)}}{\sqrt{|z_1|^2+|z_2|^2}}, \frac{|z_1|^2-|z_2|^2}{2 \sqrt{|z_1|^2+|z_2|^2}}, x \right).
\end{align} 
Additionally, $\mb S^4/S^1_{p,q}$ is still a topological $3$-sphere that can be parameterized by an ellipsoid:
\begin{align}\label{Eq: parameterize S4/S1pq}
	\mb S^4/S^1_{p,q}&= \left\{(w,t,x)\in \mb C\times \mb R\times \mb R: 4(|w|^2+t^2)+x^2=1 \right\}\nonumber 
	\\
	&=\left\{\left(\frac{1}{2}\sin(\varrho)\sin(\eta)e^{i\vartheta}, \frac{1}{2}\sin(\varrho)\cos(\eta), \cos(\varrho) \right ) : \varrho,\eta\in [0,\pi], \vartheta\in [0,2\pi] \right\},
\end{align}
where $\varrho:=\arccos(x)$, 
$\eta=\arccos(\frac{|z_1|^2-|z_2|^2}{|z_1|^2+|z_2|^2})$
and $\vartheta :=q \arg(z_1) - p\arg(z_2)$. 
Note that apart from the singularities at the north and south poles $(0^{(2)},0,\pm 1)\in \mb S^4/S^1_{p,q}$, there are also two possible singular arcs $\Xi_0^\pm$ (cf. \eqref{Eq: half axis z on S3})
by suspending the possible singular points \eqref{Eq: singular points on S3/S1pq} in $E_0/S^1_{p,q}$ along the $x=\cos(\varrho)$ coordinate. 
Then, we define the metric $g_{\mb S^4/S^1_{p,q}}$ by 
\begin{align}\label{Eq: reduced metric in S4/S1pq}
	g_{\mb S^4/S^1_{p,q}}:= d\varrho^2+ \sin^2(\varrho) g_{\mb S^3/S^1_{p,q}} 
\end{align}
so that the quotient map $\pi_{p,q}$ in \eqref{Eq: quotient of weighted Hopf action} is a Riemannian submersion when restricted to 
\begin{align}
	 \mathring{\mb S}^4_{p,q}:= \mb S^4\setminus \pi_{p,q}^{-1}(\{(0^{(2)},0,\pm 1)\}\cup \Xi_0^\pm). 
\end{align}
Additionally, for any $(w,t,x)\in \mb S^4/S^1_{p,q}$ and $(z_1,z_2,x)\in \pi_{p,q}^{-1}(w,t,x)$, we have $|z_1|^2+|z_2|^2=1-x^2$ and $|z_1|^2-|z_2|^2= 2t\sqrt{1-x^2}$, which implies $|z_1|^2=\frac{1-x^2}{2}+t\sqrt{1-x^2}$ and $|z_2|^2=\frac{1-x^2}{2}-t\sqrt{1-x^2}$. 
Hence, combined with \eqref{Eq: orbit volume of S1pq action on S3}, the orbit volume function $V_{p,q}$ on $\mathring{\mb S}^4_{p,q}/S^1_{p,q}$ is 
\begin{align}\label{Eq: orbit volume of S1pq action on S4}
	V_{p,q}(w,t,x):= {\rm Length}(\pi_{p,q}^{-1}(w,t,x)) = 2\pi \sqrt{\frac{p^2+q^2}{2}(1-x^2)+ (p^2-q^2)t\sqrt{1-x^2} } .
\end{align}

For general coprime integers $p,q\in \mb Z\setminus\{0\}$, the antipodal map $A$ in \eqref{Eq: antipodal} may fail to be an isometry on $\mb S^4/S^1_{p,q}$ under the metric $g_{\mb S^4/S^1_{p,q}}$, and the volume function $V_{p,q}$ of $S^1_{p,q}$-orbits may also fail to be $A$-invariant.  
Nevertheless, by \eqref{Eq: reduced metric of S3/S1pq}, \eqref{Eq: reflectional Dn action in S3/S1pq}, and \eqref{Eq: reduced metric in S4/S1pq}, we still have the {\em modified reflectional $D_n$-action on $\mb S^4/S^1_{p,q}$}, $n\in \{2,3,\dots\}$, generated by 
\begin{align}\label{Eq: Dn reflectional action on S4/S1pq}
	r: (w,t,x)\mapsto (e^{i\frac{2\pi}{n} }w, t,x) \qquad {\rm and}\qquad s: (w,t,x)\mapsto (\bar w, t,-x)
\end{align}
as an orientation-preserving isometric $D_n$-action on $(\mb S^4/S^1_{p,q}, g_{\mb S^4/S^1_{p,q}})$ so that $V_{p,q}$ is $D_n$-invariant. 
Hence, after continuously extending $V_{p,q}$ from $\mathring{\mb S}^4_{p,q}/S^1_{p,q}$ to $\mb S^4/S^1_{p,q}$, we also have the $D_n$-invariant weighted metric $\wti g_{\mb S^4/S^1_{p,q}}:= V_{p,q}\cdot g_{\mb S^4/S^1_{p,q}}$ so that  
\[\mc H^3_{g_{\mb S^4}}(\Sigma) = \mc H^2_{\wti g_{\mb S^4/S^1_{p,q}}}(\pi_{p,q}(\Sigma))\] 
for any $S^1_{p,q}$-invariant hypersurface $\Sigma\subset \mb S^4$. 

This $D_n$-action on $\mb S^4/S^1_{p,q}$ can also be lifted into an isometric $D_n$-action on $\mb S^4$. 
Specifically, since $(p,q)=1$, there exist $a,b\in \mb Z$ so that 
\begin{align}\label{Eq: ab corresponding to pq}
     qa+pb=1
\end{align}
by B{\'e}zout's theorem. 
Then, define the modified reflectional $D_n$-action on $\mb S^4$ to be generated by 
\begin{align}\label{Eq: Dn reflectional action on S4pq}
	\hat r: (z_1,z_2,x)\mapsto  (e^{\frac{2\pi a}{n} i}z_1, e^{-\frac{2\pi b}{n}i} z_2, x)
    \quad {\rm and }\quad 
    \hat s: (z_1,z_2,x)\mapsto (\bar z_1, \bar z_2, -x).
\end{align}
One verifies that $\pi_{p,q}\circ \hat r= r\circ \pi_{p,q}$ and $\pi_{p,q}\circ \hat s=s\circ \pi_{p,q}$ by \eqref{Eq: quotient of weighted Hopf action} and \eqref{Eq: Dn reflectional action on S4/S1pq}. 
Additionally, when $p=q=1$ and $(a, b)=(1,0)$, the above $D_n$-action reduces to \eqref{Eq: Dn reflectional action on S4}. 

Note that the pair $a,b\in \mb Z$ is not unique since $a_m=a+mp$ and $b_m=b-mq$ also satisfy $qa_m+pb_m=1$ for all $m\in\mb Z$. 
Nevertheless, since 
\begin{align}\label{Eq: different pairs of ab}
	R_{ 2\pi m/n}^{p,q} \circ \hat r: (z_1,z_2,x)\mapsto (e^{\frac{2\pi (a+mp)}{n} i}z_1, e^{-\frac{2\pi (b-mq)}{n}i} z_2 , x)= (e^{\frac{2\pi a_m}{n} i}z_1,  e^{-\frac{2\pi b_m}{n}i} z_2 , x),
\end{align}
the suspended $S^1_{p,q}$-action extends uniquely through the modified reflectional $D_n$-action on $\mb S^4$. 
To be exact, we see from the definitions in \eqref{Eq: suspended weighted Hopf action} and \eqref{Eq: Dn reflectional action on S4pq} that 
\[ 
	R_{\theta}^{p,q}\circ \hat r=\hat r\circ R_{\theta}^{p,q}
	\qquad {\rm and} \qquad 
	R_{\theta}^{p,q}\circ \hat s=\hat s\circ R_{-\theta}^{p,q},
\]
where $R_{\theta}^{p,q}\in S^1_{p,q}$ and $\hat r,\hat s\in D_n$. 
Hence, consider the homomorphism $\psi: D_n\to {\rm Aut}(S^1)$ generated by $\psi(\hat r):= id$ and $\psi(\hat s): e^{i\theta}\mapsto e^{-i\theta}$. 
We then have a Lie group 
\begin{align}\label{Eq: Gnpq action on S4}
	G_{n}^{p,q} := S^1\rtimes_\psi D_n
\end{align}
acting isometrically on $\mb S^4$ by $(e^{i\theta}, \hat s^{k} \hat r^l)\cdot (z_1,z_2,x):=R_\theta^{p,q}\circ \hat s^{k} \hat r^l(z_1,z_2,x)$, which is well-defined. 
By \eqref{Eq: different pairs of ab}, different choices of $a,b$ only provide different generators so that $\{S^1_{p,q}, \hat r, \hat s\}$ always generates the same $G_n^{p,q}$-action. 
Moreover, it is easy to verify that the $G_n^{p,q}$-action admits no special exceptional $G_n^{p,q}$-orbit on $\mb S^4$ (\cite{wang2026spherical}*{Definition 2.2}).

\medskip
Using the $S^1_{p,q}$-action, we have the following characterization of Seifert invariants for any embedded $G_n^{p,q}$-invariant hypersurface in $\mb S^4$. 

\begin{lemma}\label{Lem: S1pq invariant Seifert manifolds}
    Given $n\in\{2,3,\dots\}$ and coprime $p,q\in \mb Z\setminus\{0\}$, let $\Gamma\subset \mb S^4$ be an embedded connected $G_n^{p,q}$-invariant hypersurface. 
    Then, $\Gamma$ is a Seifert fibered $3$-manifold with the following system of invariants (under a suitable orientation):
    \begin{align}\label{Eq: seifert invariants}
        \left(O,o;\mk g~\Big|~\mk  b_0,(|p|,\mk a)^{s(p)(\delta+j_+)},(|p|,|p|-\mk a)^{s(p)j_+}, (|q|,\mk b)^{s(q)(\delta+j_-)},(|q|,|q|-\mk b)^{s(q)j_-} \right),
    \end{align}
    where $O$ and $o$ indicate the orientability of $\Gamma$ and $\Gamma/S^1_{p,q}$ respectively, 
    \begin{itemize}
        \item[(i)] $\mk g:=\genus(\Gamma/S^1_{p,q})$;
        \item[(ii)] $\delta=1$ if $\Gamma/S^1_{p,q}$ separates the north and south poles of $\mb S^4/S^1_{p,q}$; otherwise, $\delta=0$;
        \item[(iii)] $\delta+2j_\pm=\mk m_\pm:=\#((\Gamma/S^1_{p,q})\cap \Xi_0^\pm)$ 
        (cf. \eqref{Eq: half axis z on S3});
        \item[(iv)] $0\leq \mk a<|p|$ and $0\leq \mk b<|q|$ are the unique integers satisfying $\mk a |q|\equiv -{\rm sgn} (pq)$ mod $|p|$ and $\mk b|p|\equiv -{\rm sgn}(pq)$ mod $|q|$, which can be obtained by using B{\'e}zout's Theorem; 
        \item[(v)] $s:\mb Z\to \{0,1\}$ is the characteristic function of $\mb Z\setminus \{\pm 1\}$; every pair of the form $(\alpha,\beta)$ in \eqref{Eq: seifert invariants} is the invariant determined uniquely by one of the $S^1_{p,q}$-fibers in $\Gamma\cap \pi_{p,q}^{-1}(\Xi_0)$, and $(\alpha,\beta)^j$ means that there are $j$ such pairs ($j$ could be $0$, meaning there is no such pair);
        \item[(vi)] $\mk b_0:= -s(p)\cdot j_+- s(q)\cdot j_--\delta\cdot \kappa(p,q) \in \mb Z$ with $\kappa(p,q):=\frac{\mk a|q|+\mk b|p|+{\rm sgn}(pq)}{|pq|}\in \{-1,0,1\}$ so that the Euler number of the Seifert fibration is 
        \begin{align}\label{Eq: Euler number of Seifert manifold}
            \mk e : =-\left(\mk b_0
            + s(p)\sum_{i=1}^{\delta+j_+}\frac{\mk a}{|p|}
            +s(p)\sum_{i=1}^{j_+}\frac{|p|-\mk a}{|p|}
            +s(q)\sum_{i=1}^{\delta+j_-}\frac{\mk b}{|q|}
            +s(q)\sum_{i=1}^{j_-}\frac{|q|-\mk b}{|q|}\right) 
            =\frac{\delta}{pq}. 
        \end{align}
    \end{itemize} 
    In particular, the system of invariants of $\Gamma$ is determined by the topological data $(\mk g, \mk m_+, \mk m_-)$. 
    Moreover, $\Gamma/S^1_{p,q}$ is the orbifold $\Sigma_{\mk g}(|p|^{s(p)\mk m_+}, |q|^{s(q)\mk m_-})$ formed by a closed orientable $\mk g$-genus surface $\Sigma_{\mk g}$ with $s(p)\mk m_+$ cone points of order $|p|$ and $s(q) \mk m_-$ cone points of order $|q|$ so that 
    \begin{align}\label{Eq: orbifold Euler charateristic}
        \chi_{\rm orb}(\Gamma/S^1_{p,q})= (2-2\mk g) - \mk m_+\left(1 - \frac{1}{|p|} \right) - \mk m_-\left(1 - \frac{1}{|q|} \right),
    \end{align}
    where $\chi_{\rm orb}(\Gamma/S^1_{p,q})$ is the orbifold Euler characteristic of $\Gamma/S^1_{p,q}$. 
\end{lemma}

\begin{remark}\label{Rem: Seifert invariants}
    In \eqref{Eq: seifert invariants}, each pair of the form $(\alpha,\beta)$ corresponds to one of the exceptional orbits of $\Gamma$. 
    Hence, we use $s(\cdot)$ in the superscript to omit the corresponding pairs when $|p|=1$ or $|q|=1$. 
    We also make the following remarks. 
    \begin{itemize}
        \item[(i)] If $|pq|=1$, then the unoriented $G_n^{p,q}$-hypersurface $\Gamma$ is a principal $S^1$-bundle, topologically determined by $\mk g$ and $|\mk e|\equiv \mk m_\pm$ mod $2$. 
        In particular, the system of invariants in \eqref{Eq: seifert invariants} is now given by $(O,o;\mk g|\pm\delta)$. 
        \item[(ii)] By \cite{seifert1980textbook}*{Part II, \S7, Theorem 5}, every orientable Seifert fibered manifold is in one-to-one correspondence with a system of invariants similar to \eqref{Eq: seifert invariants}. 
            By \cite{seifert1980textbook}*{Part II, \S7, Theorem 6}, after reversing the orientation of $\Gamma$, one simply replaces every pair $(\alpha,\beta)$ in \eqref{Eq: seifert invariants} by $(\alpha, \alpha-\beta)$, and replaces $\mk b_0$ by $-( s(p)\mk m_++s(q)\mk m_-)-\mk b_0$. 
        \item[(iii)] By \cite{scott1983three-manifolds}*{Theorem 3.8}, the Seifert fibration of orientable Seifert manifolds is unique up to isomorphism, with the exception of a few special cases. 
        For instance, the equator $E_0$ is $G_n^{p,q}$-invariant for all $n,p,q$, and admits different structures of Seifert fibration. 
        Nevertheless, by \cite{scott1983three-manifolds}*{Theorem 5.3, 3.8}, if 
        \begin{align}\label{Eq: unique seifert fibration}
            \chi_{\rm orb}(\Gamma/S^1_{p,q}) <0 ,\qquad{\rm or} \qquad \chi_{\rm orb}(\Gamma/S^1_{p,q}) =0 \quad{with}\quad \mk e\neq 0,
        \end{align}
        then \eqref{Eq: seifert invariants} gives the unique structure (up to orientation) of the Seifert fibration of $\Gamma$, and the topological type of $\Gamma$ is uniquely determined. 
    \end{itemize}
\end{remark}
    
\begin{proof}
    {\bf Step 1}. 
    We begin with some notations. Firstly, consider the coordinates $(\theta_1,\theta_2,\eta,\varrho)\in [0,2\pi]^2\times [0,\pi]^2$ of $\mb S^4$ defined by 
    \[ z_1=\sin(\varrho)\cos(\eta/2)e^{i\theta_1},\qquad z_2=\sin(\varrho)\sin(\eta/2)e^{i\theta_2},\qquad x=\cos(\varrho), \]
    and let the ordered set $(\bd_{\theta_1}, \bd_{\theta_2},\bd_\eta,\bd_\varrho)$ be the orientation of $\mb S^4$, which indeed coincides with the usual orientation of $\mb S^4$ induced from $ \mb B^5\subset \mb R^5$. 
    Combining the embeddedness of $\Gamma$ with \cite{wang2026spherical}*{Lemma 2.8}, we know $\Gamma/S^1_{p,q}$ intersects $\Xi_0^\pm$ orthogonally at some points $\{u_i^{\pm}\}_{i=1}^{\mk m_\pm}\subset \Xi_0^\pm \cap (\mathring{\mb S}^4/S^1_{p,q})$, where $u_1^\pm,u_2^\pm,\dots u_{\mk m_\pm}^\pm$ are arranged along $\Xi_0^\pm$ in order from south to north, i.e. $\varrho(u_i^\pm)$ is decreasing in $i$. 
    Since $\mb S^4/S^1_{p,q}$ is a simply connected $3$-sphere, $\Gamma/S^1_{p,q}$ is an orientable $C^0$-surface that separates $\mb S^4/S^1_{p,q}$ into two components. 
    Hence, we can take an open $S^1_{p,q}$-set $\Omega$ containing the south pole so that $\Gamma=\bd \Omega$. 
    Note that $(-1)^{i}\bd_\varrho$ points outward of $\Omega$ at $u_i^{\pm}$. 
    Therefore, since $((-1)^i\bd_\varrho, (-1)^{i+1}\bd_{\theta_1},\bd_{\theta_2},\bd_\eta)$ is the positive orientation of $\mb S^4$, we equip $\Gamma=\bd\Omega$ with the induced boundary orientation represented by the ordered set $((-1)^{i+1}\cdot \bd_{\theta_1},\bd_{\theta_2},\bd_\eta)$ near every $\pi_{p,q}^{-1}(u_i^\pm)$. 
    
    Next, note that for any $u\in \Xi_0\cap(\mathring{\mb S}^4/S^1_{p,q})$, unlike \eqref{Eq: number of points in orbits on Xi_0}, $\# (D_n\cdot u)$ is now either $1$ (if $u\in (E_0/S^1_{p,q})\cap \Xi_0$) or $2$ (if $u\in \Xi_0\setminus (E_0/S^1_{p,q})$) since we do not have the antipodal symmetry. 
    Nevertheless, since $\Gamma/S^1_{p,q}$ separates $\mb S^4/S^1_{p,q}$, the circle $\Xi_0$ must be separated by $\Gamma/S^1_{p,q}$ into an even number of arcs, which implies that $\mk m_+$ and $\mk m_-$ have the same parity. 
	Hence, similar to \eqref{Eq: intersection with C0 for n even} and \eqref{Eq: R-H formula of intersection 4j}, we conclude that 
	$\mk m_\pm = 2j_{\pm}+\delta$ 
	for some $j_\pm\in \{0,1,\dots\}$, where $\delta\in \{0,1\}$ is given as in (ii). 

    \medskip
    {\bf Step 2}. {\it The invariants $(\alpha_i^\pm,\beta_i^\pm)$ near exceptional orbits $\pi_{p,q}^{-1}(u_i^\pm)$}. 

    In the system of Seifert invariants \eqref{Eq: seifert invariants}, every pair $(\alpha,\beta)$ is uniquely determined by an exceptional orbit. 
    In particular, if $|p|=1$ (resp. $|q|=1$), then $\{\pi^{-1}_{p,q}(u_i^+)\}_{i=1}^{\mk m_+}$ (resp. $\{\pi^{-1}_{p,q}(u_i^-)\}_{i=1}^{\mk m_-}$) are not exceptional orbits, and thus the pairs of the form $(|p|,\cdot)$ (resp. $(|q|,\cdot)$) should be removed from \eqref{Eq: seifert invariants}. This is equivalent to $s(p)=0$ (resp. $s(q)=0$) in our notations (v). 

    If $|p|>1$, then we consider $u_i^+=(0^{(2)},\sin(\varrho_i^+)/2, \cos(\varrho_i^+))\in (\Gamma/S^1_{p,q})\cap \Xi_0^+$ and follow the constructions in \cite{seifert1980textbook}*{p. 385} to compute the invariants near the exceptional orbit $\pi_{p,q}^{-1}(u_i^+)=\{(\sin(\varrho_i^+)e^{i\theta_1}, 0^{(2)}, \cos(\varrho_i^+) ): \theta_1\in \mb R/2\pi\mb Z\}\subset \Gamma$. 
    Recall that $\Gamma$ intersects $\pi_{p,q}^{-1}(\Xi_0^+)$ orthogonally at the $S^1_{p,q}$-orbit $\pi_{p,q}^{-1}(u_i^+)$.
    Hence, up to an $S^1_{p,q}$-equivariant diffeomorphism, we can parameterize a small tubular neighborhood $B_i^+\subset \Gamma$ of $\pi_{p,q}^{-1}(u_i^+)$ as 
    \[ B_i^+\cong \{(z_1,z_2,\cos(\varrho_i^+))\in \mb S^4: |z_2|< \epsilon\sin(\varrho_i^+)\}\qquad {\rm for ~ some~constant~} \epsilon>0\]
    with coordinates $(\theta_1,\theta_2,\eta)\in (\mb R/2\pi\mb Z)^2\times [0, 2\arcsin(\epsilon))$, and identify the torus $T_i^+:=\bd B_i^+$ with 
    \[T_i^+\cong \left\{\left(\sqrt{1-\epsilon^2} \sin(\varrho_i^+) e^{i\theta_1},\epsilon \sin(\varrho_i^+)e^{i\theta_2}, \cos(\varrho_i^+)\right)\in \mb S^4: \theta_1,\theta_2\in \mb R/2\pi\mb Z\right\}\] with coordinates $(\theta_1,\theta_2)\in (\mb R/2\pi\mb Z)^2$. 
    Equip $T_i^+=\bd B_i^+$ with the boundary orientation induced from $\Gamma$, which is represented by the ordered set $((-1)^{i+1} \bd_{\theta_1}, \bd_{\theta_2})$ since $\bd_\eta$ points outward of $B_i^+$. 
    
    Using the coordinates $(\theta_1,\theta_2)$, we define
    \[L(t):=(t,0) \qquad {\rm and}\qquad N(t):=(0,t) \qquad \forall t\in \mb R/2\pi\mb Z\] 
    as two oriented (induced from $t$) circles on $T_i^+$ with the corresponding homology classes $[L],[N]\in H_1(T_i^+;\mb Z)$. 
    Using the orientation of $T_i^+$, $L$ and $ N$ have the intersection number $s_i^+:=(-1)^{i+1}$, i.e. the ordered set $(s_i^+L, N)$ agrees with the orientation of $T_i^+$. 

    Next, let $M_i^+:=N\subset T_i^+$ be the oriented meridian. 
    Given $\theta_1^0,\theta_2^0\in [0,2\pi]$, it follows from \eqref{Eq: suspended weighted Hopf action} that the $S^1_{p,q}$-orbit (i.e. the fiber) through $(\theta_1^0,\theta_2^0)\in T_i^+$ is the circle $H_{(\theta_1^0,\theta_2^0)}(t):=(\theta_1^0+tp,\theta_2^0+tq)\subset T_i^+$, $t\in \mb R/(2\pi \mb Z)$ with its orientation induced from $t$. 
    Hence, we have \[[H_{(\theta_1^0,\theta_2^0)}]=p[L]+q[N]\in H_1(T_i^+;\mb Z).\] 
    By \cite{seifert1980textbook}*{p. 385, (2)(3)}, there is an oriented crossing circle $Q_i^+\subset T_i^+$ and integers $\alpha_i^+\geq 1$, $0\leq \beta_i^+<\alpha_i^+$, $\tilde s_i^+\in \{\pm 1\}$, so that for any $(\theta_1^0,\theta_2^0)\in (\mb R/2\pi\mb Z)^2$, 
    \begin{itemize}
        \item[(1)] $Q_i^+$ intersects the fiber $H_{(\theta_1^0,\theta_2^0)}$ at a single point with the intersection number $+1$;
        \item[(2)] $[M_i^+]=\tilde s_i^+ \left (\alpha_i^+ [Q_i^+]+\beta_i^+[H_{(\theta_1^0,\theta_2^0)}] \right )$, 
    \end{itemize}
    which determines the pair of invariants $(\alpha_i^+,\beta_i^+)$ at $u_i^+$. In particular, we see that $\tilde s_i^+\alpha_i^+$ is the intersection number of the oriented circles $M_i^+$ and $H$, and thus $\tilde s_i^+ \alpha_i^+=-s_i^+p$, which implies 
    \[\alpha_i^+=|p|,\qquad {\rm and}\qquad \tilde s_i^+=- s_i^+\cdot {\rm sgn}(p).\] 
    It then follows from (2) that $[N]=-s_i^+p[Q_i^+]- s_i^+\cdot {\rm sgn}(p)\beta_i^+(p[L]+q[N])$, and thus  
    \begin{align}\label{Eq: crossing curve Qi+}
        [Q_i^+]= -{\rm sgn}(p)\beta_i^+ [L]- \frac{s_i^+\cdot {\rm sgn}(p)+\beta_i^+ q}{|p|} [N] .
    \end{align}
    In particular, $(\alpha_i^+,\beta_i^+)$ with $0\leq \beta_i^+<\alpha_i^+=|p|$ satisfies that $\beta_i^+ q\equiv -s_i^+\cdot {\rm sgn}(p)=(-1)^i{\rm sgn}(p)$ mod $|p|$. 
    Therefore, $\beta_i^+=\mk a$ for $i$ odd, and $\beta_i^+=|p|-\mk a$ for $i$ even, where $\mk a,\mk b$ are given as in (iv). 
    One verifies that the definition of $(\alpha_i^+,\beta_i^+)$ is independent of the orientation of $M_i^+$. 
    Together with $s(p)=1$, we have $(|p|,\mk a)^{j_++\delta}$ and $(|p|,|p|-\mk a)^{j_+}$ in \eqref{Eq: seifert invariants}.

    Similarly, if $|q|>1$, then near the exceptional orbit $\pi_{p,q}^{-1}(u_i^-)=\{(0^{(2)},\sin(\varrho_i^-)e^{i\theta_2}, \cos(\varrho_i^-)):\theta_2\in \mb R/2\pi\mb Z\} \subset \Gamma$, the vector field $-\bd_\eta$ points outward of a tubular neighborhood $B_i^-\subset\Gamma$ of $\pi_{p,q}^{-1}(u_i^-)$, and the ordered set $((-1)^i\bd_{\theta_1},\bd_{\theta_2})$ gives the induced orientation of $T_i^-=\bd B_i^-$. 
    In this case, we shall take $s_i^-:=(-1)^i$ so that $s_i^-\cdot L$ and $N$ have the intersection number $+1$. 
    Then we have the oriented meridian $M_i^-:=L$, the oriented fiber $H$ with $[H]=p[L]+q[N]$, and the oriented crossing circle $Q_i^-$ so that (1)(2) remain valid for some $\tilde s_i^-\in \{\pm 1\}$, $\alpha_i^-\geq 1$, and $0\leq \beta_i^-<\alpha_i^-$. 
    By a similar computation, we see 
    \[\alpha_i^-=|q|, \qquad \tilde s_i^-=s_i^-\cdot {\rm sgn}(q),\qquad{\rm and}\qquad \beta_i^- p\equiv s_i^-\cdot {\rm sgn}(q)=(-1)^i{\rm sgn}(q)~{\rm mod~}|q|,\]
    so that 
    \begin{align}\label{Eq: crossing curve Qi-}
        [Q_i^-]= \frac{s_i^-\cdot {\rm sgn}(q)-\beta_i^-p}{|q|} [L] -{\rm sgn}(q)\beta_i^-[N], 
    \end{align}
    which implies that $\beta_i^-=\mk b$ if $i$ is odd, and $\beta_i^-=|q|-\mk b$ if $i$ is even. 
    Together with $s(q)=1$, we have $(|q|,\mk b)^{j_-+\delta}$ and $(|q|,|q|-\mk b)^{j_-}$ in \eqref{Eq: seifert invariants}. 

    \medskip
    {\bf Step 3}. {\it The invariant $\mk b_0$.}
    
    In this step, we only consider the case $|p|,|q|>1$, while the cases $|p|=1$ or $|q|=1$ are very similar. 
    We now follow \cite{seifert1980textbook}*{p.386} to compute $\mk b_0$. 
    
    To be exact, let $\Gamma^1:=\Gamma\setminus(\cup_{\kappa\in\{\pm\}}\cup_{i=1}^{\mk m_\kappa} B_i^\kappa)$ be the equivariant hypersurface with all tubes $B_i^\pm$ drilled out. 
    Then, after filling the holes by gluing ordinary tubes $V_i^\pm\cong \mb D^2\times \mb S^1$ with $\bd\mb D^2\times\{1\}\cong Q_i^\pm\subset T_i^\pm =\bd V_i^\pm$ as meridians, one obtains $\overline \Gamma^1$ as an $S^1$-bundle without exceptional orbits. 
    In particular, every $Q_i^\pm$ bounds a disk $D_i^\pm\subset V_i^\pm\subset \overline{\Gamma}^1$ that intersects every $S^1$-orbit in $V_i^\pm$ at a single point, i.e. $V_i^\pm\cong D_i^\pm\times S^1$. 
    
    Fix any small neighborhood $B_0$ of an $S^1_{p,q}$-orbit $S^1_{p,q}\cdot u_0\subset \interior(\Gamma^1)\subset \overline{\Gamma}^1$, and let $T_0:=\bd B_0$. 
    Consider the $S^1_{p,q}$-hypersurfaces $\Gamma^0:= \Gamma^1\setminus B_0$ and $\overline{ \Gamma}^0:=\overline{\Gamma}^1\setminus B_0$ with $B_0$ drilled out. 
    We claim that there is a section $f_0\subset \Gamma^0$ (i.e. a surface intersecting every orbit at a single point with $ \Gamma^0\cong f_0\times S^1$) so that $\bar f_0:=f_0\cup (\cup_{1\leq i\leq \mk m_{\kappa},\kappa\in\{\pm\}} D_i^\kappa)$ is a section of $\overline{ \Gamma}^0$ (i.e. $\overline{ \Gamma}^0\cong \bar f_0\times S^1$). 

    Indeed, since $\pi_{p,q}(\Gamma^0)\subset \mathring{\mb S}^4_{p,q}/S^1_{p,q}$ is a surface with non-empty boundary, $\pi_{p,q}(\Gamma^0)$ retracts onto a $1$-dimensional CW-complex, and thus  $H^2(\pi_{p,q}(\Gamma^0);\mb Z)=0$, which implies the $S^1$-bundle $\Gamma^0\to \pi_{p,q}(\Gamma^0)$ is trivial, i.e. $ \Gamma^0\cong \pi_{p,q}(\Gamma^0)\times S^1$.
    Using this global trivialization, every crossing circle $Q_i^\pm\subset \bd  \Gamma^0$ is identified with $\{(u,\phi_i^\pm(u))\in \pi_{p,q}(\Gamma^0)\times S^1: u\in \pi_{p,q}(T_i^\pm)\} $ for some continuous map $\phi_i^\pm: \pi_{p,q}(T_i^\pm)\to S^1$. 
    To construct the desired $f_0$, it is sufficient to find a continuous map 
    \[\Phi: \pi_{p,q}(\Gamma^0)\to S^1 \]
    so that $\Phi\llcorner \pi_{p,q}(T_i^\pm) = \phi_i^\pm$.  
    Let $d_i^\pm:=\deg(\phi_i^\pm)$ be the degree of $\phi_i^\pm$. 
    Then we take any continuous map $\phi_0: \pi_{p,q}(T_0)\to S^1$ with degree $d_0:=\deg(\phi_0)$ so that $d_0+\sum_{1\leq i\leq \mk m_{\kappa}, \kappa\in\{\pm\}} d_i^\kappa = 0$. 
    By \cite{mcduff2017topology}*{Lemma 2.7.3}, the boundary maps $\phi_0,\phi_i^\pm$ admit a continuous extension $\Phi$. 
    Therefore, the surface $\{(u,\Phi(u))\in \pi_{p,q}(\Gamma^0)\times S^1: u\in \pi_{p,q}( \Gamma^0)\}$ corresponds to the desired section $f_0\subset \Gamma^0$. 

    Let $O_f$ be an orientation of $f_0$ so that $(O_f,H)$ forms the orientation of $\Gamma$, where $[H]=p[L]+q[N]$ is the oriented fiber.
    Define the circle $Q_0:=\bd f_0 \llcorner T_0$ with boundary orientation induced from $f_0$. 
    By (1) in {\bf Step 2}, every $(\nu_i^\pm,Q_i^\pm,H)$ also gives the orientation of $\Gamma$, where the vector $\nu_i^\pm$ at $\spt(Q_i^\pm)$ points outward (resp. inward) of the solid torus $V_i^\pm$ (resp. $f_0$). 
    Thus, $O_f=(\nu_i^\pm ,Q_i^\pm)$, and the circle $\bd f_0\llcorner T_i^\pm$ with boundary orientation  is indeed $-Q_i^\pm$, which implies $[Q_0]-\sum_{1\leq i\leq \mk m_{\kappa}, \kappa\in\{\pm\}}[Q_i^\kappa]=[\bd f_0]=0\in H_1( \Gamma^0;\mb Z)$. 
    
    In addition, if the vector $\nu_0$ at $\spt(Q_0)$ points outward (resp. inward) of the tube $B_0$ (resp. $f_0$), then $(-\nu_0,Q_0)=O_f$, and $(\nu_0,Q_0,-H)$ gives the orientation of $\Gamma$, which implies that $(Q_0,-H)$ is the boundary orientation of $T_0=\bd B_0$. 
    
    Let $M_0\subset T_0$ be a meridian with any orientation. 
    Then, since $B_0$ is an ordinary tube, $M_0$ and $Q_0$ are both sections of $T_0$ (i.e. intersect every orbit of $T_0$ at a single point). 
    Hence, there are $\alpha_0\in \{\pm1\}$ and $\beta_0\in\mb Z$ so that 
    \[[M_0]=\alpha_0[Q_0]+\beta_0[-H]=1\cdot (\alpha_0[Q_0]) + \alpha_0\beta_0(\alpha_0[-H]) \in H_1(T_0;\mb Z).\]
    By the definition in \cite{seifert1980textbook}*{p. 386}, $\mk b_0:=\alpha_0\beta_0$ is independent of the orientation of $M_0$.
    Then, after embedding into $\Gamma^1$, we conclude that $0=[M_0]=\alpha_0\cdot \sum_{1\leq i\leq \mk m_{\kappa}, \kappa\in\{\pm\}} [Q_i^\kappa]- \beta_0 [H]\in H_1(\Gamma^1;\mb Z)$. 
    Combining \eqref{Eq: crossing curve Qi+}\eqref{Eq: crossing curve Qi-} with $[H]=p[L]+q[N]$, we obtain $\mk b_0=\alpha_0\beta_0=-j_+-j_--\delta \kappa(p,q)$, where $\kappa(p,q)=\frac{\mk a|q|+\mk b|p|+{\rm sgn}(pq)}{|pq|}\in \{-1,0,1\}$. 
    Note that for general $|p|,|q|\geq 1$, one can use $s(p)[Q_i^+]$ and $s(q)[Q_i^-]$ in place of $[Q_i^+]$ and $[Q_i^-]$ respectively to show that $\mk b_0=-s(p)j_+-s(q)j_--\delta\kappa(p,q)$. 

    Finally, \eqref{Eq: Euler number of Seifert manifold} and \eqref{Eq: orbifold Euler charateristic} follow directly from their definitions (cf. \cite{scott1983three-manifolds}*{p. 437, p. 427}). 
    %
    %
\end{proof}

\subsection{\texorpdfstring{Min-max constructions for minimal Seifert hypersurfaces in $\mb S^4$}{Min-max constructions for minimal Seifert hypersurfaces in S4}}

Using the constructions of equivariant sweepouts in Section \ref{Sec: new minimal hypersurfaces in S4}, we have the following existence results for minimal Seifert fibered manifolds in $\mb S^4$. 
Note that the Seifert fibration structure of a $G_n^{p,q}$-hypersurface $\Gamma$ is determined by $\genus(\Gamma/S^1_{p,q})$ and $\#((\Gamma/S^1_{p,q})\cap \Xi_0^\pm)$ (cf. Lemma \ref{Lem: S1pq invariant Seifert manifolds}). 
Hence, we mainly consider these topological data for $\Gamma$ in the following constructions. 

\begin{theorem}\label{Thm: minimal Seifert manifold in S4 - same side}
	Given $n\in \{2,3,\dots\}$ and coprime $p,q\in \mb Z\setminus\{0\}$, there exists an embedded minimal $G_n^{p,q}$-hypersurface $\widehat\Gamma_n^{p,q}\subset \mb S^4$ that is topologically a Seifert fibered $3$-manifold satisfying
    \[(\mk g, \mk m_+, \mk m_-)\in \{(0,2,0),(0,0,2), (n-1,2,2),(n-1,4,0),(n-1,0,4)\}\]
    where $\mk g:=\genus(\widehat\Gamma_n^{p,q}/S^1_{p,q})$ and $\mk m_\pm:= \#((\widehat\Gamma_n^{p,q}/S^1_{p,q})\cap \Xi_0^\pm)$ (cf. \eqref{Eq: half axis z on S3}). 
\end{theorem}


\begin{proof}
    {\bf Case I}. $n=2k\geq 2$ is even. 
    
	Since $\mb S^4/S^1$ and $\mb S^4/S^1_{p,q}$ have the same parameterizations \eqref{Eq: parameterize S4/S1}\eqref{Eq: parameterize S4/S1pq}, we have a $D_n$-equivariant sweepout $\{\Sigma_t\}_{t\in [0,1]}$ of $\mb S^4/S^1_{p,q}$ as constructed in the proof of Lemma \ref{Lem: Gn-sweepout n-even}. 
	Then, $\{\wti\Sigma_t^{p,q}:=\pi_{p,q}^{-1}(\Sigma_t)\}_{t\in [0,1]}$ is a $G_n^{p,q}$-equivariant sweepout of $\mb S^4$. 
	Combining the equivariant min-max (\cite{wang2026spherical}*{Theorem 1.5}) with the proof of Proposition \ref{Prop: new minimal hypersurface of even n} and Remark \ref{Rem: valid in symmetric S4 with Ric>0 for n-even}, we obtain a connected embedded minimal $G_n^{p,q}$-hypersurface $\widehat \Gamma_n^{p,q}$ with multiplicity one so that $\mk g= \genus(\widehat\Gamma_n^{p,q}/S^1_{p,q}) \leq n-1$. 
    By Lemma \ref{Lem: S1pq invariant Seifert manifolds}(ii)(iii) and the proof of \eqref{Eq: intersection with C0 for n even}\eqref{Eq: R-H formula of intersection 4j}, we conclude that 
	$\mk m_\pm = 2j_\pm$ 
	for some $j_\pm\in \{0,1,\dots\}$, and 
    $\mk g=(j_++j_--1)(n-1)\leq n-1$. 
    This proves the desired result. 

    {\bf Case II}. $n=2k+1\geq 3$ is odd. 
    
    For any $\rho\in (0,1)$ and $\epsilon>0$ small enough, we construct a surface $\widehat\Gamma_{\rho,\epsilon}\subset \mb S^4/S^1_{p,q}$ as 
    \[
    \widehat\Gamma_{\rho,\epsilon}:= [(S_\rho \cup S_{-\rho})\setminus \cup_{l=1}^n B_\epsilon(\Xi_l^{{\rm sgn}((-1)^l)})] \cup \{(w,t,x)\in \cup_{l=1}^n\bd B_\epsilon(\Xi_l^{{\rm sgn}((-1)^l)}): -\rho<x<\rho\}, 
    \]
    which is formed by two spheres $S_\rho\cup S_{-\rho}$ connected via $n$ necks centered at the longitudes $\cup_{l=1}^n \Xi_l^{{\rm sgn}((-1)^l)}$ (Figure \ref{Fig: sweepout n-even}). 
    Since $n$ is odd, one easily sees that $\cup_{l=1}^n \Xi_l^{{\rm sgn}((-1)^l)}$ is $D_n$-invariant, which implies that $\widehat\Gamma_{\rho,\epsilon}$ is also $D_n$-invariant. 
    Additionally, similar to Lemma \ref{Lem: Gn sweepout}, one can use $\widehat\Gamma_{\rho,\epsilon}$ in place of $\Gamma_{\rho,\epsilon}$ and apply an opening-up deformation to construct a $D_n$-equivariant sweepout $\{\Sigma_t\}_{t\in [0,1]}$ so that $\genus(\Sigma_t) = n-1$ for $t\in (0,1)$, 
    $\Sigma_0\subset\cup_{l=1}^n\Upsilon_l$, 
    and $\Sigma_1= \cup_{l=1}^n\Xi_l^{{\rm sgn}((-1)^l)}$. 
    By applying the equivariant min-max (\cite{wang2026spherical}*{Theorem 1.5}) to the $G_n^{p,q}$-sweepout $\{\wti\Sigma_t:=\pi_{p,q}^{-1}(\Sigma_t)\}_{t\in [0,1]}$, we obtain the desired minimal hypersurface as in {\bf Case I}. 
\end{proof}

\begin{theorem}\label{Thm: minimal Seifert manifold in S4 - different sides}
	Given $n\in \{2,3,\dots\}$ and coprime $p,q\in \mb Z\setminus\{0\}$, there exists an embedded minimal $G_n^{p,q}$-hypersurface $\widetilde\Gamma_n^{p,q}\subset \mb S^4$ that is topologically a Seifert fibered $3$-manifold satisfying: 
    \[({\mk  g}, \mk m_+, \mk m_-)\in \{(0,1,1),(n-1,3,1), (n-1,1,3),(2n-2,3,3),(2n-2,5,1),(2n-2,1,5)\},\]
    where $ {\mk g}:=\genus(\wti\Gamma_n^{p,q}/S^1_{p,q})$ and $\mk m_\pm:= \#((\wti\Gamma_n^{p,q}/S^1_{p,q})\cap \Xi_0^\pm)$ (cf. \eqref{Eq: half axis z on S3}). 
    Moreover, there exists $\wti N_{p,q}\geq 1$ (depending on $p,q$) large enough so that $\mk  g\neq 0$ for all $n\geq \wti N_{p,q}$, which implies that 
    \begin{itemize}
        \item[(i)] $\wti{\mc M}_{p,q} := \{\wti\Gamma_n^{p,q}\subset \mb S^4: n\geq \wti N_{p,q}\}$ contains infinitely many embedded minimal hypersurfaces in $\mb S^4$ of distinct topological types;
        \item[(ii)] $\wti{\mc M}_{p_1,q_1}\cap \wti{\mc M}_{p_2,q_2}=\emptyset$ for any two distinct pairs of coprime integers $\{p_i>q_i\geq 1\}_{i=1}^2$. 
    \end{itemize}
\end{theorem}


\begin{proof}
    {\bf Case I}. $n=2k+1\geq 3$ is odd. 
    
    In this case, we can directly use the $D_n$-sweepout $\{\Sigma_t\}_{t\in [0,1]}$ of $\mb S^4/S^1_{p,q}$ as constructed in Lemma \ref{Lem: Gn sweepout}, and the desired minimal $G_n^{p,q}$-hypersurface $\wti\Gamma_n^{p,q}$ can be obtained in essentially the same way as in Proposition \ref{Prop: new minimal hypersurface} by using $S^1\rtimes_\psi \langle\hat{s}\rangle \subset G_n^{p,q}$ in place of ${\rm Pin}_-(2)$. 
    Hence, we only point out the modifications that are required to show the regularity and the topological control of the equivariant min-max $G_n^{p,q}$-varifold $V_n^{p,q}$ as constructed in Proposition \ref{Prop: new minimal hypersurface}. 
    
    Note that $\mb S^4/S^1_{p,q}$ may not be smooth at $\Xi_0$ (cf. \eqref{Eq: axes on S3}), and thus the arguments in Proposition \ref{Prop: new minimal hypersurface} only ensure the regularity of $V_n^{p,q}$ in $\mb S^4\setminus (E_0\cap \pi_{p,q}^{-1}(\Xi_0))$ in general. 
    Nevertheless, since $E_0/S^1_{p,q}$ is transversal to $\Xi_0$ at two (possible singular) points $\{u_i\}_{i=1}^2 = (E_0/S^1_{p,q})\cap \Xi_0$, for any small $G_n^{p,q}$-annulus $An$ centered at $\pi_{p,q}^{-1}(u_i)$, we can set $\epsilon>0$ small enough so that $[(An\cap B_{2\epsilon}(E_0) )/S^1_{p,q}] \cap \Xi_0=\emptyset$. 
    Hence, one can apply the $G_n^{p,q}$-isotopy area minimizing constructions to obtain smooth minimal $G_n^{p,q}$-equivariant replacements of $V_n^{p,q}$ in $An$ with the required compactness property, where the regularity of the $G_n^{p,q}$-replacements in $An\setminus B_\epsilon(E_0)$ and $(An\cap B_{2\epsilon}(E_0) )/S^1_{p,q}$ follows from \cite{wang2026spherical}*{\S 6.3.2} and \cite{ketover2016free}*{\S 7.2} respectively. 
    Combining this with the proof of \cite{wang2026spherical}*{Proposition 6.16}, one further shows that the tangent cone of $V_{n}^{p,q}$ at any $\tilde u_i\in \pi_{p,q}^{-1}(u_i)$ is an integer multiple of $T_{\tilde u_i}(G_n^{p,q}\cdot \tilde u_i) \times C$, where $T_{\tilde u_i}(G_n^{p,q}\cdot \tilde u_i)$ is the tangent space of the orbit $G_n^{p,q}\cdot \tilde  u_i$ at $\tilde u_i$, and $C$ is a regular minimal cone in the normal vector space $N_{\tilde u_i}(G_n^{p,q}\cdot \tilde u_i):=\{v\in T_{\tilde u_i}\mb S^4: v\perp T_{\tilde u_i}(G_n^{p,q}\cdot \tilde u_i)\}$. 
    Since $\codim(G_n^{p,q}\cdot \tilde u_i)=3$, the link of $C$ is a closed geodesic of $\mb S^2_1\subset N_{\tilde u_i}(G_n^{p,q}\cdot \tilde u_i)\cong \mb  R^3$, which implies that $C$ is a plane. 
    Then, the regularity of $V_n^{p,q}$ follows from \cite{colding2002minmax}*{p.29, Step4}.

    Next, combining Lemma \ref{Lem: S1pq invariant Seifert manifolds}(ii)(iii) with the proof of \eqref{Eq: intersection with C0} and \eqref{Eq: R-H formula of intersection 2+4j}, 
    we see that $\mk m_\pm = 1+2j_{\pm}$ and $ {\mk g}=(j_++j_-)(n-1)$, which shows the desired topological data.

     {\bf Case II}. $n=2k\geq 2$ is even. 

     For any $\rho\in (0,1)$ and small $\epsilon>0$, we can construct a $D_n$-surface $\widehat\Gamma_{\rho,\epsilon}\subset \mb S^4/S^1_{p,q}$ by connecting three spheres $S_\rho\cup  S_0\cup S_{-\rho}$ 
     via $2n$ necks (Figure \ref{Fig: sweepout n-odd}) centered at each $\{\widehat\Xi_l^{\pm}\}_{l=1}^n$ respectively:
     \begin{align*}
         \widehat C_{\rho,\epsilon}^+ &:= \{(w,t,x)\in \cup_{b=1}^k \bd B_\epsilon(\widehat\Xi_{2b}) = \cup_{b=1}^k\cup_{\kappa\in\{\pm\}} \bd B_\epsilon(\widehat\Xi_{2b}^\kappa) : 0<x<\rho \}, \\
        \widehat C_{\rho,\epsilon}^-&:= \{(w,t,x)\in \cup_{b=1}^k \bd B_\epsilon(\widehat\Xi_{2b-1})= \cup_{b=1}^k\cup_{\kappa\in\{\pm\}} \bd B_\epsilon(\widehat\Xi_{2b-1}^\kappa): -\rho <x<0 \}. 
     \end{align*} 
     Using $\widehat\Gamma_{\rho,\epsilon}$ in place of $\Gamma_{\rho,\epsilon}$, the constructions in Lemma \ref{Lem: Gn sweepout} can be easily adapted to construct a $D_n$-equivariant sweepout $\{\Sigma_t\}_{t\in [0,1]}$ in $\mb S^4/S^1_{p,q}$ with similar properties. 
     Combined with the proof of Proposition \ref{Prop: new minimal hypersurface} and the modifications in {\bf Case I}, we obtain the desired minimal $G_n^{p,q}$-hypersurface. 

    \medskip
     Finally, the proof of Theorem \ref{Thm: new minimal hypersurfaces of odd n} can be easily adapted to show the existence of the desired $\wti N_{p,q}\in \mb N$. 
     Moreover, one readily verifies \eqref{Eq: unique seifert fibration} for $n\geq \wti N_{p,q}$ by \eqref{Eq: orbifold Euler charateristic}. 
     Therefore, (i) and (ii) follow from Lemma \ref{Lem: S1pq invariant Seifert manifolds} and Remark \ref{Rem: Seifert invariants}(iii). 
\end{proof}

\begin{theorem}\label{Thm: minimal Seifert manifold in S4 - Lawson type}
    Given $n\in \{2,3,\dots\}$ and coprime $p,q\in \mb Z\setminus\{0\}$, there exists an embedded minimal $G_n^{p,q}$-hypersurface $\Gamma_n^{p,q}\subset \mb S^4$ that is topologically a Seifert fibered $3$-manifold satisfying
    \[ (\mk g, \mk m_+,\mk m_-)\in \{(0,1,1), (n-1,3,1), (n-1,1,3)\}, \]
    where $\mk g:=\genus(\Gamma_n^{p,q}/S^1_{p,q})$ and $\mk m_\pm:= \#((\Gamma_n^{p,q}/S^1_{p,q})\cap \Xi_0^\pm)$ (cf. \eqref{Eq: half axis z on S3}). 
    Moreover, there exists $N_{p,q}\geq 1$ (depending on $p,q$) large enough so that $\mk g=n-1$ for all $n\geq N_{p,q}$, which implies
    \begin{itemize}
        \item[(i)] ${\mc M}_{p,q} := \{\Gamma_n^{p,q}\subset \mb S^4: n\geq  N_{p,q}\}$ contains infinitely many embedded minimal hypersurfaces in $\mb S^4$ of distinct topological types;
        \item[(ii)] ${\mc M}_{p_1,q_1}\cap {\mc M}_{p_2,q_2}=\emptyset$ for any two distinct pairs of coprime integers $\{p_i>q_i\geq 1\}_{i=1}^2$. 
    \end{itemize}
\end{theorem}
	
\begin{proof}
    Firstly, we reparameterize $\mb S^4/S^1_{p,q}$ as the unit sphere $\mb S^3$ by multiplying the coordinates $(w,t)$ by $2$ in \eqref{Eq: parameterize S4/S1pq}. 
    Then, under the stereographic projection $\wti P$ at $(0^{(2)},1,0)$ given by 
    \[\wti P: (w,t,x)\in \mb S^3\setminus\{(0^{(2)},1,0)\} \mapsto \left(\frac{{\rm Re}(w)}{1-t}, \frac{{\rm Im}(w)}{1-t}, \frac{x}{1-t} \right)\in \mb R^3,\]
    the modified reflectional $D_n$-action \eqref{Eq: Dn reflectional action on S4/S1pq} on $\mb S^4/S^1_{p,q}$ induces the standard rotational $D_n$-action on $\mb R^3$ (cf. \cite{ketover2016free}\cite{carlotto2022free}) generated by $\tilde r:=\wti P\circ r\circ \wti P^{-1}$ and $\tilde s:=\wti P\circ s\circ \wti P^{-1}$ (as explained after \eqref{Eq: s action using stereographic projection}). 
    Note that, using the projection $\wti P$, the singular north and south poles are mapped to $(0,0,\pm 1 )\in \mb R^3$ respectively, the (possible singular) arcs $\Xi_0^\pm$ are mapped to the two portions of the vertical $z$-axis lying outside/inside $\mb B^3_1$ respectively, and the circles $ \{\Upsilon_l\}_{l=1}^n$ (Lemma \ref{Lem: fixed points set of s}) are mapped to the axes of the $\pi$-rotations $\{\tilde s \tilde r^{n-l}\}_{l=1}^n$ in the horizontal plane. 
    Hence, following the constructions in \cite{buzano2025self-shrinker}*{Lemma 2.2}, one can desingularize the union of the plane $ \wti P(\{(w,t,0)\in \mb S^3\})$ and the round sphere $\mb S^2_{\tan(\pi \tau/2)}$ ($\tau\in (0,1)$) in $\mb R^3$ by a sufficiently small $D_n$-equivariant perturbation, which resembles the Lawson $\xi_{1,n-1}$ desingularization picture. 
    This produces a $D_n$-equivariant sweepout $\{\Sigma_\tau'\}_{\tau\in [0,1]}$ in $\mb R^3$ so that 
    \begin{itemize}
        \item $\Sigma_\tau:=\closure (\wti P^{-1}(\Sigma_\tau'))$ has genus $n-1$ for every $\tau\in (0,1)$, and $\Sigma_0=\Sigma_1=E_0/S^1_{p,q}$;
        \item the singular north and south poles $\{(0^{(2)}, 0,\pm 1)\}\subset \mb S^4/S^1_{p,q}$ are contained only in $\Sigma_{1/2}$, and each $\wti \Sigma_\tau:=\pi_{p,q}^{-1}(\Sigma_\tau)$ is a smoothly embedded $G_n^{p,q}$-hypersurface for $\tau\neq 1/2$;
        \item $\Sigma_\tau$ is a $D_n$-equivariant desingularization of $E_0/S^1_{p,q} \cup \{(w,t,x)\in \mb S^3:t=-\cos(\pi \tau)\}$ containing $\Upsilon_1,\dots,\Upsilon_n$ defined in \eqref{Eq: circle on E0/S1} for every $\tau\in [0,1]$;
        \item for any $\tau\in [0,1]$ and $N_\tau := \pi_{p,q}^{-1}(\{(w,t,x)\in \mb S^3:t=-\cos(\pi \tau)\}) \subset \mb S^4$, 
        \[\mc H^3_{g_{\mb S^4}}(\wti\Sigma_\tau)\leq \mc H^3_{g_{\mb S^4}}(E_0) + \mc H^3_{g_{\mb S^4}}(N_\tau) + \frac{1}{100} 
        \leq 
        2\pi^2+\pi^3+ \frac{1}{100} < 6\pi^2 = 3 \mc H^3_{g_{\mb S^4}}(E_0),\] 
        where $2\pi^2$ is the area of $E_0=\pi_{p,q}^{-1}(\{(w,t,0)\})$, $\pi^3$ is the area of the suspension $N_{1/2}=\{(z_1,z_2,x)\in \mb S^4: |z_1|=|z_2|\}$ of the Clifford torus, and $N_\tau = \{(z_1,z_2,x)\in\mb S^4 : |z_1|^2-|z_2|^2=-\cos(\pi\tau)\sqrt{1-x^2}\}$ has an area of at most $\pi^3$; 
        \item $\{\wti \Sigma_\tau\}_{\tau\in [0,1]}$ induces a $1$-sweepout of $\mb S^4$ in the sense of Almgren-Pitts (by \cite{buzano2025self-shrinker}*{Lemma 2.4}).
    \end{itemize}
    By the proof of \cite{buzano2025self-shrinker}*{Lemma 2.4(iv)}, $\hat s\in G_n^{p,q}$ exchanges the two $S^1_{p,q}$-invariant components of $\mb S^4\setminus\wti\Sigma_\tau$ for any $\tau\in [0,1]$. 
    Hence, combining the equivariant min-max theory \cite{wang2026spherical}*{Theorem 1.5} with the modification for regularity (cf. Theorem \ref{Thm: minimal Seifert manifold in S4 - different sides}), we can apply the proof of Proposition \ref{Prop: new minimal hypersurface} with the ${\rm Pin}_-(2)$-action replaced by the action of $S^1\rtimes_\psi \langle\hat{s}\rangle \subset G_n^{p,q}$. 
    This shows the existence of  
    a minimal $G_n^{p,q}$-hypersurface $\Gamma_n^{p,q}\subset \mb S^4$ with multiplicity one so that $\mc H^3_{g_{\mb S^4}}(\Gamma_n^{p,q}) > \mc H^3_{g_{\mb S^4}}(\mb S^3)$ and $\genus(\Gamma_n^{p,q}/S^1_{p,q})\leq n-1$. 
    Additionally, one easily modifies the proof of \eqref{Eq: intersection with C0} and \eqref{Eq: R-H formula of intersection 2+4j} to show that $\mk m_\pm=1+2j_\pm$ and $\mk g=(j_++j_-)(n-1)$, which verifies the topological data. 
    Finally, the compactness analysis in Theorem \ref{Thm: new minimal hypersurfaces of odd n} {\bf Part IV} can be easily adapted to show the last statement, and thus (i) and (ii) follow from Lemma \ref{Lem: S1pq invariant Seifert manifolds} and Remark \ref{Rem: Seifert invariants}. 
\end{proof}

\bibliographystyle{abbrv}

\bibliography{reference.bib}   
\end{document}